\documentclass[11pt]{amsart}
\usepackage[margin=1.2in]{geometry}
\usepackage{graphicx}
\usepackage{amssymb}
\usepackage{epstopdf}
\usepackage{amsmath}
\usepackage{color}
\usepackage{hyperref}
\usepackage{pdfsync}
 
\usepackage[all]{xy}
\xyoption{matrix}
\xyoption{arrow}
\usepackage{tikz}

\newtheorem{thm}{Theorem}[subsection]
\newtheorem{lem}[thm]{Lemma}
\newtheorem{cor}[thm]{Corollary}
\newtheorem{prop}[thm]{Proposition}

\theoremstyle{definition}
\newtheorem{defn}[thm]{Definition}
\newtheorem{eg}[thm]{Example}

\theoremstyle{remark}
\newtheorem{rem}[thm]{Remark}
 
\numberwithin{equation}{section}

\newcommand{\mat}[1]{\ensuremath{
\left[\begin{matrix}#1
\end{matrix}\right]
}}
 
\def\xrarrow{\xrightarrow} %right arrow {label on top}

\def\then{\Rightarrow}

\newcommand{\cofib}{\rightarrowtail}

\def\-{\text{-}}
\newcommand{\ifff}{\Leftrightarrow}
 
\newcommand{\into}{\hookrightarrow}
 \newcommand{\onto}{\twoheadrightarrow}
 \newcommand{\cof}{\rightarrowtail}

\def\smallcoprod{\,{\textstyle{\coprod}}\,}

\def\<{\left<}
\def\>{\right>}

\DeclareMathOperator{\Ext}{Ext}%
\DeclareMathOperator{\End}{End}%
\newcommand{\field}[1]{\mathbb{#1}}
\newcommand{\ZZ}{\ensuremath{{\field{Z}}}}
\newcommand{\CC}{\ensuremath{{\field{C}}}}
\newcommand{\RR}{\ensuremath{{\field{R}}}}

\newcommand{\NN}{\ensuremath{{\field{N}}}}

\newcommand{\commentout}[1]{}

\def\el{\ell}
\def\ll{\lambda}

\newcommand{\cA}{\ensuremath{{\mathcal{A}}}}

\newcommand{\cB}{\ensuremath{{\mathcal{B}}}}
\newcommand{\cC}{\ensuremath{{\mathcal{C}}}}

\newcommand{\cF}{\ensuremath{{\mathcal{F}}}}

\newcommand{\cH}{\ensuremath{{\mathcal{H}}}}

\newcommand{\cL}{\ensuremath{{\mathcal{L}}}}

\newcommand{\cP}{\ensuremath{{\mathcal{P}}}}

\newcommand{\cR}{\ensuremath{{\mathcal{R}}}}

\newcommand{\cT}{\ensuremath{{\mathcal{T}}}}

\newcommand{\cX}{\ensuremath{{\mathcal{X}}}}

\def\a{\alpha}
\def\b{\beta}
\def\g{\gamma}

\def\d{\partial}

\def\e{\epsilon}
\def\f{\varphi}

\def\r{\rho}
\def\s{\sigma}

\def\t{\tau}
\def\th{\theta}

\def\z{\zeta}

\def\ov{\overline}

\def\ul{\underline}

\newcommand\mm{\mathfrak{m}}

\title{Continuous cluster categories of type $D$}
\author{Kiyoshi Igusa}
\address{Department of Mathematics, Brandeis University, Waltham, MA 02454}\email{igusa@brandeis.edu}
 \thanks{The first author is supported by NSA Grant \#H98230-13-1-0247}

\author{Gordana Todorov}
\address{Department of Mathematics, Northeastern University, Boston, MA 02115}
\email{g.todorov@neu.edu}
\begin{document}

\begin{abstract} We construct continuous Frobenius categories of type $D$. The stable categories of these Frobenius categories are cluster categories which contain the standard cluster categories of type $D_n$. When $n=\infty$, maximal compatible sets of indecomposable objects are laminations of the punctured disk. Discrete laminations are clusters. This new construction is topological and it also gives an algebraic interpretation of the ``tagged arcs'' which occur in Schiffler's geometric description \cite{S} of clusters of type $D_n$. 
\end{abstract}

\maketitle

\tableofcontents

%\newpage

%{\color{blue}
\section*{Introduction}

This paper is a sequel to the poset paper. We consider cyclic posets $\cX$ with the action of a finite group $G$ so that the effective stabilizer $H_X=\{h\in G\,|\,hX\cong X\}$ of any $X\in \cX$ is abelian. Take $R$ a discrete valuation ring so that the characteristic of the residue field $K=R/\mm$ does not divide the order $n$ of the group $G$. Assume also that $R$ contains all $n$th roots of unity. Assuming that the linearization $\cF=\cF(\cX,\cX_0,R)$ of the cyclic poset is a Frobenius category $\cF$ with $\cX_0\subseteq\cX$ begin the set of indecomposable projective-injective objects. Then $G$ acts on $\cF$ and the orbit category $\cF^G$ is Frobenius and Krull-Schmidt and we give a complete description of all indecomposable objects. In an appendix we extend this to the case when $H_X$ is nonabelian.
%(Actually we can weaken the hypotheses: For any finite group $G$, $\cF^G$ is Frobenius. The Krull-Schmidt theorem requires that the effective stabilizer $H_X=\{h\in G\,|\,hX\cong X\}$ of any $X\in \cX$ is abelian of order $n$ not divisible by $char\,K$ and $R$ needs to have all $n$th roots of unity only for $n=|H|$.)

To study continuous cluster categories of type $D$, we specialize this general setup to the case when $G$ is the cyclic group of order 2 with nontrivial element $\psi$ acting on the continuous Frobenius category $\cF=\cF(S^1)$ the unit circle by rotation by $\pi$. We assume that $char\,R/\mm\neq2$. Then $R$ contains $\pm1$, the required roots of unity of unity. In this case, the objects of the orbit Frobenius category, which we denote $\cF^\psi$, are pairs $(X,\xi)$ where $X$ is an object of $\cF_\phi$ and $\xi$ is an isomorphism $\xi:\psi X\cong  X$. We prove the crucial Krull-Schmidt theorem and show that the indecomposable objects are
\begin{enumerate}
\item \emph{regular objects} which have the form $(X\oplus\psi X,\xi)$ where $\xi$ simply reverses the two summands
\[
	\xi=\mat{0 & 1\\1 & 0}: \psi X\oplus X\cong X\oplus \psi X
\]
\item \emph{singular objects} $(X,\xi)$ where $X$ is indecomposable.
\end{enumerate}

We show that there are two kinds of singular objects corresponding to tagged half-edges which are ``plain'' and ``notched'' in Ralf Schiffler's geometric description of cluster categories of type $D$ \cite{S} which, in turn, is a categorification of a special case of \cite{FST}.

We also show that, for any odd prime $p$, the orbit category of the continuous cluster category under the action of the cyclic group $\ZZ/p$ acting by rotation by $2\pi/p$ has a cluster structure and, moreover, the group of orientation preserving homeomorphisms of the circle acts transitively on the set of clusters. They are also geometrically equivalent to the cluster structure we obtain for $p=2$ in the case when $char\,K=2$.

The paper starts with \emph{approximation categories}. These are pairs $(\cB,\cB_0)$ like we had before, but with the observation that:

\emph{There is at most one exact structure on $\cB$ which will make it a Frobenius category with $\cB_0$ being the full subcategory of projective-injective objects.}

This implies that the Frobenius category $\cF(\cX,\cX_0;R)$ is completely determined by $R$ and the pair $(\cX,\cX_0)$ which we call a \emph{Frobenius cyclic poset}. 

We continue with an analysis of all maximal compatible subsets of the orbit cluster category $\cC_\pi^{\ZZ/2}$. We call these \emph{laminations}. We show that laminations are closed and, therefore, limit points cannot be removed and therefore cannot be mutated. The discrete laminations on the other hand form a cluster structure since they have no limit points.
%}

 %\newpage
 
%-----------------------------------------------------------------------------------------
% INTRODUCTION
%-----------------------------------------------------------------------------------------

\section{Construction of the Frobenius category}

\subsection{Approximation categories}

We define Frobenius category structures using approximation sequences, similar to \cite{BM94}.

\begin{defn}
Suppose that $\cB$ is an additive Krull-Schmidt category which is idempotent complete. Let $\cB_0$ be a full subcategory of $\cB$ which is closed under isomorphism, direct sum and direct summands. Then, by a \emph{two-way $\cB_0$-approximation sequence} we mean a sequence
\[
	X\xrarrow i J\xrarrow p Y
\]
satisfying the following.
\begin{enumerate}
\item $X$ is the kernel of $p$ and $Y$ is the cokernel of $i$. In particular, $p\circ i=0$.
\item $J$ is a right $\cB_0$-approximation of $Y$ and a left $\cB_0$-approximation of $X$.
\end{enumerate}
We say that the sequence is \emph{minimal} if neither $X$ nor $Y$ contains a summand in $\cB_0$.

We say that $\cB$ has \emph{enough} two-way $\cB_0$-approximation sequences if every indecomposable object of $\cB$ not in $\cB_0$ is equal the kernel in one sequence and the cokernel in another sequence.
\end{defn}

To illustrate the assumptions in the definition: suppose that $Y=P\oplus Z$ where $P\in\cB_0$. Since $J$ is a right $\cB_0$-approximation of $Y$, the inclusion morphism $P\to Y$ lifts to a map $f:P\to J$. Then $(f,0)\circ p$ is an idempotent in $End(J)$. Since $\cB$ is idempotent complete, we get a decomposition $J\cong P\oplus Q$ where $Q\in\cB_0$ and $p=id_P\oplus q:P\oplus Q\to P\oplus Z$ for some $q:Q\to Z$. It follows that
\[
	X\to Q\xrarrow q Z
\]
is a two-way $\cB_0$ approximation sequence. By a similar argument we can factor out any summands of $X$ belonging to $\cB_0$ and we obtain a {minimal sequence} $X\to Q\to Z$ where $X,Z$ have no summands in $\cB_0$. This shows that, if $\cB$ has enough two-way $\cB_0$-approximation sequences then every object of $\cB$ having no components in $\cB_0$ is the kernel in one minimal sequence and the cokernel in another minimal sequence.

The key point of this definition is that it uniquely determines the Frobenius structure on $\cB$ if it exists:

\begin{prop}\label{exact structure of an approximation category}
Suppose that $\cB$ is a Frobenius category and $\cB_0$ is the full subcategory of projective injective objects. Then the exact structure of $\cB$ is given by the condition that
\[
	A\xrarrow fB\xrarrow gC
\]
is exact if and only if, for all object $P$ in $\cB_0$, the following are short exact sequences.
\begin{enumerate}
\item $0\to \cB(P,A)\to \cB(P,B)\to \cB(P,C)\to 0$
\item $0\to \cB(C,P)\to \cB(B,P)\to \cB(A,P)\to0$
\end{enumerate}
Furthermore, $\cB$ has enough two-way approximation sequences and any two-way approximation sequence is exact.
\end{prop}

We use the notation $A\cof B\onto C$ to denote an exact sequence in an exact category. %Since a 2-way approximation sequence is the pushout of itself along the identity map $A\to A$, it is exact.

\begin{proof}
It follows from the definition of a Frobenius category that it has enough two-way approximation sequences and that they are all exact. It also follows from the first part of the proposition since a two-way approximation sequence clearly satisfies the two conditions to make it exact. It is clear that every exact sequence satisfies the two listed conditions. So it remains to show that they are sufficient to imply exactness of the sequence $A\to B\to C$.

Given any object $X$ in $\cB$ we have two-way approximation sequences
\[
	Y\xrarrow iP\xrarrow pX,\quad Z\xrarrow jQ\xrarrow qY
\]
giving an exact sequence of functors
\[
	0\to \cB(X,-)\xrarrow{p^\ast}\cB(P,-)\xrarrow{(iq)^\ast}\cB(Q,-)
\]
Thus, the snake lemma, applied to the following diagram
\[
\xymatrixrowsep{12pt}\xymatrixcolsep{10pt}
\xymatrix{%begin xy matrix
0\ar[r]&\cB(P,A)\ar[d]\ar[r] &
	\cB(P,B)\ar[d]\ar[r] &
	\cB(P,C)\ar[d]\ar[r]&0\\
0\ar[r]&\cB(Q,A) \ar[r]& 
	\cB(Q,B) \ar[r]&
	\cB(Q,C)\ar[r]&0
	}%end xy matrix
\]
whose rows are exact by (1), gives an exact sequence $0\to \cB(X,A)\to \cB(X,B)\to \cB(X,C)$, showing that $A$ is the kernel of $g:B\to C$. Similarly, exactness of (2) shows that $C$ is the cokernel of $f:A\to B$. In particular, $gf=0$.

Next, we will show that the sequence $A\cofib B\onto C$ is the pushout of the 2-way approximation sequence $X\cofib P\onto C$ for $C$ along a map $X\to A$. By the exactness of (1), the morphism $P\to C$ lifts to $B$ and we have an induced map of kernels $X\to A$ giving
\[
\xymatrixrowsep{14pt}\xymatrixcolsep{16pt}
\xymatrix{%begin xy matrix
X\ar[d]\ \ar@{>->}[r] &
	P\ar[d]\ar@{->>}[r] &
	C\ar[d]^=
\\
A\ \ar@{>->}[r] &
	B\ar@{->>}[r] &
	C
	}%end xy matrix
\]
For any $Q$ in $\cB_0$, this induces the following commuting diagram with exact rows.
\[
\xymatrixrowsep{12pt}\xymatrixcolsep{10pt}
\xymatrix{%begin xy matrix
0\ar[r]&\cB(C,Q)\ar[d]_=\ar[r] &
	\cB(B,Q)\ar[d]\ar[r] &
	\cB(A,Q)\ar[d]\ar[r]&0\\
0\ar[r]&\cB(C,Q) \ar[r]& 
	\cB(P,Q) \ar[r]&
	\cB(X,Q)\ar[r]&0
	}%end xy matrix
\]
By the Mayer-Vietoris argument, this gives a short exact sequence
\[
	0\to \cB(B,Q)\to \cB(A\oplus P,Q)\to \cB(X,Q)\to0
\]
Therefore, $B$ is the cokernel in the sequence $X\to P\oplus A\to B$. In other words, $A\cofib B\onto C$ is the pushout of $X\cofib P\onto C$ along the given morphism $X\to A$. So, $A\cofib B\onto C$ is one of the designated exact sequences in the Frobenius category $\cB$.
\end{proof}

\begin{cor}\label{cor: cB has a unique exact structure} Given an additive category $\cB$ and full subcategory $\cB_0$, there is at most one exact structure on $\cB$ to make it into a Frobenius category with $\cB_0$ being the subcategory of projective-injective objects.
\end{cor}

\begin{cor}\label{cor: psi is automatically exact}
Let $\cB$ be a Frobenius category with $\cB_0$ being the subcategory of projective-injective objects. Then any additive automorphism $\psi$ of $\cB$ so that $\psi \cB_0=\cB_0$ will be an exact functor.
\end{cor}

\begin{proof}
Given any exact sequence $A\cofib B\onto C$ in $\cB$ and any $P\in\cB_0$, we have two exact sequences of additive groups: $0\to \cB(P,A)\to\cB(P,B)\to\cB(P,C)\to0$ and $0\to \cB(C,P)\to\cB(B,P)\to\cB(A,P)\to0$. Since $\psi$ is an additive automorphism of $\cB$, this gives an exact sequence
\[
	0\to \cB(\psi P,\psi A)\to \cB(\psi P,\psi B)\to \cB(\psi P,\psi C)\to0
\]
and the dual sequence. Since $\psi\cB_0=\cB_0$, this sequence and the dual sequence, show that $\psi A\to\psi B\to\psi C$ is exact.
\end{proof}

\begin{defn}\label{defn: approximation subcategory}
In case $\cB$ admits the structure of a Frobenius category with $\cB_0$ being the subcategory of projective-injective objects then we say that $\cB_0$ is an \emph{approximation subcategory} for $\cB$.
\end{defn}

We need an easy special case of this, namely, $\cB_0=\cB$ is always an approximation subcategory of $\cB$ for trivial reasons. We call the resulting Frobenius category \emph{trivial}.

\begin{prop}
Any additive Krull-Schmidt category $\cB$ becomes a trivial Frobenius category if we take $\cB_0=\cB$ and define a sequence $A\cofib B\onto C$ to be exact if and only if it is split exact (i.e., $B\cong A\oplus C$.)\qed
\end{prop}

%\newpage

\subsection{The Frobenius category} Suppose $G$ is a finite group of order $|G|$ and $R$ is a discrete valuation ring with unique maximal ideal $\mm$ generated by $u\in R$. Suppose that the residue field $K=R/\mm$ has characteristic not dividing $n$, the order of $G$, (and thus $R$ contains $\frac1{n}$). Suppose also that $R$ contains all $n$th roots of unity, i.e., the group of units of $R$ contains a cyclic group of order $n$.

Let $\cA$ be a small additive Krull-Schmidt $R$-category and let $\cX$ be the set of indecomposable objects of $\cA$. Suppose that, for any two objects $X,Y\in \cA$, the morphism set is $\cA(X,Y)\cong R$ with a chosen generator $f_{XY}$ which is equal to the identity when $X=Y$. We call $f_{XY}$ the \emph{basic morphism} from $X$ to $Y$. Suppose further that, for all $X,Y,Z\in \cX$ we have
\begin{equation}\label{composition rule for tN categories}
	f_{YZ}\circ f_{XY}=u^n f_{XZ}
\end{equation}
for some nonnegative integer $n$. Then, it is easy to see that $n=c(X,Y,Z)$ where $c:\cX^3\to\NN$ is a reduced cocycle in the following sense.

\begin{defn}\cite{IT11}\label{defn of reduced cocycle}
For any set $\cX$, a \emph{reduced cocycle} on $\cX$ is defined to be a function $c:\cX^3\to\NN$ satisfying the following two conditions.
\begin{enumerate}
\item $c(X,Y,Z)-c(W,Y,Z)+c(W,X,Z)-c(W,X,Y)=0$
for all $W,X,Y,Z\in\cX$.
\item $c(X,X,Y)=0=c(X,Y,Y)$
for all $X,Y\in\cX$. 
\end{enumerate}
We define a \emph{cyclic poset} to be a set $\cX$ together with a reduced cocycle $c:\cX^3\to\NN$. Note that every subset of a cyclic poset is a cyclic poset whose cocycle is given by restricting the cocycle of the larger set to the smaller set. Two elements $X,Y\in\cX$ are \emph{equivalent} if $c(X,Y,X)=0$. (It is easy to show that this is an equivalence relation.)
\end{defn}

Conversely, given the cyclic poset $\cX$ and DVR $R$ with uniformizer $u$, we can reconstruct the category $\cA$ as $add\,\cP(\cX)$ where $\cP(\cX)$ is given as follows. (See \cite{IT11}.)

\begin{defn}
Let $\cP(\cX)=\cP(\cX,c,R)$ denote the $R$-category whose object set is $\cX$ with all morphism sets equal to $R$ and composition given by the rule \eqref{composition rule for tN categories} with $n=c(X,Y,Z)$.
\end{defn}

As an example, a cyclically ordered set is equivalent to a cyclic poset where $c$ is bounded by $1$. In that case $X,Y,Z$ are cyclically ordered if and only if $c(X,Y,Z)=0$. We are particularly interested in the cyclically ordered sets $(\cX,c)=S^1=\RR/2\pi\ZZ$ and $S^1_\pi=\RR/\pi\ZZ$.
%The cyclic poset $\cX_\pi$ is not cyclically ordered since its uniquely determined cocycle is bounded by $2$.

Finally, let $\cX_0$ be a subset of $\cX$ which is closed under isomorphism so that the additive subcategory $\cA_0$ generated by $\cX_0$ is an approximation subcategory of $\cA$. By Definition \ref{defn: approximation subcategory} this means that $\cA$ admits a (uniquely determined) structure of a Frobenius category so that $\cA_0$ is the full subcategory of projective-injective objects. 

\begin{defn}
The pair $(\cX,\cX_0)$ consisting of a cyclic poset $\cX$ and subset $\cX_0$ will be called a \emph{Frobenius cyclic poset} and $\cX_0$ will be called \emph{approximation subset} of $\cX$ if, for any choice of $(R,u)$, $add\,\cP(\cX)$ admits the structure of a Frobenius category with $\cX_0$ being the set of indecomposable projective-injective objects. This Frobenius category will be denoted $\cF(\cX,\cX_0)$. (As a category it depends only on $\cX$, but its exact structure, depending only on $\cX_0$, is given by Proposition \ref{exact structure of an approximation category}.) Let $\cF(\cX_0)=add\,\cP(\cX_0)$ denote full subcategory of projective-injective objects of $\cF(\cX,\cX_0)$.
\end{defn}

%we say that $\cX_0$ is an \emph{approximation subset} of the cyclic poset $\cX$

The particular cases that we are interested in are when $\cA=\cF_\pi$ is the continuous Frobenius category and also the full subcategories $\cF_b$ for $0<b\le\pi$ constructed in \cite{IT09}. We will recall the definitions when we restrict to these special cases. First, we develop the theory more generally using the easy case $\cX=\cX_0=S^1$ as an example.

We take the objects of $\cF(\cX,\cX_0)=add\,\cP(\cX)$ to be ordered direct sum of elements of $\cX$. Such sums form a full subcategory equivalent to the entire category, so there is no loss of generality, only a simplification of the topology of the category. With this assumption, the Frobenius category $\cF(\cX,\cX_0)$ is completely determined by $\cX$ with cocycle $c$, subset $\cX_0$ and $(R,u)$.

%\begin{defn} Given any cyclic poset $(\cX,c)$ and $(R,u)$ as above, we use the notation $\cP(\cX,c,R)$ (with $u$ understood, or simply $\cP(\cX)$ if $R, u$ and $c$ are understood) for the $R$-category whose object set is $\cX$ with all morphism sets equal to $R$ and composition given by the rule \eqref{composition rule for tN categories} with $n=c(X,Y,Z)$. Then the Frobenius category $\cF(\cX,\cX_0,c,R)$ (sometimes abbreviated to $\cF(\cX,\cX_0)$) is equal to the additive category $add\,\cP(\cX)$ generated by $\cP(\cX)=\cP(\cX,c,R)$ with subcategory of projective-injective objects equal to $add\,\cP(\cX_0)$. One needs to verify that $\cF(\cX,\cX_0)$ admits the structure of a Frobenius category. If it does then Corollary \ref{cor: cB has a unique exact structure} implies that this structure is unique.\end{defn}

\begin{eg}
One easy but important example is the case $\cX=S^1=\RR/2\pi\ZZ$ with cocycle $c$ given by the cyclic ordering of $S^1$, namely, $c(X,Y,Z)=0$ if $X,Y,Z\in S^1$ are represented by real numbers $x\le y\le z\le x+2\pi$ and $c(X,Y,Z)=1$ otherwise. In that case we define $\cF(S^1)$ to be the trivial Frobenius category $\cF(S^1,S^1,R_0)$ where $R_0$ is a discrete valuation ring with uniformizer $t$. (This is related to $(R,u)$ by the equation $t=u^2$. I.e., $R\cong R_0[\sqrt t]$.) ``Trivial'' means that a sequence in $\cF(S^1)$ is exact if and only if it is split exact. We use $P_x$ to denote the unique indecomposable object of $\cF(S^1)$ corresponding to the point $x+2\pi\ZZ\in S^1$. (Thus $P_x=P_{x+2\pi n}$ for any integer $n$.)
\end{eg}

%\newpage

\subsection{Action of $G$} Suppose that $(\cX,\cX_0)$ is a Frobenius cyclic poset with cocycle $c$.

\begin{defn}\label{def of group action}
By an \emph{action} of $G$ on $(\cX,\cX_0,c)$ we mean an action of $G$ on $\cX$ which preserves $c$, leaves $\cX_0$ invariant. I.e.:
\begin{enumerate}
\item $
	c(X,Y,Z)=c(\g X,\g Y,\g Z)
$
for all $X,Y,Z\in \cX, \g\in G$,
\item $\g \cX_0=\cX_0$ for all $\g\in G$. % This section does not require H_X te be abelian!
\end{enumerate}
\end{defn}

An action of $G$ on $\cX$ induces an $R$-linear action of $G$ on the category $\cF=\cF(\cX,\cX_0)$ where each $\g\in G$ acts by sending $\bigoplus X_i$ to $\bigoplus \g X_i$ and sending basic morphisms to basic morphisms: $\g f_{XY}=f_{\g X,\g Y}$ and extending $R$-linearly.

\begin{lem}
The action $\g:\cF\to\cF$ of any $\g\in G$ is an exact functor. 
\end{lem}

\begin{proof}
This is an example of Corollary \ref{cor: psi is automatically exact}.
\end{proof}

\begin{defn}\label{defn: definition of completed orbit category}
Given a Frobenius category $\cF=\cF(\cX,\cX_0)$ with an action of a finite group $G$ induced by an action of $G$ on the Frobenius cyclic poset $(\cX,\cX_0)$, let $\cF^G=\cF^G(\cX,\cX_0)$ be the exact category given as follows.
\begin{enumerate}
\item The \emph{objects} of $\cF^G$ are pairs $(X,\xi)$ where $X\in\cF$ and $\xi$ is a family of isomorphisms \[
\xi_\g :\g X\xrarrow\approx  X
\]
for all $\g\in G$ with the property that
\[
	\xi_{\a\b}=\xi_\a\circ \a\xi_\b:\a\b X\to \a X\to X
\]
for all $\a,\b\in G$.
\item A \emph{morphism} $f:(X,\xi^X)\to (Y,\xi^Y)$ in $\cF^G$ is a morphism $f:X\to Y$ in $\cF$ so that $f\circ \xi^X_\g=\xi^Y_\g\circ \g f$ for all $\g\in G$:
\[
%\xymatrixrowsep{10pt}\xymatrixcolsep{10pt}
\xymatrix{%begin xy matrix
X\ar[r]^f &
	Y\\
\g X \ar[u]^{\xi^X_\g}\ar[r]^{\g f}& 
	\g Y\ar[u]_{\xi^Y_\g}
	}%end xy matrix
\]
\item An \emph{exact sequence} in $\cF^G$ is a sequence of morphisms in $\cF^G$:
\[
	(X,\xi^X)\xrarrow f(Y,\xi^Y)\xrarrow g(Z,\xi^Z)
\]
so that $X\cofib_f Y\onto_g Z$ is exact in $\cF$.
\end{enumerate}
\end{defn}
We will see that $\cF^G$ is the idempotent completion of the orbit category $\cF(\cX,\cX_0)^G$ which, in our notation, is the full subcategory of $\cF^G$ with objects $SX$ for all $X\in\cF(\cX,\cX_0)$.

It is straightforward to show that $\cF^G$ is an exact category. For example, given an exact sequence as in the definition above and a morphism $f:(X,\xi^X)\to (W,\xi^W)$, let $W\cofib P\onto Z$ be the pushout of $X\cofib Y\onto Z$ in $\cF$. Then $\g W\cofib \g P\onto \g Z$ is the pushout of $\g X\cofib \g Y\onto \g Z$ along $\g f:\g X\to \g W$ and, therefore, there is a unique induced map $\xi_\g:\g P\to P$:
\[
%\xymatrixrowsep{10pt}\xymatrixcolsep{10pt}
\xymatrix{%begin xy matrix
X\ar[r] &
	Y\oplus W\ar[r] &
	P\\
\g X\ar[u]^{\xi^X_\g} \ar[r]& 
	\g Y\oplus \g W \ar[u]^{\xi^Y_\g\oplus \xi^W_\g} \ar[r]&
	\g P\ar@{-->}[u]_{\xi_\g} 
	}%end xy matrix
\]
By uniqueness of induced maps on cokernels, we have $\xi^P_{\a\b}=\xi^P_\a\circ \a\xi^P_\b$ since the corresponding operations on $X,Y,W$ satisfy this formula. So, we have the pushout: $(W,\xi^W)\cofib (P,\xi^P)\onto (Z,\xi^Z)$ in $\cF^G$ and it is exact.

%The following observation follows directly from the definitions.

\begin{prop}
The $G$-action on $\cF$ extends to an exact $G$-action on $\cF^G$ given on objects by $\g(X,\xi)=(\g X, \xi^\g)$ where $(\xi^\g)_\b=(\xi_\g)^{-1}\circ \xi_{\b\g}:\b\g X\to X\to \g X$ and sending the morphism $ f:(X,\xi)\to (Y,\xi)$ to the morphism $\g f:(\g X, \xi^\g)\to(\g Y, \xi^\g)$.% I.e., $\psi$ commutes with the forgetful functor $\cF^\psi\to\cF$.
\end{prop}

\begin{proof}
This follows from the definitions. For example, the following diagram commutes for any $f:(X,\xi)\to (Y,\xi)$ showing that $\g  f:(\g  X, \xi^\g )\to(\g  Y, \xi^\g )$ is a morphism in $\cF^G$.
\[
%\xymatrixrowsep{10pt}\xymatrixcolsep{5pt}
\xymatrix{%begin xy matrix
\g X\ar[dr]^{\xi_\g }\ar[rr]^{\g f} &&
	\g Y\ar[dr]^{\xi_\g }\\
	& X\ar[rr]^(.3)f &&Y\\
\b\g X\ar[uu]^{\xi^\g _\b}\ar[ru]_{\xi_{\b\g }} \ar[rr]_{\b\g f}&& 
	\b\g Y \ar[uu]^(.3){\xi^\g _\b}\ar[ru]_{\xi_{\b\g }}
	}%end xy matrix
\]
\end{proof}

%\newpage
\subsection{Adjoint functors}

We use adjoint functors to show that $\cF^G$ has enough projective-injective objects.

\begin{defn}
For any object $X$ in $\cF$, let $SX=(\bigoplus \g X,\xi)$ where $\xi_\b:\b SX\to SX$ is the map which sends the summand $\b(\g X)$ of $\b SX$ to the summand $(\b\g)X$ of $SX$ by the identity map. I.e., $\xi$ is given by a permutation matrix.
\end{defn}

\begin{prop}\label{adjunction formula}
$S:\cF\to\cF^G$ is an exact functor which is both left and right adjoint to the forgetful functor $F:\cF^G\to\cF$. In other words, $\cF^G(SX,(Y,\xi))\cong \cF(X,Y)$ and $ \cF^G((X,\xi),SY)\cong \cF(X,Y) $.
\end{prop}

\begin{proof} $S$ is an exact functor since $F\circ S=\bigoplus \g_\ast$ is exact, being a direct sum of exact functors $\g_\ast$ which give the exact action of $\g$ on $\cF^G$. The adjunction is given by sending the morphism $f:X\to Y$ to
\[
\sum_{\g\in G} \g f\circ \xi_\g^{-1}:(X,\xi_X)\xrarrow{\sum\xi_\g^{-1} }\bigoplus \g X\xrarrow{\bigoplus \g f} \bigoplus \g Y=SY
\]
and to $\sum \xi_\g\circ \g f:SX=\bigoplus \g X\xrarrow{\bigoplus \g f} \bigoplus \g Y\xrarrow{\sum\xi_\g}  (Y,\xi_Y)$.% and to 
 % where $g=\psi f\circ\xi_X$ and $f=\psi g\circ\xi_X$.
\end{proof}

\begin{cor} Let $X$ be an object of $\cF$. Then the following are equivalent.
\begin{enumerate}
\item $SX$ is projective in $\cF^\psi$.
\item $SX$ is injective in $\cF^\psi$.
\item $X$ is projective-injective in $\cF$ (i.e., $X$ is a direct sum of elements of $\cX_0$).
\end{enumerate}
\end{cor}

\begin{proof} $SX$ is projective in $\cF^G$ if $\cF^G(SX,-)$ is an exact functor on $\cF^G$. But $\cF^G(SX,-)=\cF(X,F(-))$ and the forgetful functor $F$ is exact by definition. So, $(3) \then (1)$. Conversely, assume (1). Then, $\bigoplus \cF(X,\g(-))=\cF^G(SX,S(-))$ is an exact functor implying that $ \cF(X,\g(-))$ is an exact functor on $\cF$ for all $\g\in G$. So, $X$ is projective in $\cF$ showing that $(1)\ifff(3)$. The dual argument shows $(2)\ifff(3)$.
% and similarly, $(3)\then (2)$.
%This follows from the fact that $S$ is left and right adjoint to the forgetful functor. For example, suppose that $(Y,\xi)\to (Z,\xi)$ is a proper epimorphism and $f:SX\to (Z,\xi)$ is any morphism. Then $f$ lifts to $(Y,\xi)$ if and only it its adjoint $f_0:X\to Z$ lifts to $Y$. Therefore, $SX$ is projective in $\cF^\psi$ if and only if $X$ is projective in $\cF$. So, (1) is equivalent to (3). Similarly, the fact that $S$ is right adjoint to the forgetful functor implies that (2) is equivalent to (3).
\end{proof}

\begin{cor}
Given an object $X$ in $\cF$ and an object $(Y,\xi)$ in $\cF^G$, a morphism $X\to Y$ in $\cF$ factors through a projective injective object $P$ if and only if its adjoint $SX\to (Y,\xi)$ factors through $SP$. Dually, a morphism $Y\to X$ in $\cF$ factors through $P$ if and only if its adjoint $(Y,\xi)\to SX$ factors through $SP$. Consequently, in the stable categories $\ul\cF$ and $\cF^G$ we have the adjunction:
\[
	\ul\cF^G(SX,(Y,\xi))\cong\ul\cF(X,Y)
\]
\[
	\ul\cF^G((Y,\xi),SX)\cong\ul\cF(Y,X).
\]\qed
\end{cor}

\begin{cor}\label{formula for F(SX,SY)}
For any two indecomposable objects $X,Y$ in $\cF$,
\[
	\cF^G(SX,SY)\cong \bigoplus_{\g\in G}\cF(\g X,Y)\cong R^G
\]
where $(r_\g:\g \in G)\in R^G$ corresponds to
\[
	\left( r_{\b^{-1}\a}f_{\a X,\b Y} \right): \bigoplus \a X=SX\to SY=\bigoplus \b Y
\]
where $f_{AB}:A\to B$ denotes the basic map from $A$ to $B$.
Furthermore, the $\a X\to \g Z$ component of the composition of $(s_\g):SX\to SY$ and $(r_\g):SY\to SZ$ is given by
\[
	\sum_{\b\in G} r_{\g^{-1}\b}s_{\b^{-1}\a} u^n f_{\a X,\g Z}
\]
where $n=c(\a X,\b Y,\g Z)$.\qed
\end{cor}

\begin{thm}\label{FG is a Frobenius category} Given an action of a finite group $G$ on a Frobenius cyclic poset $(\cX,\cX_0)$, the construction above gives a Frobenius category $\cF^G=\cF^G(\cX,\cX_0)$ whose projective-injective objects are the components of $SP$ for some $P\in\cF_0$.
\end{thm}

\begin{proof}
We have seen that $\cF^G$ is an exact category and that the objects $SP$ are projective-injective. It follows that all components of $SP$ are also projective-injective. It remains to show that there are enough projectives and that all projective and injective objects are components of objects of the form $SP$ for some $P\in \cF_0$.

To show that there are enough projectives, let $(X,\xi)\in \cF^G$. Let $P\to X$ be a projective cover in $\cF$. Then the adjoint map $SP\to (X,\xi)$ is a proper epimorphism. 

To show that all projective objects are components of objects of the form $SP$, let $(Q,\xi)$ be any projective object of $\cF^G$. Choose a projective cover $P\to Q$ for $Q$ in $\cF$. Then, by adjunction, we have a proper epimorphism $SP\to (Q,\xi)$. Since $(Q,\xi)$ is projective, this epimorphism splits and $(Q,\xi)$ is a direct summand of $SP$ as claimed.\end{proof}

%\newpage

\subsection{Krull-Schmidt Theorem} Suppose now that the \emph{effective stabilizer}
\[
	H_X:=\{\b\in G\,|\, \b X\cong X\}
\]
of every $X\in \cX$ is abelian. Then we will show that every object of $\cF^G$ is a direct sum of indecomposable objects of the form $S_\ll X$ which we now define. (For completeness the case of nonabelian $H_X$ is treated in the appendix. The extension to infinite groups $G$ will be explained in the next paper.)

%\subsubsection{Singular orbits}

\begin{defn} Let $X\in \cX$ and let $\ll:H\to R^\times$ be any homomorphism from $H=H_X$ to the group of units of $R$. Then we define the object $S_\ll X$ of $\cF^G$ as follows. First choose representatives $\s_i$ from the left cosets of $H$ in $G$ so that $G=\coprod \s_i H$. Then
\[
	S_\ll X=\left(
	\bigoplus \s_i X,\xi^\ll
	\right)
\]
where, for each $\g\in G$, $\xi^\ll_\g:\g\bigoplus \s_iX\to \bigoplus \s_jX$ is given on each component by
\[
	\xi^\ll_\g=\ll(\eta_j):\g \s_i X\to \s_jX
\]
which denotes $\ll(\eta_j)$ times the unique basic isomorphism $\g \s_i X\to \s_jX$ where $\eta_j\in H$ and $\s_j\in G$ are uniquely determined by the equation $\g \s_i=\s_j\eta_j$. Note that $S_1X=SX$ when $X$ is regular (so that $H_X$ is the trivial group).
\end{defn}

We will verify that $S_\ll X$ is an object of $\cF^G$ using the notation $\eta_j=\eta_j(\g)$. Let $\g,\b\in G$. Then %we have $\b \s_i=\s_j \eta_j(\b)$ and $\g \s_j=\s_k \eta_k(\g)$. So, 
\[
	\g\b\s_i=\g\s_j \eta_j(\b)=\s_k \eta_k(\g)\eta_j(\b)=\s_k\eta_k(\g\b)
\]
and we conclude that $\eta_k(\g\b)=\eta_k(\g)\eta_j(\b)$. So,
\[
	\xi^\ll_{\g\b}=\ll(\eta_k(\g\b)):\g\b\s_i X\to  \s_kX
\]
is equal to
\[
	\xi^\ll_\g\circ \g\xi^\ll_\b=\ll(\eta_k(\g))\ll(\eta_j(\b)):\g\b\s_i X\to \g \s_jX\to \s_kX
\]
as required.

\begin{eg}\label{eg:singular objects for cyclic groups}
Suppose that $H_X=G\cong\ZZ/n$ is a cyclic group of order $n$ generated by $\g$. Then $\ll(\g)=z\in R$ is an $n$th root of unity and $S_\ll X=(X,\xi)$ with $\xi_{\g^k}:\g^k X\to X$ equal to $z^k$ times the unique basic morphism $\g^kX\to X$ which is an isomorphism. Since $S_\ll X$ depends only on $\ll(\g)=z$, we denote it by $Z_z(X)$.
\end{eg}

\begin{lem}
$S_\ll X$ is independent of the choice of representatives $\s_i$ up to isomorphism in $\cF^G$.
\end{lem}

\begin{proof}
Suppose that $\s_i'$ is another choice of representatives of the left cosets of $H$ in $G$. Then $\s_i'=\s_i \a_i$ for some $\a_i\in H$. Then an isomorphism $(\bigoplus \s_i' X,\xi')\to (\bigoplus \s_i X,\xi)$ is given on each component $\s_i'X\to \s_iX$ by $\ll(\a_i)$ times the unique basic isomorphism.
\end{proof}

%\begin{rem}The restriction of the cocycle $c$ to the orbit $GX$ of a regular object $X$ gives a $G$-equivariant cocycle on $G$. So, it corresponds to a central extension $\ZZ\to U\to G$. For $G=\ZZ/2=\{1,\psi\}$, we have $U\cong \ZZ$ if and only if the number $n=c(X,\psi X,X)$ is odd. In the examples below we will have $n=1$ for all regular objects in $\cF^\psi=\cF^\psi(S^1)$ and $n=2$ for all regular objects in $\cF_b^\psi$ for all $0<b\le\pi$.\end{rem}

%An element $X\in \cX$ will be called \emph{singular} if it is not regular, i.e., $H_X$ is nontrivial.

%\begin{defn}We define a \emph{regular object} of $\cF^G$ to be one of the form $SX$ where $X$ is a regular object of $\cF$, i.e., $X$ is indecomposable and $X\not\cong \g X$ for all $\g\neq1$ in $G$. A \emph{singular object} of $\cF^G$ is one of the form $S_\ll Y$ where $Y$ is a singular object of $\cF$ and $\ll:H_X\to R^\times$ is a one-dimensional representation of the effective stabilizer $H_X$ of $X$.\end{defn}

%We will show that each $S_\ll X$ is indecomposable (including the case $S_1X=SX$ for regular $X$) and that every object of $\cF^G$ is isomorphic to a direct sum of regular objects $SX$ and singular objects $S_\ll Y$.

\begin{lem}\label{decomposition of singular SX}
$SX$ is isomorphic to the direct sum of $S_{\ll_j}X$ for all homomorphisms $\ll_j:H_X\to R^\times$. In the special case when $H_X=G=\ZZ/n$ we get $SX=\bigoplus_{z^n=1} Z_z(X)$.
\end{lem}

\begin{proof}
We recall that, by assumption, $H=H_X$ is abelian and there are $m$ distinct representations $\ll_j:H\to R^\times$ where $m=|H|$. Therefore, $SX\cong\bigoplus S_{\ll_j}X$ as objects of $\cF$ since both contain $m$ direct summands isomorphic to $\s_i X$ for each coset $\s_i H$ of $H$ in $G$.

We note that the $m\times m$ matrix $T$ with entries $\ll_j(\eta_i), \eta_i\in H,\ll_j:H\to R^\times$ is invertible since it becomes the character table of $H$ with inverse given by $(\frac1m\ll_i(\eta_j)^{-1})$ when its entries are reduced modulo $(u)$. This implies that the $\cF^G$-morphism $f:SX\to \bigoplus S_\ll X$ which is adjoint to the diagonal morphism $X\to \oplus_\ll X\into \bigoplus_\ll \bigoplus \s_k X$ is an isomorphism since the matrix of $f$ is the block diagonal matrix $T\oplus T\oplus \cdots \oplus T$. (The $(\s_k \eta_i,\ll_j \s_k)$-entry of $f$ is $\ll_j(\eta_i)$ times the identity morphism $\s_k X\to \s_k X$ and the other entries are zero.)
\end{proof}

The proof of the above lemma can be modified to prove the following lemma which will imply that the components $S_\ll X$ of $SX$ are indecomposable and nonisomorphic.

\begin{lem}
An $\cF^G$ morphism $f:SX\to \bigoplus S_{\ll_j}X$ is an isomorphism if and only if the $m$ components $r_j\in \End_\cF(X)=R$ of the adjoint map composed with the projection of each $S_{\ll_j}X$ to the factor $\s_1X=X$:
\[
	X\to \bigoplus S_{\ll_j}X\to \bigoplus_j X 
\]
are all invertible.
\end{lem}

\begin{proof}
To determine if $f$ is an isomorphism, it suffices to examine, for each $\s_k$, the $m\times m$ block of $f$ which represents the induced endomorphism of $\bigoplus_m\s_k X$. $f$ is invertible if and only if these blocks are invertible for all $k$. However, modulo $(u)$ the $m^2$ entries of the $k$-th block are $r_j \ll_j(\eta_i)$. This gives an invertible matrix if and only if $r_j$ is invertible.
\end{proof}

We say that an object of any additive category is \emph{strongly indecomposable} if its endomorphism ring is a local ring (i.e., the complement of the group of units is a two-sided ideal). It follows easily that such objects are indecomposable since the equation $e(1-e)=0$ implies that $e=0$ or 1.% This lemma says that $SX$ is strongly indecomposable for regular objects $X$.

\begin{lem}\label{lem: all objects are strongly indecomposable}
$S_{\ll_j}X$ are strongly indecomposable nonisomorphic objects in $\cF^G$.% with local endomorphism ring.
\end{lem}

\begin{proof} If $h\in \End_\cF(X)=R$, let $\ov h\in K=R/(u)$ be the reduction of $h$ modulo the maximal ideal $(u)$. Consider the ring homomorphism $\pi:E=\End_{\cF^G}(S_\ll X)\to K$ given by sending the endomorphism $f$ of $S_\ll X\cong \bigoplus \s_k X$ to $\ov f_{11}$, the reduction modulo $(u)$ of the $X=\s_1X\to \s_1X=X$ component $f_{11}$ of $f$. The kernel of $\pi$ is a 2-sided ideal in $E$. To show that $S_\ll X$ is strongly indecomposable, it suffices to show that $f\in E$ is invertible if and only if $\pi(f)\neq0$. This condition is clearly necessary. To show that it is sufficient, let $f$ be an endomorphism of $S_\ll X$ so that $\pi(f)\neq0$. If we compose any isomorphism $SX\cong \bigoplus S_{\ll_j}X$ with the endomorphism of $S_{\ll_j}X$ which is the endomorphism $f$ on the component $S_\ll X$ and the identity on all other components, then the result will be an isomorphism by the previous lemma. Therefore $f$ is an isomorphism.

To see that the $S_{\ll_j}X$ are nonisomorphic, take any $\cF^G$ morphism $g:S_\ll X\to S_{\ll'}X$. If $\ll\neq\ll'$ then there must be an element $\eta\in H$ so that $\ll(\eta)\neq\ll'(\eta)$ and they remain nonequal modulo $(u)$. But this implies that $\ov g_{11}=0$ since $\ll(\eta)g_{11}=\ll'(\eta)g_{11}$ by the assumption that $g$ is a morphism in $\cF^G$. But then, $g$ is not an isomorphism in $\cF$.
\end{proof}

The strongly indecomposable objects $S_\ll X$ can be isomorphic. 

\begin{lem}\label{isomorphism between singular objects} Suppose $\a,\b\in G$ and $h:\a X\cong \b Y$. Then we must have $\a H_X\a^{-1}=\b H_Y\b^{-1}$. For any homomorphism $\ll:H_X\to R^\times$, let $\ll':H_Y\to R^\times$ be the homomorphism given by $\ll'(\eta)=\ll(\b\a^{-1}\eta\a\b^{-1})$. Then we have an isomorphism $g:S_\ll X\cong S_{\ll'}Y$ whose $\a X,\b Y$ component is equal to $h$. Furthermore, any $\cF^G$ morphism $f:S_\ll X\to S_{\ll'} Y$ so that $f$ induces an isomorphism on the $ij$ components: $\s_i X\cong \t_jY$ is an isomorphism.
\end{lem}

\begin{rem}
When $G$ is abelian we have $H_X=H_Y$ and $\g=\g':H_X\to R^\times$.
\end{rem}

\begin{proof} If $\s_i$ are representatives for the left cosets of $\a H_X\a^{-1}=\b H_Y\b^{-1}$, then $S_\ll X=\bigoplus \s_i\a X$ and $S_{\ll'}Y=\bigoplus \s_i\b Y$ and an isomorphism $S_\ll X=\bigoplus \s_i\a X\cong S_{\ll'}Y=\bigoplus \s_i\b Y$ is given by the direct sum of the isomorphisms $\s_i h:\s_i\a X\cong \s_i\b Y$. 

For the second statement we note that the composition $g^{-1}f:S_\ll X\to S_\ll X$ must be an automorphism since it is an automorphism on the $\cF$ component $\s_i X$ and $S_\ll X$ is indecomposable.
\end{proof}

\begin{thm}
If the effective stabilizer of every $X\in\cX$ is abelian then every object in $\cF^G$ is isomorphic to a direct sum of strongly indecomposable objects of the form $S_\ll X$.
\end{thm}

\begin{proof}
Let $Z$ be an object of $\cF^G$ of minimal length as an object of $\cF$ so that $Z$ is not a direct sum of objects of the form $S_\ll X$. Then $Z$ must be indecomposable in $\cF^G$. Consider the $\cF^G$ morphism $f:SZ\to Z$ which is adjoint to the identity map $Z\to Z$. Then $f$ is a split epimorphism as a morphism in $\cF$. Similarly, we have a $\cF^G$ morphism $g:Z\to SZ$ which is a split monomorphism in $\cF$. The composition $SZ\to Z\to SZ$ is an isomorphism on certain components as a morphism in $\cF$. When we decompose $SZ$ into indecomposable summands of the form $S_\ll X$, there must be some component of the composition: $S_\ll X\to Z\to S_{\ll'}Y$ which, when considered as a morphism in $\cF$, becomes an isomorphism on some component. By Lemma \ref{isomorphism between singular objects} above, this must be an isomorphism $S_\ll X\cong S_{\ll'}Y$ showing that $Z\cong S_\ll X\cong S_{\ll'}Y$ as claimed.
\end{proof}

\subsection{The $G=\ZZ/p$ case} Suppose that $G=\ZZ/p=\<\psi|\psi^p\>$ where $p$ is a prime not equal to the characteristic of the field $R/\mm$ and $R$ contains all $p$-th roots of unity. Then $X\in\cX$ is singular if and only if $\psi X\cong X$. Recall from Example \ref{eg:singular objects for cyclic groups} the notation $Z_z(X):=S_{\ll_z} X$ where $z^p=1$ and $\ll_z:G\to R^\times$ is the character given by $\ll_z(\psi^n)=z^n$.

\begin{cor}\label{Frobenius category for cyclic groups}
When $G=\ZZ/p$ every object of $\cF^G$ is isomorphic to a direct sum of strongly indecomposable elements of the form $SX$ for regular $X$ and $Z_z(Y)$ with $z^p=1$ for singular $Y$. Furthermore, 
\begin{enumerate}
\item $Z_x(X)\cong Z_y(Y)$ if and only if $x=y$ and $Y\cong \g X$ in $\cF$ for some $\g\in G$. 
\item There is no nonzero morphism $Z_x(X)\to Z_y(Y)$ for any singular $X,Y$ when $x\neq y$.
\item $\cF^G(Z_z(X),Z_z(Y))=\cF(X,Y)\cong R$ for any singular $X,Y$ and any $z$ with $z^p=1$.
\item For singular $X$ we have
\[
	SX\cong\bigoplus_{z^p=1}Z_z(X).
\]
\end{enumerate}
\end{cor}

\begin{proof}
By the theorem, objects in $\cF^G$ have components $S_\ll X$. Example \ref{eg:singular objects for cyclic groups} shows that the singular objects have the form $Z_z(Y)$. (1) follows from Lemma \ref{isomorphism between singular objects} and the remark that follows it.  

To prove (2) and (3), suppose that $f:Z_x(X)\to Z_y(Y)$ is a morphism in $\cF^G$. Considered as a morphism in $\cF$, $f$ is a scalar, say $r$ times the basic morphism $X\to Y$. Then $\g f:\g X\to \g Y$ is the same scalar $r$ times the basic morphism $\g X\to \g Y$. For $f$ to be a morphism in $\cF^G$ the following diagram must commute:
\[
\xymatrixrowsep{15pt}\xymatrixcolsep{15pt}
\xymatrix{%begin xy matrix
\g X\ar[d]_x\ar[r]^r &
	\g Y\ar[d]^y\\
X \ar[r]^r& 
	Y
	}%end xy matrix
\]
I.e., we must have $yr=rx\in R$. Therefore, when $x\neq y$ we must have $r=0$. This proves (2). When $x=y$, there is no restriction on $r$. So, all $\cF$ morphisms $f:X\to Y$ are also $\cF^G$ morphisms, proving (3). Finally, (4) is a special case of Lemma \ref{decomposition of singular SX}.
\end{proof}

%\newpage

\subsection{The $G=\ZZ/2$ case}

We now restrict to the case $G=\ZZ/2$. Then, the only restriction on $R$ is that $K=R/\mm$ has characteristic different from $2$. Then $R$ will automatically contain two distinct square roots of unity: $\pm1$. When $G$ has order $2$, we will denote the nontrivial element of $G$ by $\psi$ and write $\cF^\psi$ instead of $\cF^G$. The general results of the last few sections, when restricted to the case $G=\ZZ/2$ give the following.

\begin{thm}\label{thm: indecomposable in Fpsi}
$\cF^\psi$ is a Krull-Schmidt $R$ category with indecomposable objects given by
\begin{enumerate}
\item $SX=(X\oplus\psi X,\xi)$ for every regular object $X$ in $\cF$, i.e., $X$ is indecomposable and not isomorphic to $\psi X$, where $\xi=\mat{0&1\\1&0}:\psi X\oplus X\to X\oplus\psi X$.
\item $Z_+(Y)=(Y,\z_Y)$ and $Z_-(Y)=(Y,-\z_Y)$ for every singular object $Y\cong \psi Y$ in $\cF$ where $\z_Y:\psi Y\cong Y$ is the basic morphism.
\end{enumerate}
Furthermore, $\cF^\psi$ is a Frobenius category with indecomposable projective-injective objects defined to be those of the form $SX,Z_\pm (Y)$ where $X,Y$ are regular or singular projective-injective objects of $\cF$.
\end{thm}

\begin{defn}\label{Z2 singular objects}
Indecomposable objects of $\cF^\psi$ of the form $SX$ will be called \emph{regular} objects. $Z_+(X)$ and $Z_-(X)$ will be called \emph{positive} and \emph{negative singular objects}. We call $Z_+(X),Z_-(X)$ a \emph{pair of singular objects} or simply a \emph{singular pair} because, as a special case of Corollary \ref{Frobenius category for cyclic groups}, we have:
\[
	SX\cong Z_+(X)\oplus Z_-(X)
\]
for all singular $X$ in $\cF$.
\end{defn}

\begin{rem}\label{char 2 case}
We note that if $char\,K=2$ then $SX$ is indecomposable in $\cF^\psi$ for all indecomposable objects $X$ of $\cF$. This is certainly true when $X$ is regular. When $X$ is singular this holds because the matrix $\xi=\mat{0&1\\1&0}$, when reduced modulo $\mm$ is conjugate to the matrix $\mat{1&1\\0&1}$ which is a Jordan block. This implies that there is a canonical nonsplit exact sequence
\[
	Z(X)\cof SX\onto Z(X)
\]
where $Z(X)=Z_+(X)=Z_-(X)$ for all singular $X$.
%the endomorphism ring of $SX$ in the stable category $\ul\cF^\psi$ of $\cF^\psi$ is the ring of dual numbers $K[\e]/(\e^2)$.
%When $char\,K\neq 2$ the matrix $\xi=\mat{0&1\\1&0}$ is conjugate to the diagonal matrix $\mat{1&0\\0&-1}$ which implies that
\end{rem}

%\newpage

\subsection{Trivial example}

We need the trivial case of Theorem \ref{thm: indecomposable in Fpsi} when $\cX_0=\cX$, i.e., the case when all exact sequence in $\cF$ split. Assume also that $\psi X\not\cong X$ for all $X\in\cX_0=\cX$. Then all objects of $\cF^\psi$ are regular. An important example of this is $\cF(S^1)=\cF(S^1,S^1,R_0)$ with involution given by rotation by $\pi$: $\psi(P_x)=P_{x+\pi}$ where $P_x$ is the object of $\cF(S^1)$ corresponding to $[x]=x+2\pi\ZZ\in S^1=\RR/2\pi\ZZ$. Then all indecomposable objects of $\cF^\psi(S^1)$ are regular and given by
\[
	SP_x=(P_x\oplus P_{x+\pi},\xi),\quad \xi=\mat{0&1\\1&0}:P_x\oplus P_{x+\pi}\xrarrow{\cong} P_{x+\pi}\oplus P_x
\]
\begin{lem}
The endomorphism ring of each regular object $SP_x$ of $\cF^\psi(S^1)$ is equal to $R=R_0[u]$ where $u^2=t$. This is a discrete valuation ring with uniformizer $u$ and the same residue field as $R_0$. Furthermore $\cF^\psi(S^1)(SP_x,SP_y)$ is a free $R$-module with one generator.
\end{lem}

\begin{proof}
Using the adjunction formula we have
\[
	\cF^\psi(S^1)(SP_x,SP_x)=\cF(S^1)(P_x,P_x\oplus P_{x+\pi})=R_0^2
\]
where the two generators of $R_0^2$ are the identity map on $P_x$ and the basic map $P_x\to P_{x+\pi}$ which is ``rotation by $\pi$''. Since rotation by $2\pi$ is multiplication by $t$, this ring is $R=R_0[X]/(X^2-t)=R_0[\sqrt t]$. We let $u=\sqrt t$ be the operator which rotates every point by $\pi$. This is a central operator, i.e., an element of the center of the category $\cF^\psi(S^1)$ and therefore $\cF^\psi(S^1)$ is an $R$-category.

For any two regular $SP_x,SP_y$ of $\cF^\psi(S^1)$, we have $\cF^\psi(S^1)(SP_x,SP_y)=R_0^2$ with one generator equal to $u$ times the other generator. Therefore, this is a free $R$ module of rank 1 as claimed.
\end{proof}

\begin{prop}
$\cF^\psi(S^1)=\cF^\psi(S^1,S^1,R_0)$ is a trivial Frobenius category equivalent to the Frobenius category $\cF(S^1_\pi)=\cF(S^1_\pi,S^1_\pi,R)$ where $S^1_\pi$ is the cyclically ordered set $S^1_\pi=\RR/\pi \ZZ$ and $R=R_0[\sqrt u]$.
\end{prop}

%We will write $\cF(S^1_\pi)=\cF(S^1_\pi,S^1_\pi,R)$.

\begin{proof}
We first show that $\cF^\psi(S^1)$ is a trivial Frobenius category, i.e., that every exact sequence $A\to_f B\to_g C$ in $\cF^\psi(S^1)$ splits. Since all objects are regular, we may assume that $A=SX$ and $B=(Y,\xi)$ where $\xi:\psi Y\to  Y$ is an isomorphism so that $\psi(\xi)=\xi^{-1}$.

By definition of the exact structure of $\cF^\psi(S^1)$, $f:SX\to Y$ is a split monomorphism in $\cF(S^1)$. Since $f$ is a morphism in $\cF^\psi(S^1)$, we have $f=(f_0,f_1):X\oplus \psi X\to Y$ with $f_1=\xi\circ \psi f_0$. By assumption there is a retraction $g=(g_0,g_1):Y\to X\oplus\psi X$ so that $g_1\circ f_0=0$ and $g_1\circ f_1=id_{\psi X}$. This implies that 
\[
	id_X=\psi(g_1\circ f_1)=\psi g_1\circ \psi f_1=\psi g_1\circ \xi^{-1}\circ f_0
\] 
\[
0=\psi(g_1\circ f_0)=\psi g_1\circ \psi f_0=\psi g_1\circ \xi^{-1}\circ f_1
\]
Therefore, $(\psi g_1\circ\xi^{-1},g_1):Y\to X\oplus\psi X$ which is a morphism in $\cF^\psi(S^1)$ is a retraction of $f=(f_0,\psi \xi\circ\psi f_0)$. So, every exact sequence in $\cF^\psi(S^1)$ splits.

To show that $\cF^\psi(S^1)\cong \cF(S^1_\pi,S^1_\pi,R)$, we consider the generators $f_{xy},f_{yz},f_{xz}$ of \[
\cF^\psi(S^1)(SP_x,SP_y)\cong \cF^\psi(S^1)(SP_y,SP_z)\cong \cF^\psi(S^1)(SP_x,SP_z)\cong R\]
Replacing $SP_y$ with the isomorphic $SP_{y+\pi}$ if necessary and doing the same for $SP_z$, we may assume that $x\le y<x+\pi$ and $y\le z<y+\pi$. Then $f_{xy}(P_x)\subseteq P_y$ and $f_{xy}(P_{x+\pi})\subseteq P_{y+\pi}$ and similarly for $f_{yz}$. If $x\le z<x+\pi$ then $f_{yz}\circ f_{xy}=f_{xz}$ since it it the shortest morphism $P_x\to P_z$. If $z\ge x+\pi$ than the composition $P_x\to P_y\to P_z$ factors through $P_{z+\pi}$ and we get $f_{yz}\circ f_{xy}=uf_{xz}$. Therefore, the morphism sets for $\cF^\psi(S^1)$ and its composition rule agree with that of $\cF(S^1_\pi)$ up to isomorphism. (There are two object in each isomorphism class of indecomposable objects in $\cF^\psi(S^1)$ but only one in $\cF(S^1_\pi)$.) Therefore, $\cF^\psi(S^1)$ is equivalent to $\cF(S^1_\pi)$.
\end{proof}

%\newpage

\section{Clusters of continuous type $D$}

We now specialize to the case when the Frobenius category $\cF$ is the continuous Frobenius category $\cF_\pi$. Recall \cite{IT09} that the indecomposable objects are indexed by ordered pairs of points on the circle: $E(x,y)$ and that reversing the order gives an isomorphic but not equal object: $E(x,y)\cong E(y,x+2\pi)$ We also use the notation $M(x,y)=E(x,y+\pi)$ so that $M(x,x)$ is a diameter of the circle. (See definition below.) The involution $\psi$ is given by rotation by $\pi$, or equivalently, reflection through the center. More generally, the generator of $\ZZ/p$ acts by rotation by $2\pi/p$. Singular objects are $M(x,x)$ which are isomorphic to their reflections. And there are no singular objects when $p$ is an odd prime.

As in the case of the continuous cluster category of type $A$, clusters in $\ul\cF^\psi_\pi$ are defined to be discrete laminations of the punctured disk where a \emph{lamination} of the punctured disk is defined to be a maximal compatible subset of $Ind\,\ul\cF^\psi_\pi$.

%We recall that two standard objects $X$ and $Y$ are compatible if there is a sequence of standard objects $X_i$ converging to $X$ so that $\cD(X_i,Y)=0=\cD(Y,X_i)$ for all $i$. 

In keeping with the idea that continuous cluster categories are limits of cluster categories of finite type, we define compatibility in terms of limits. Thus, two indecomposable objects $X,Y$ are \emph{compatible} if there is a sequence of objects $Y_n$ converging to $Y$ so that $\Ext^1(X,Y_n)=\Ext^1(Y_n,X)=0$ for all $n$. For this to make complete sense we first need a topology on the set of indecomposable objects of the category. However, we postpone the technicalities of the topology for another paper. In this paper, we use only a heuristic description to derive a sensible definition of compatibility in the limiting case and the expected homotopy type of the space of objects of the stable category of $\cF_b^\psi$.

\subsection{Continuous Frobenius category} We recall the definition of the continuous Frobenius category $\cF_b$ for any positive $b\le\pi$. This is defined to be the category $\cF(\cX_b,\d\cX_b,R)$ where $\cX_b$ is the cyclic poset given below. (See \cite{IT11} for the general theory. This particular example is fully explained below.)%  whose covering poset $\tilde\cX_\pi$ is given as follows.

We start with what is called the \emph{covering poset} $\tilde\cX_\pi$. This is two copies of a closed strip:
\[
	\tilde\cX_b=\{(x,y,\e)\in \RR\times\RR\times\{+,-\}\ |\ |x-y|\le b\}
\]
We take the partial ordering $(x,y,\e)\le(x',y',\e')$ if and only if $x\le x'$ and $y\le y'$. In particular $(x,y,+)\approx(x,y,-)$ are equivalent in the partial ordering. Define an automorphism $\s$ of $\tilde\cX_b$ by
\[
	\s(x,y,\e)=(y+\pi,x+\pi,-\e)
\]
This has the property that
\begin{enumerate}
\item $X\le \s X$ for all $X\in\tilde\cX_b$.
\item For all $X,Y\in\tilde\cX_b$ there is an $m$ so that $X\le \s^mY$.
\end{enumerate}
Next, we define $\cX_b$ to be the set of $\s$ orbits of elements of $\tilde\cX_b$. We denote the elements of $\cX_b$ by $M(x,y)$ and $M(x,y)'$ for the orbits of $(x,y,+)$ and $(x,y,-)$ respectively. In particular, $M(x,y)'=M(y+\pi,x+\pi)$. In the $E(x,y)$ notation we have:
\[\begin{array}{rcl}
	M(x,y)&=&E(x,y+\pi)\\
	M(x,y)'&=&E(y+\pi,x+2\pi).
\end{array}
\]

Let $c:\cX_b^3\to\NN$ be the mapping which takes any triple of elements $(X,Y,Z)$ in $\cX_b$ to the following number. Choose representatives $\tilde X,\tilde Y,\tilde Z$ of $X,Y,Z$ in $\tilde\cX_\pi$. Then define
\[
	c(X,Y,Z)=i+j-k
\]
where $i,j,k$ are minimal so that $\tilde X\le\s^i\tilde Y,\tilde Y\le\s^j\tilde Y$ and $\tilde X\le \s^k\tilde Z$. It is not too hard to see that this is well-defined and that it is a reduced cocycle (Def. \ref{defn of reduced cocycle}). Finally, we define $\d\cX_b$ to be the subset of $\cX_b$ consisting of all $M(x,y)$ where $|x-y|=b$. (The notation comes from the fact that $\cX_b$ has a natural topology of a compact surface with boundary $\d\cX_b$.)

\begin{thm}\cite{IT09} For any $0<b\le\pi$, $\d\cX_b$ is an approximation subset of $\cX_b$, i.e., the category $\cF(\cX_b,\d\cX_b,R)$ is a Frobenius category with uniquely determined exact structure given by Proposition \ref{exact structure of an approximation category}.
\end{thm}

We will denote this Frobenius category as $\cF_b$. We need the following proposition proved in \cite{IT09}. Recall that $R_0$ is a DVR with uniformizer $t$ and $R=R_0[\sqrt t]$. Thus $R_0\subset R$. We use the notation $E(x,y):=M(x,y-\pi)$.

\begin{prop}\label{prop: exactness is equivalent to split exactness on ends}
There is an exact $R_0$-linear functor
\[
	F:\cF_b=\cF(\cX_b,\d\cX_b,R)\to \cF(S^1)=\cF(S^1,S^1,R_0)
\]
which sends $E(x,y)$ to $P_x\oplus P_y$ and basic morphisms to direct sums of basic morphisms. Furthermore, a sequence of morphisms $A\to B\to C$ in $\cF_b$ is exact if and only if its image $FA\to FB\to FC$ is split exact in $\cF(S^1)$.
\end{prop}

%\newpage

\subsection{Triangulation and involution}

We recall that the stable category of a Frobenius category is triangulable and a specific triangulation is given by a choice of two-way approximation sequences for every object in the category. For the Frobenius category $\cF_b$, we choose the two-way approximation sequences:
\begin{equation}\label{2-way approximation sequence for Fb}
	M(x,y)\xrarrow{\binom{1}{-1}}M(y+b,y)\oplus M(x,x+b)\xrarrow{(1,1)}M(y+b,x+b)
\end{equation}
Since $M(y,x)'=M(x+\pi,y+\pi)$, the consistent choice of two-way approximations for $M(y,x)'$ is given by switching coordinates and putting $'$:
\begin{equation}\label{2-way approximation sequence for Fb prime}
	M(y,x)'\xrarrow{\binom{1}{-1}}M(y,y+b)'\oplus M(x+b,x)'\xrarrow{(1,1)}M(x+b,y+b)'
\end{equation}
Then we have $M(x,y)[1]=M(y+b,x+b)$ and $M(x,y)'[1]=M(y+b,x+b)'=M(x,y)[1]'$. In the case $b=\pi$, $M(x,y)[1]=M(x,y)'$ is isomorphic but not equal to $M(x,y)$.

The key point about the sequence \eqref{2-way approximation sequence for Fb} is that it is invariant under addition of a constant to all coordinates in the sense that, if we add $a$ to both $x$ and $y$ then we get another sequence of the same kind (with $x$ replaced by $x+a$ and $y$ replaced by $y+a$). But, if we switch the two coordinates, the two summands in the middle will switch roles and the sign of the first map will change. This is a subtle point when $x=y$.

Given a morphism $\ov f:X=\bigoplus X_i\to Y$ in the stable category $\ul\cF_b$ of $\cF_b$ represented by a morphism $f$ in $\cF_b$, the distinguished triangle $X\xrarrow {\ov f} Y\xrarrow {\ov g} Z\xrarrow {\ov h} X[1]$ is given by taking the pushout along $f:X\to Y$ of the direct sum of all approximation sequences \eqref{2-way approximation sequence for Fb} starting at each $X_i$:
\[
%\xymatrixrowsep{10pt}\xymatrixcolsep{10pt}
\xymatrix{%begin xy matrix
X=\bigoplus M(x_i,y_i)\ar[d]_{f}\ar[r] &
	\bigoplus M(y_i+b,y_i)\oplus M(x_i,x_i+b)\ar[d]\ar[r] &
	\bigoplus M(y_i+b,x_i+b)\ar[d]^=\\
Y \ar[r]^g& 
	Z \ar[r]^h&
	X[1]
	}%end xy matrix
\]

\subsubsection{Involution}

The involution $\psi$ is defined on $\cX_b$ by $\psi M(x,y)=M(x+\pi,y+\pi)=M(y,x)'$ or, equivalently, $\psi E(x,y)=E(x+\pi,y+\pi)$. Then $\psi^2$ is the identity since $M(x+2\pi,y+2\pi)=M(x,y)$. We extend to the $t^\NN$ category $\cP(\cX_b)$ by letting $\psi$ take the basic morphism $f_{XY}$ to the basic morphism $f_{\psi X,\psi Y}$ and extending $R$-linearly to all morphisms in $\cP(\cX_b)$. Then we extend $\psi$ additively to obtain an involution on all of $\cF_b=add\,\cP(\cX_b)$.

We can take the two-way approximation sequence for $SX=X\oplus \psi X$ to be the direct sum of the two-way approximation sequences for $X$ and $\psi X$. For $X=M(x,y),\psi X=M(y,x)'$, this is the direct sum of the sequences \eqref{2-way approximation sequence for Fb} and \eqref{2-way approximation sequence for Fb prime}:
\begin{equation}\label{2way approx sequence for SX}
\xymatrixrowsep{0pt}\xymatrixcolsep{40pt}
\xymatrix{%begin xy matrix
&\quad M(x,x+b)\ar[rdd]\\
&\oplus\\
M(x,y)\ar[ruu]^(.5){-1}\ar[ddr] &\quad M(x+b,x)'\ar[rdd] & \quad M(y+b,x+b)\\
\oplus &\oplus&\oplus\\
M(y,x)'\ar[ruu]^(.6){-1}\ar[ddr] &\quad M(y+b,y)\ar[ruu] & \quad M(x+b,y+b)'\\
&\oplus\\
&\quad M(y,y+b)'\ar[ruu]
	}%end xy matrix
\end{equation}
where the arrows represent basic morphisms except for the two labeled with $-1$ which are negative basic morphisms. The involution $\psi$ switches the two summands. We have the standard isomorphism
\[
	\xi=\mat{0&1\\1&0}:  \psi( X\oplus\psi X)=\psi X\oplus X\to X\oplus \psi X
\]
and this isomorphism extends to a isomorphism of two-way approximation sequences. For example, on the middle term, the isomorphism is $\mat{\xi&0\\0&\xi}$ since it switches the first two summands and it switches the last two summands. 

When $x\neq y$ we have $X\not\cong \psi X$. Then $SX=(X\oplus \psi X,\xi)$ is a regular indecomposable object of $\cF_b$ and \eqref{2way approx sequence for SX} gives a two-way approximation sequence for $SX$ which is compatible with $\xi$. This implies the following where $\cC_b^\psi$ is the stable category of $\cF_b^\psi$.

\begin{prop}
In the triangulated category $\cC_b^\psi$, the shift functor $[1]$ applied to the object $SM(x,y)$ is
\[
	(SM(x,y))[1]=(M(y+b,x+b)\oplus \psi M(y+b,x+b),\xi)=SM(y+b,x+b)
\]
In other words, $(SX)[1]=S(X[1])$. On morphisms, $[1]$ takes a morphism $f:SX\to SY$ to the morphism $f[1]:SX[1]\to SY[1]$ which has the same matrix as $f$. (See below.)
\end{prop}

Since $\cC_b^\psi(SX,SY)\cong \cC_b(X,Y\oplus \psi Y)$, morphisms $f:SX\to SY$ are given by a pair of scalars $r,s\in R$ so that $f=r f_{XY}+sf_{X,\psi Y}$ where $f_{XY}$ is the basic morphism $X\to Y$ in $\cC_b$ (or $f_{XY}=0$ if there is no nonzero morphism $X\to Y$. Then $f$ is given by the $2\times 2$ matrix:
\[
	f=\mat{r&s\\s&r}:X\oplus \psi X\to Y\oplus\psi Y
\]

\begin{cor}\label{cor: S takes triangles to triangles}
If $X\xrarrow f Y\xrarrow g Z\xrarrow h X[1]$ is a distinguished triangle in $\cC_b$ then $SX\xrarrow{Sf} SY\xrarrow{Sg} SZ\xrarrow{Sh} SX[1]$ is a distinguished triangle in $\cC_b^\psi$.
\end{cor}

\subsubsection{Singular objects}

When $x=y$, $X=M(x,x)\cong M(x,x)'$ and, by Corollary \ref{Frobenius category for cyclic groups}, $SX$ decomposes as a direct sum of two singular objects of $\cF_b^\psi$, namely, $SX\cong (X,\z)\oplus (X,-\z)$.

\begin{prop}\label{shift of singular objects in Cbpsi} In $\cC_b^\psi$, the shift of a singular object is another singular object with the opposite sign, i.e., for $\e=+$ or $-$, we have
\[
(M(x,x),\e\z)[1]\cong (M(x+b,x+b),-\e\z).
\]
%Furthermore, there are no nonzero morphisms between singular objects of opposite sign.
\end{prop}

As a special case of Definition \ref{Z2 singular objects}, we will use the notation 
\[
Z_\e(x):=(M(x,x),\e\z)
\]
for $\e=\pm$. Then the proposition says $Z_\e(x)[1]\cong Z_{-\e}(x+b)$. By definition each singular object $Z_\e(x)$ is isomorphic but not equal to $Z_\e(x)'=(M(x,x)',\e\z)=Z_\e(x+\pi)$.

\begin{proof}
Let $\e=+$ or $-$. Then we have the following commuting diagram in which each row is a standard 2-way approximation sequence in $\cF_b$ and the automorphism $\psi$ takes the top row to the bottom row.
\[
\xymatrixrowsep{40pt}\xymatrixcolsep{30pt}
\xymatrix{%begin xy matrix
M(x,x)\ar[r]^(.3){\tiny\mat{1\\-1}} &
	 M(x+b,x)\oplus M(x,x+b)\ar[r]^(.58){[1,1]} &
	M(x+b,x+b)\\
M(x,x)' \ar[u]_{\e}\ar[r]^(.3){\tiny\mat{1\\-1}}& 
	M(x,x+b)'\oplus M(x+b,x)' \ar[u]^{\tiny\mat{0&-\e\\-\e&0}}\ar[r]^(.58){[1,1]} &
	M(x+b,x+b)'\ar[u]^{-\e}
	}%end xy matrix
\]
$X=M(x,x)$ together with the isomorphism from $\psi X$ with sign $\e$ gives the singular object $(M(x,x),\e\z)$. On the right we have $X[1]=M(x+b,x+b)$ together with the vertical isomorphism from $\psi X[1]$ with sign $-\e$ giving the object $(M(x+b,x+b),-\e\z)$. Thus $(M(x,x),\e\z)[1]=(M(x+b,x+b),-\e\z)$ as claimed.
\end{proof}

We need one more observation relating the 2-way approximation sequences for $SX$ and for singular $X$.

\begin{prop}\label{prop: decomposition of SE(x,x+pi)}
For $X=M(x,x)$, the isomorphism $SX\cong (X,\z)\oplus (\psi X,-\z)$ is given by the matrix
\[
	\mat{1&\z\\-\z&1}:X\oplus \psi X\to X\oplus \psi X
\]
sending the 2-way approximation sequence for $SX=SM(x,x)$ isomorphically onto the direct sum of the 2-way approximation sequences for $(X,\z)=Z_+(x)$ and $(\psi X,-\z)=Z_-(x+\pi)\cong Z_-(x)$.
\end{prop}

%\newpage

\subsection{Compatibility}

We will define clusters in the category $\cC_b^\psi$ to be maximal sets of indecomposable objects satisfying certain conditions, the first of which is compatibility.

\begin{defn} For $0<b<\pi$, two objects $X,Y$ of $\cC_b^\psi$ are defined to be \emph{compatible} if $\Ext^1(X,Y)=0=\Ext^1(Y,X)$. 
\end{defn}

\begin{prop}
$\Ext^1(Z,SM(x,y))\neq0$ in $\cC_b^\psi$  if and only if $\Ext^1(Z,M(x,y))\neq0$ in $\cC_b$. This in turn occurs if and only if either
\begin{enumerate}
\item $Z\cong SM(z,w)\cong SM(w,z)$ where $x<z\le y+b$ and $y<w\le x+b$ or
\item $Z\cong Z_\pm(z)$, with either sign, where $\max(x,y)<z\le \min(x+b,y+b)$.
\end{enumerate}
Furthermore, $\Ext^1(Z,Z_\e(x))\neq0$ if and only if either
\begin{enumerate}
\item[(3)] $Z\cong SM(z,w)$ where $x<z,w\le x+b$ or
\item[(4)] $Z\cong Z_{-\e}(z)$ where $x<z\le x+b$.
\end{enumerate}
\end{prop}

\begin{proof}
(1) and (2) follow from the fact that, in $\cC_b$, $M(x,y)[1]\cong M(y+b,x+b)$ and $\cC_b(M(z,w),M(y+b,x+b))\neq0$ if and only $x<z\le y+b$ and $y<w\le x+b$. 

When $x=y$, this condition is $x<z,w\le x+b$. For (4), we also use Proposition \ref{shift of singular objects in Cbpsi}: There are no morphisms from $Z_\e(z)$ to $Z_\e(x)[1]\cong Z_{-\e}(x+b)$.
\end{proof}

For $b=\pi$, this is not the right definition of compatibility since, e.g., $X\cong X[1]$ for all regular objects in $\cC_\pi^\psi$. Using the idea that $\cC_\pi^\psi$ is a limit of cluster categories of type $D_n$ as $n$ goes to $\infty$, we define compatibility in $\cC_\pi^\psi$ in terms of compatibility in $\cC_b^\psi$ for $b$ arbitrarily close to $\pi$. 

First, note that any nonzero object $SM(x,y)$ or $Z_\pm(x)$ in $\cC_\pi^\psi$ is represented by an object of $\cF_\pi^\psi$ with the same name. And, for any $b$ sufficiently close to $\pi$, there will also be another nonzero object with the same name in $\cC_b^\psi$.

\begin{defn}
Two indecomposable objects of $\cC_\pi^\psi$ are defined to be \emph{compatible} if the objects in $\cC_b^\psi$ with the same name are compatible for all $b$ sufficiently close to $\pi$.
\end{defn}

To give a better description of the compatibility relation we need some notation. %First, let $S^1_\pi=\RR/\pi\ZZ$ be a circle with circumference $\pi$ instead of $2\pi$.

\begin{defn}
For each indecomposable object $Z$ of $\cC_\pi^\psi$, let
\[
	J(Z)\subseteq S^1_\pi=\RR/\pi\ZZ
\]
be the subset of $S^1_\pi$, the circle with circumference $\pi$, defined as follows.
\begin{enumerate}
\item For a regular object $Z=SM(x,y)=SE(x,y+\pi)$ with $x<y+\pi<x+2\pi$ (and $x\neq y$ since $Z$ is regular), let $J(Z)$ be the open interval in $S^1_\pi$ given by
\[
	J(Z)=\begin{cases}  (x,y+\pi)+\pi \ZZ& \text{if } y<x\\
  (y-\pi,x)+\pi\ZZ  & \text{if } y>x
    \end{cases}
\]
In both cases, the length of the interval is $\pi-|x-y|<\pi$. Note that $J(Z')=J(Z)$ by symmetry since $Z'=SE(y-\pi,x)$.  In the $SE$-notation we have
\[
	J(SE(x,y))=\begin{cases}  (x,y)+\pi\ZZ& \text{if } y<x+\pi\\
  (y,x+2\pi)+\pi\ZZ  & \text{if } y>x+\pi
    \end{cases}
\]
%if $y-x<\pi$ and $J(SE(x,y))=(y,x+2\pi)+\pi\ZZ$
Thus the $SE$ notation is more convenient if we choose $x,y$ so that $x<y<x+\pi$.
\item For the singular object $Z_\e(x)$ with either sign $\e$, let $J(Z_\e(x))$ be the point $x+\pi\ZZ\in S^1_\pi$. Note that $J(Z_\e(x)')=J(Z_\e(x))$ since $x+\pi+\pi\ZZ=x+\pi\ZZ$.
\end{enumerate}
\end{defn}

\begin{prop}
Two regular objects $X,Y$ are isomorphic if and only if $J(X)=J(Y)$. Two singular objects $Z,W$ are isomorphic if and only if $J(Z)=J(W)$ and $Z,W$ have the same sign.
\end{prop}

\begin{proof}
If $J(X)=(x,y)+\pi\ZZ$ then, by definition, either $X=SE(x,y)$ or $X=SE(y-2\pi,x)$. But these are isomorphic. The singular case is clear.
\end{proof}

\begin{prop}\label{criteria for compatibility in C-pi-psi}\label{nonC}
Two indecomposable objects $X,Y$ in $\cC_\pi^\psi$ are compatible if and only if they satisfy one of the following conditions.
\begin{enumerate}
\item $X,Y$ are regular objects and $J(X),J(Y)$ are either disjoint or one contains the other.
\item $X,Y$ are singular objects with the same sign.
\item $X,Y$ are singular objects with opposite sign and $J(X)=J(Y)$.
\item One of the objects, say $X$, is regular and the other object $Y$ is singular and $J(X),J(Y)$ are disjoint.
\end{enumerate}
\end{prop}

We say that two regular objects $X,Y$ are \emph{noncrossing} (or\emph{crossing}) if they satisfy (or don't satisfy) Condition (1), respectively.

\begin{proof}
Case (1). First, consider two regular objects $X=SM(x,y)$ and $Y=SM(z,w)$. If $X,Y$ are crossing then, by symmetry, we have either
\begin{enumerate}
\item[(a)] $y<w<x+\pi<z+\pi$ or
\item[(b)] $x<y<z<w<x+\pi$
\end{enumerate}
However, (b) implies (a). And (a) implies that $y<w\le x+b$ for all $b$ sufficiently close to $\pi$. We also have $x<z<w<x+\pi<y+\pi$. So, $x<z\le y+b$ for $b$ sufficiently close to $\pi$. Therefore, $SM(x,y)$ and $SM(z,w)$ are not compatible in $\cC_\pi^\psi$. 

Conversely, suppose that $\Ext^1(SM(z,w),SM(x,y))\neq0$ in $\cC_b^\psi$ for $b$ close to $\pi$. Then either $y<w\le x+b$ and $x<z\le y+b$ which implies $y<w<x+\pi$ and $x+\pi<z+\pi$ making $X,Y$ crossing or $y<z\le x+b$ and $x<w\le y+b$ which also imply that $X,Y$ are crossing. 

Case (2). Singular objects of the same sign are compatible in $\cC_b^\psi$ for all $b$ by the previous proposition. For two singular objects of opposite sign, such as $X=Z_+(x)$ and $Y=Z_-(y)$, We have $\Ext^1(X,Y)\neq 0$ in $\cC_b^\psi$ when $y<x\le y+b$. This will hold for $b$ arbitrarily close to $\pi$ whenever $x\neq y$. So, $Z_+(x),Z_-(y)$ are not compatible if $x\neq y$. When $x=y$, $Z_+(x),Z_-(x)$ are compatible in any $\cC_b^\psi$. This shows Case (3).

Case (4). Let $X=SM(x,y)$ and $Y=Z_\pm(z)$. Then $\Ext^1(Y,X)\neq0$ in $\cC_b^\psi$ iff $y<z\le x+b$. This holds for all $b$ sufficiently close to $\pi$ iff $y<z<x+\pi$, i.e., $z\in J(X)$. And $\Ext^1(X,Y)\neq0$ in $\cC_b^\psi$ iff $z<x,y\le z+b$. This holds for all $b$ close to $\pi$ iff $z<x,y<z+\pi$ which is equivalent to saying $z+\pi\in J(X)$. So, $X,Y$ are incompatible iff $J(Y)$ is one point in $J(X)$. They are compatible iff $J(X),J(Y)$ are disjoint.
\end{proof}

This proposition justifies the standard visualization of these objects as the geodesic on the orbifold given by modding out the action of $\ZZ/2$ on the Poincare disk by a rotation of $\pi$ around the center. We draw a standard object $X=SE(x,y)$ as the image in this orbifold of the geodesic connecting the ideal points $x,y$ on the circle at infinity. Drawn on the disk of radius $\frac12$ in the plane, this becomes an embedded path connecting boundary points $x$ and $y$ in the complement of the center point $\ast$ so that the path is homotopic to $J(X)$ fixing the endpoints and so that the homotopy avoids the center.
\begin{figure}[htbp]
\begin{center}
%
%\vs5
{
\setlength{\unitlength}{.8in}
%\centerline
{\mbox{
\begin{picture}(3,2.2)
      \thicklines
%    \thinlines
  %
    \put(1.5,1,1){
    \qbezier(1,0)(1,.4142)(.7071,.7071)
    \qbezier(.7071,.7071)(.4142,1)(0,1)
    \qbezier(-1,0)(-1,.4142)(-.7071,.7071)
    \qbezier(-.7071,.7071)(-.4142,1)(0,1)
    \qbezier(1,-0)(1,-.4142)(.7071,-.7071)
    \qbezier(.7071,-.7071)(.4142,-1)(0,-1)
    \qbezier(-1,-0)(-1,-.4142)(-.7071,-.7071)
    \qbezier(-.7071,-.7071)(-.4142,-1)(0,-1)
    }
    \put(1.5,1){$\ast$}
\put(-.02,0){     \put(2.47,1){$\bullet\ x$}
}
\put(-.04,0){     \put(.4,.7){$y\ \bullet$}
}
    \qbezier(.57,.75)(1.4,2)(2.5,1.05)
   \put(1.4,1.58){$X$}
\put(-.2,0){     \put(.6,2){$J(X)$}
}
    \qbezier(.57,.75)(1.6,.5)(2.5,1.05)
    \put(1.8,.55){$Y$}
    \put(2.2,0.2){$J(Y)$}
\end{picture}}
}}
%\vs5
\caption{If $x<y<x+\pi$ then $X=SE(x,y)$ and $Y=SE(y,x+\pi)$ have complementary intervals $J(X),J(Y)$ in $S^1_\pi$.}
\label{circle}
\end{center}
\end{figure}
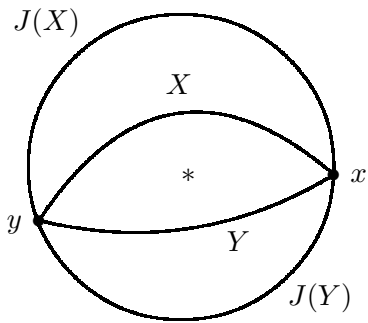
%

%\newpage

\subsection{Laminations}

We can now give a complete description of all maximal pairwise compatible sets of indecomposable objects in $\cC_\pi^\psi$. %We call these \emph{laminations}. Topologically, these correspond to maximal sets of nonintersecting geodesics in the punctured hyperbolic plane. First, we describe all possible configurations of singular objects in a lamination.

\begin{defn} A maximal set of pairwise compatible nonisomorphic indecomposable objects in $\cC_\pi^\psi$ will be called a \emph{lamination} in $\cC_\pi^\psi$. If $\cL=\{X\}$ is any lamination in $\cC_\pi^\psi$, we define $S(\cL)\subseteq S^1_\pi$ to be the union of all $J(X)$ where $X\in\cL$ is singular and we define $R(\cL)\subseteq S^1_\pi$ to be the union of all $J(X)$ where $X\in\cL$ is regular.
\end{defn}

\begin{lem}\label{R(L) cup S(L)}
For any lamination $\cL$ in $\cC_\pi^\psi$, the set $R(\cL)$ is a proper open subset of $S^1_\pi$ and $S(\cL)$ is a nonempty closed subset of the complement of $R(\cL)$ in $S^1_\pi$. Furthermore, either
\begin{enumerate}
\item $R(\cL)\cup S(\cL)=S^1_\pi$ or
\item $S^1_\pi\backslash R(\cL)$ consists of exactly two point and $S(\cL)$ is one of those two points.
\end{enumerate}
\end{lem}

\begin{proof}
It is clear that $R(\cL)$ is an open set since $J(X)\subseteq S^1_\pi$ is open for all regular $X$ and any union of open sets is open. We need to show that it is a proper open set in $S^1_\pi$.

For each point $x\in R(\cL)$ consider the collection of all open neighborhoods of $x$ in $R(\cL)$ of the form $J(X)$ where $X\in\cL$. By Proposition \ref{criteria for compatibility in C-pi-psi}, Case (1), the collection of these open sets is totally ordered by inclusion. Since $S^1_\pi$ is compact, the union of these neighborhoods cannot be all of $S^1_\pi$: otherwise a finite subset of intervals would cover the circle and, being totally ordered by inclusion, this would imply that $S^1_\pi$ was equal to one of these intervals, which is not possible. Therefore, the union of all $J(X),X\in\cL$ containing $x$ is a proper subset of $S^1_\pi$. Being connected and open, this proper subset must be an interval, say $I_x=(v,w)$, where $v<x<w\le v+\pi$. Note that each point $x\in R(\cL)$ is contained in such an interval $I_x$. Furthermore, these intervals are either disjoint or equal: if $z\in I_x\cap I_y$ then we must have $I_x\subseteq I_z$ and $I_y\subseteq I_z$ forcing $I_x=I_y=I_z$. Therefore, $R(\cL)=\coprod I_{x_i}$ is a disjoint union of such maximal open intervals $I_{x_i}$.

By Proposition \ref{criteria for compatibility in C-pi-psi}, $S(\cL)$ and $R(\cL)$ are disjoint. Furthermore, $\cL$ contains at least one singular object because, if not, we can add the object $Z_+(z)$ for any point $z\in S^1_\pi$ which is not in the set $R(\cL)$ contradicting the maximality of $\cL$. Therefore, $S(\cL)$ is nonempty. We consider two cases. Either $\cL$ contains two singular objects of opposite sign or it does not.

\ul{Case 1}. Suppose that $\cL$ contains two singular objects $X,Y$ of opposite sign. Then, by Proposition \ref{criteria for compatibility in C-pi-psi}, Case (3), $J(X)=J(Y)$ must consist of one point $z\in S^1_\pi$ and there are no other singular objects in $\cL$. In this case, either $R(\cL)$ is the complement of $z$ in $S^1_\pi$ or $S^1\backslash R(\cL)$ contains more than the point $z$. We are claiming that, in the second case, $S^1\backslash R(\cL)$ has exactly two elements. To see this suppose that $S^1\backslash R(\cL)$ contains more than two elements, say $x<y<z<x+\pi\in S^1\backslash R(\cL)$. In that case, $SE(x,z)$ would be compatible with all elements of $\cL$ and $SE(x,z)\notin\cL$ since $y\in J(SE(x,z))=(x,z)$. This contradicts the maximality of $\cL$. Therefore, $R(\cL)$ is missing at most two elements of $S^1$.

\ul{Case 2}. Suppose that all singular objects of $\cL$ have the same sign, say positive. In that case the singular object $Z_+(z)$ is compatible with all objects in $\cL$ for all $z\in S^1_\pi\backslash R(\cL)$. Therefore, $S(\cL)\cup R(\cL)=S^1_\pi$ in Case 2.

We conclude that $R(\cL)$ and $S(\cL)$ are disjoint and the only case in which $R(\cL)\cup S(\cL)$ is not equal to all of $S^1_\pi$ is when $\cL$ contains exactly two singular objects of opposite sign at the same point $z\in S^1_\pi$ and the complement of $R(\cL)$ contains $z$ and one more point.
\end{proof}

In the proof of the lemma above, we examined the connected components of the open set $R(\cL)$. We extract the conclusions together with some additional observations in a separate lemma.

\begin{lem}
For any lamination $\cL$ of $\cC_\pi^\psi$, the open set $R(\cL)$ is a (possibly empty) disjoint union of open intervals $R(\cL)=\coprod I_\a$ where each $I_\a$ is described as follows.
\begin{enumerate}
\item Each interval $I_\a\subset S^1_\pi$ is represented by an interval $(x,y)\subset\RR$ where $x<y\le x+\pi$. 
\item If $I_\a=(x,y)$ has length $y-x<\pi$ then the lamination $\cL$ contains an object isomorphic to $SE(x,y)$.
\item For any point $z\in I_\a$, the set $I_\a$ is equal to $I_z$, the union of all $J(X)$ for which $z\in J(X)$ and $X$ is a regular object in $\cL$.
\end{enumerate}
\end{lem}

\begin{proof} We showed in the proof of the previous lemma that the intervals $I_z$ described in (3) are either disjoint or equal and that their union is equal to $\cR(\cL)$. Therefore $R(\cL)$ is a disjoint union of intervals satisfying (3).

Statement (1) is obvious since this describes all possible connected proper open subsets of $S^1_\pi=\RR/\pi\ZZ$. Statement (2) is also clear since, if $x+\pi<y$, then $SE(x,y)$ is a regular object of $\cC_\pi^\psi$ which is compatible with every object in $\cL$ and thus belongs to $\cL$ by maximality.
\end{proof}

The next thing we need to show is that any interval $I=(x,y)$ with $x<y\le x+\pi$ can occur as one of the components of $R(\cL)$. We do this by constructing a ``sublamination'' for each such interval.

\begin{defn} An \emph{interval} $I\subset S^1_\pi$ is any connected open proper subset. (Equivalently, $I=(x,y)$ for some $x<y\le x+\pi$.) A \emph{sublamination supported by an interval} $I$ is a maximal collection of pairwise compatible nonisomorphic indecomposable regular objects $X$ in $\cC_\pi^\psi$ so that $J(X)$ is contained in $I$.
\end{defn}

\begin{lem}\label{all intervals support sublaminations}
Any interval $I\subset S^1_\pi$ supports a sublamination. Furthermore, this sublamination contains an object $X$ with $J(X)=I$ if and only if $I$ has length less than $\pi$.
\end{lem}

\begin{proof}
If $I=(x,y)$ with $x<y<x+\pi$ then one example of a sublamination supported by $I$ is the set of all regular object $SE(x,z)$ where $x<z\le y$. Since $J(SE(x,z))=(x,z)$ and $(x,z)\subseteq (x,z')$ for all $z\le z'$, this is a compatible set of objects. It is clearly maximal since any interval $(a,b)$ contained in $(x,y)$ will contain some point $z$ and therefore, by compatibility, must contain $(x,z)$ making $a\equiv x$ modulo $\pi\ZZ$. Furthermore, any sublamination supported by $I$ will contain an object isomorphic to $SE(x,y)$ since this object is compatible with all objects with support in $I$.

In the case $I=(x,x+\pi)$, one example of a sublamination is given by taking all $SE(x,z)$ where $x<z<x+\pi$. The extremal case $z=x+\pi$ must be excluded since $SE(x,x+\pi)$ is not a regular object.
\end{proof}

%This Lemma and its proof prove the following theorem describing the set of singular objects in any lamination.

%Recall that the topological definition of a lamination is a closed set of nonintersecting geodesics.

\begin{thm}\label{description of all laminations}
All laminations $\cL$ in $\cC_\pi^\psi$ are described as follows. First, the set $U=R(\cL)$ can be any (possibly empty) proper open subset of $S_\pi^1$. Next, given $U=R(\cL)=\coprod I_\a$, the objects of $\cL$ supported in each $I_\a$ is a sublamination which can be chosen arbitrarily and independently for different $I_\a$ and the set of singular objects of $\cL$ is as follows.

Case 1: Suppose the complement of $U$ in $S^1_\pi$ consists of a single point $x$. Then the lamination contains two singular objects at $x$ of opposite sign and no other singular objects. We call this a \emph{singular pair} at $x$.

Case 2: When the complement of $U$ consists of exactly two points, say $x,y$, the lamination $\cL$ contains exactly two singular objects which are either (a) of the same sign at each of the points $x,y$ or (b) a singular pair at one of these two points.

Case 3: When the complement of $U$ in $S^1$ contains at least three points, $\cL$ contains exactly one singular object at each of these points and all of them have the same sign.
\end{thm}

\begin{proof}
Suppose that $U$ is any proper open subset of $S^1_\pi$. Then $U$ is a disjoint union of intervals $U=\coprod I_\a$. By the previous lemma, each interval $I_\a$ supports a sublamination $\cL_\a$. By Proposition \ref{criteria for compatibility in C-pi-psi}, the objects in different $\cL_\a$ are compatible.

Lemma \ref{R(L) cup S(L)}, together with the fact that compatible singular objects supported at different points must have the same sign, establishes the necessity of the conditions imposed on the singular objects and listed in Cases 1,2,3. Conversely, they are also sufficient since singular objects at all points in $S^1_\pi\backslash U$ will be compatible with all objects of all $\cL_\a$ and with each other if they are given as described in Cases 1,2,3. Maximality of this compatible set is clear in Cases 1 and 3 since $S(\cL)\cup R(\cL)=S^1_\pi$ in both cases. 

In Case 2, $U=R(\cL)$ is a union of two disjoint intervals $U=I_\a\coprod I_\b$ with endpoints $x,y$ and $S(\cL)$ is one of these points, say $x$. In that case, maximality is established as follows. Since we have singular objects at $x$ with both signs, no singular objects supported at any other point in $S^1_\pi$ will be compatible with these. Also, by Lemma \ref{all intervals support sublaminations}, $\cL_\a,\cL_\b$ will contain objects $X_\a,X_\b$ with $J(X_\a)=I_\a$ and $J(X_\b)=I_\b$, and any regular object compatible with both $X_\a$ and $X_\b$ must have support in either $I_\a$ or $I_\b$. So, $\cL_\a\cup \cL_\b$ together with a singular pair at $x$ or $y$ is a lamination of $\cC_\pi^\psi$ as claimed.
\end{proof}

%\newpage

\subsection{Topology of laminations}

Let $\cH$ be the (nonHausdorff!) topological space defined as a quotient space $\cH=\tilde\cH/\sim$ with the quotient topology where
\[
	\tilde\cH=\{(x,y,\e)\in \RR^2\times\{+,-\}\,|\, x< y\le x+\pi\}
\]
modulo the following relations.
\begin{enumerate}
\item $(x,y,+)\sim(x,y,-)$ if $y<x+\pi$.
\item $(x,y,\e)\sim (x+\pi,y+\pi,\e)$ for any $x,y,\e$.
\end{enumerate}
Let $[x,y]_\e\in\cH$ denote the equivalence class of $(x,y,\e)$. Because of Relation (1) we usually drop the $\e$ when $y<x+\pi$.

\begin{defn}
For any indecomposable object $X$ in $\cC_\pi^\psi$ let $h(X)$ denote the element of $\cH$ given as follows.
\begin{enumerate}
\item $h(SE(x,y))=[x,y]$ if $y<x+\pi$.
\item $h(SE(x,y))=[y,x+2\pi]$ if $y>x+\pi$ (so, $x+2\pi<y+\pi$).
\item $h(Z_\e(x))=[x,x+\pi]_\e$.% for any sign $\e$.
\end{enumerate}
\end{defn}
Thus $h$ gives a 1-1 correspondence between elements of $\cH$ and isomorphism classes of indecomposable objects of $\cC^\psi_\pi$.

\begin{lem}\label{lem: objects compatible with X}
For any indecomposable object $X$ in $\cC_\pi^\psi$, the set $C(X)$ of all points in $\cH$ of the form $h(Y)$ for some indecomposable object $Y$ compatible with $X$ is closed.
\end{lem}

\begin{proof}
If $X=SE(x,y)$ where $x<y<x+\pi$ then the objects compatible with $X$ are isomorphic to objects in the following list. (See Figure \ref{fig:compatible objects}.)
\begin{enumerate}
\item[(a)] $Y=SE(a,b)$ where $x\le a<b\le y$ (i.e. $J(Y)\subseteq J(X)$),
\item[(b)] $Y=SE(z,w)$ where $z\le x<y\le w<z+\pi$ (i.e. $J(X)\subseteq J(Y)$),
\item[(c)] $Y=SE(c,d)$ where $y\le c<d\le x+\pi$ (i.e. $J(Y)\cap J(X)=\emptyset$),
\item[(d)] $Z_\e(z)$ for either sign $\e$ where $z\le x<y\le z+\pi$ (i.e. $J(Z_\e(z))=z\notin J(X)$).
\end{enumerate}
Since $(a),(c),(d)$ are closed conditions, the set of all $h(Y)$ with $Y$ in (a),(c),(d) is a union of three closed subsets of $\cH$. Condition (b) is not closed, but the point $h(Y)=[z,w]$ for $Y$ in (b) converge to the points $h(Z_\e(z))=[z,z+\pi]_\e$ in (d). So, the union of these four sets is closed in $\cH$.

The objects compatible with $X=Z_\e(z)$ are:
\begin{enumerate}
\item $SE(x,y)$ where $z\le x<y\le z+\pi$ but $(x,y)\neq(z,z+\pi)$.
\item $Y=Z_\e(x)$ for any $x\in S^1_\pi$ where $Y$ has the same sign as $X$.
\item $Y=Z_{-\e}(z)$ with sign opposite that of $X$.
\end{enumerate}
The set of all $h(Y)$ with $Y$ given in Cases (2) and (3) are closed sets and the points $[x,y]$ from Case (1) converge only to the two points $[z,z+\pi]_\pm$ which lie in Cases (2) and (3). Thus their union $C(X)$ is a closed subset of $\cH$.
\end{proof}

\begin{figure}[htbp] %{fig:maps to compatible}
\begin{center}
%
%\vs5
{
\setlength{\unitlength}{.25in}
%\centerline
{\mbox{
\begin{picture}(16,8)
%      \thicklines
    \thinlines
\put(-2,1){\line(1,1)8}
\put(2,0){\line(1,1)7}
\put(0,1){\line(1,0)3}
\put(5,8){\line(1,0)3}
\put(0,3){\line(1,0)5}
\put(3,6){\line(1,0)5}
\put(0,1){\line(0,1)2}
\put(8,6){\line(0,1)2}
\put(3,1){\line(0,1)5}
\put(5,3){\line(0,1)5}
\thicklines{\color{black}
\put(0,0){
\put(3,3){\line(1,0)2}
\put(3,3){\line(0,1)3}
\put(5,6){\line(1,0)3}
}} % end color blue
\put(7.85,7.85){$\bullet$}
\put(3.2,3.2){$X$}
\put(0,3){\line(1,1)3}
\put(8.2,8.2){$X'$}
\put(3.2,2.3){$(a)$}
\put(1.7,3.8){$(b)$}
\put(0.8,4.8){$(d)$}
\put(5.7,5){$(c)$}
%\put(1,1.8){$(e)$}
%\put(6,6.8){$(e)$}
%\put(4,6.4){$(g)$}
%\put(3.7,4.5){$(f)$}
\put(.1,-.3){
\put(3,1){$x$}
\put(5,3){$y$}   
\put(8,6){$x+\pi$}   
}
\put(-1.2,.2){
\put(2.5,6){$Z_\pm(x)$}
\put(-1.5,3.1){$Z_\pm(y-\pi)$}
}
{\color{black}
\put(3,3){\line(-1,0)3}
\put(3,3){\line(0,-1)2}
\put(5,6){\line(0,-1)3}
}
\put(2.85,2.85){$\bullet$}
\put(7,1){
\put(4,3){\line(1,1){2.5}}
\put(5,0){\line(1,1){5.5}}
\put(5.5,0.5){\line(0,1)4}
\put(5.5,4.5){\line(1,0)4}
\put(5.35,4.35){$\bullet$}
\put(4,4.8){$Z_\e(z)$}
\put(6.5,3){$(1)$}
\put(5.8,5.8){$(2)$}
}
\end{picture}}
}}
%\vs5
\caption{Objects compatible with $X=SE(x,y)\cong X'$ lie in Regions $(a),(b),(c)\subseteq \cH$ and $(d)$ which marks the hypotenuse of triangle $(b)$. Objects compatible with $Z_\e(z)$ lie in Region (1) or on the diagonal line (2).}
\label{fig:compatible objects}
\end{center}
\end{figure}
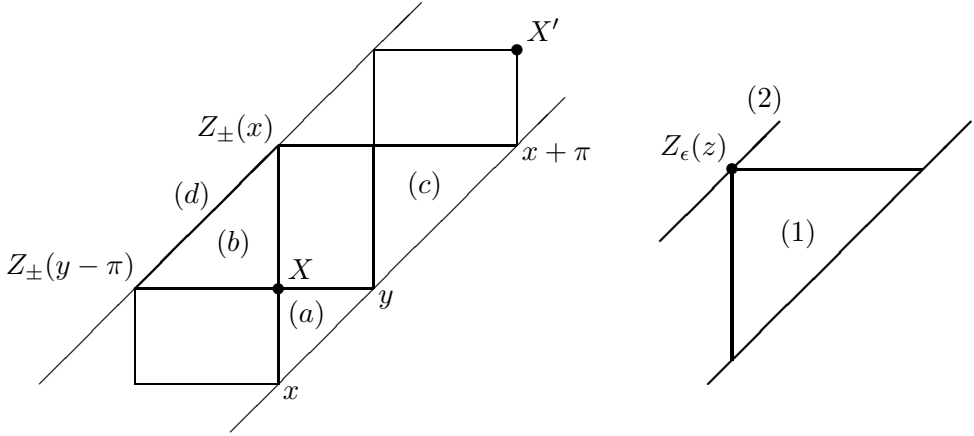
%

%%%%%%%%%%%%%%%%%%%%
%%%%%%%%  xxx123
%\newpage

\begin{lem}\label{sublaminations are closed}
Given any interval $I=(x,y)$ in $S^1_\pi$ of length $y-x<\pi$ and any sublamination $\cL_0$ of $I$, the set $h(\cL_0)$ of all $h(X)$ for $X\in\cL_0$ is a closed subset of $\cH$.
\end{lem}

\begin{proof}
Let $C=\bigcap_{X\in\cL_0} C(X)$. This is a closed subset of $\cH$ since it is an intersection of closed subsets of $\cH$. Let $D$ be the set of all $[a,b]\in\cH$ so that $x\le a<b\le y$. Then $D$ is also a closed subset of $\cH$. We claim that $h(\cL_0)=C\cap D$. The lemma follows.

Since the objects of $\cL_0$ are pairwise compatible we have $h(\cL_0)\subseteq C$. Since $\cL_0$ is a sublamination of $I=(x,y)$, we have $h(\cL_0)\subseteq D$. So, $h(\cL_0)\subseteq C\cap D$. Conversely, suppose that $[a,b]\in C\cap D$. Then $X\cong SE(a,b)$ is compatible with all objects in $\cL_0$. Being in $D$ this object has support in $I$. Therefore, $X$ lies in $\cL_0$ (up to isomorphism). So, $h(\cL_0)=C\cap D$ is a closed subset of $\cH$.
\end{proof}

\begin{thm}
For any lamination $\cL$ in $\cC_\pi^\psi$, the set $h(\cL)$ of all $h(X)$ with $X\in\cL$ is a closed subset of $\cH$.
\end{thm}

\begin{proof}
Let $C=\bigcap_{X\in\cL} C(X)$. This is a closed subset of $\cH$ since it is an intersection of closed subsets of $\cH$. Since the elements of $\cL$ are pairwise compatible, $h(\cL)\subseteq C$. Also, $C\subseteq h(\cL)$ since any point in $C$ is compatible with all $X\in\cL$ and is therefore isomorphic to some object of $\cL$. Therefore $h(\cL)=C$ is closed.
\end{proof}

%\newpage

\subsection{Clusters in $\cC_\pi^\psi$}

We define a \emph{discrete lamination} in $\cC_\pi^\psi$ to be a lamination $\cL$ for which the set $h(\cL)$ is a discrete subset of $\cH$ (in the subspace topology). This is a necessary condition for the mutation process since, by the theorem above, no limit point of a lamination can be mutated. For any interval $I$ in $S^1_\pi$, we define a \emph{discrete sublamination} of $I$ to be a sublamination $\cL_0$ of $I$ with the property that $h(\cL_0)$ is a discrete subset of $\cH$. After proving Theorem \ref{discrete laminations are clusters} below, we will refer to a discrete lamination in $\cC_\pi^\psi$ as a \emph{cluster}.

Let $\el:\cH\to (0,\pi]$ be the continuous function given by $\el[x,y]_\e=y-x$ for any $\e$. For any indecomposable object $X$ in $\cC_\pi^\psi$ let $\el(X)=\el(h(X))$. Then $\el(X)=\pi$ if and only if $X$ is singular and $\el(X)$ is the length of the interval $J(X)$ when $X$ is regular.

\begin{lem}\label{finite number of X with el X ge delta}
For any discrete lamination $\cL$ and any $\delta>0$ there are only finitely many elements $X$ in $\cL$ with $\el(X)\ge\delta$. In particular (when $\delta=\pi$) $\cL$ contains only finitely many singular objects. The same holds for any discrete sublamination of any interval.
\end{lem}

\begin{proof}
The set of points $[x,y]\in\cH$ with $\el[x,y]=y-x\ge\delta $ is compact and therefore its intersection with any closed discrete subset is finite.
\end{proof}

\begin{lem}\label{lem: discrete sublaminations}
Any interval $I=(x,y)$ with length $y-x<\pi$ admits a discrete sublamination. However, intervals of length equal to $\pi$ have no discrete sublaminations.
\end{lem}

\begin{proof}
Suppose that $I$ is an interval in $S^1_\pi$ of length $\pi$. Then any lamination $\cL_0$ of $I$ will have a sequence of points $X_i$ with $J(X_i)\subset I$ being intervals of length $\el(X_i)<\pi$ converging to $\pi$. By the above lemma, this is not possible if $\cL_0$ is discrete. Therefore an interval of length $\pi$ cannot have a discrete sublamination.

If $I=(x,y)$ with $x<y<x+\pi$ then a discrete sublamination for $I$ is given by $\cL_0=\{X_{k,n}\,|\,n,k\in\ZZ,n\ge0, 0<k\le 2^n \}$ where $X_{k,n}$ are regular objects with
\[
	J(X_{k,n})=\left(x+
	\frac{k-1}{2^n}(y-x),x+\frac{k}{2^n}(y-x)
	\right)
\]
These are the intervals given by taking the interval $I=(x,y)$ and cutting it up into $2^n$ disjoint subintervals of equal length. Clearly all such subintervals are noncrossing and thus the $X_{k,n}$ are compatible for all $n\ge0$. And it is easy to see that it is maximal.
\end{proof}

\begin{lem}\label{decomposition of discrete sublaminations}
Given an interval $I=(x,y)$ in $S^1_\pi$ of length $y-x<\pi$, and a discrete sublamination $\cL_0$ of $I$, there exists a unique $z\in I$ so that $\cL_0$ is the union of discrete sublaminations of $I_1=(x,z)$ and $I_2=(z,y)$ and an object isomorphic to $SE(x,y)$. Conversely, for any $z\in I$ and any discrete sublaminations $\cL_1,\cL_2$ of $(x,z)$ and $(z,y)$, the union of $\cL_1,\cL_2$ and $SE(x,y)$ is a discrete sublamination of $I=(x,y)$.
\end{lem}

\begin{proof} The discrete sublamination $\cL_0$ of $I$ must contain an object $X_0\cong SE(x,y)$ with $\el(X_0)=y-x$. By Lemma \ref{finite number of X with el X ge delta}, there is an object $X_1$ in $\cL_0$ so that $\el(X_1)$ is maximal among all objects of $\cL_0$ not equal to $X_0$. Equivalently, $J(X_1)$ is maximal.

\ul{Claim}: $J(X_1)$ is equal to either $(x,z)$ or $(z,y)$ for some $x<z<y$.

Pf: Suppose not. Then $J(X_1)=(a,b)$ where $x<a<b<y$. Since $J(X_1)$ is maximal, $X_0$ is the only element of $\cL_0$ with $J(X_0)$ containing $J(X_1)=(a,b)$. Thus, for any other element $Y$ in $\cL_0$, $J(Y)$ is disjoint from $(a,b)$ and therefore contained in either $(x,a)$ or $(b,y)$. This implies that $Z=SE(x,b)$ is compatible with all objects of $\cL_0$. So, $Z\in\cL_0$ with $J(Z)\supset J(X_1)$ contradicting the maximality of $J( X_1)$.

Since $X_1$ is either $SE(x,z)$ or $SE(z,y)$, we may assume by symmetry that $X_1\cong SE(x,z)$. Then for all other $Y$ in $\cL_0$, either $J(Y)$ is contained in $(x,z)$ or it is disjoint from $(x,z)$ which means $J(Y)\subseteq (z,y)$. Thus $\cL_0$ minus $X_0$ consists of discrete sublaminations of $I_1=(x,z)$ and $I_2=(z,y)$. Uniqueness of $z$ is clear.

Conversely, given such sublaminations $\cL_1,\cL_2$ of $I_1,I_2$, the elements of $\cL_1$ and $\cL_2$ are compatible and the only other object compatible with both $\cL_1$ and $\cL_2$ with support in $(x,y)$ is $SE(x,y)$ up to isomorphism.
\end{proof}

\begin{thm}\label{description of all discrete laminations}
Let $\cL$ be a discrete lamination of $\cC_\pi^\psi$. Then there are $k\ge2$ points $x_1<x_2<\cdots<x_k<x_1+\pi$ in $S^1_\pi$ so that
\begin{enumerate}
\item $R(\cL)=S^1\backslash\{x_1,x_2,\cdots,x_k\}$ is a disjoint union of $k$ intervals.
\item $\cL$ has exactly $k$ singular objects and these objects are either 
\begin{enumerate}
\item singular objects of the same sign, supported at the points $x_1,\cdots,x_k$ or 
\item a singular pair supported at either $x_1$ or $x_2$ (when $k=2$).
\end{enumerate}
\end{enumerate}
Conversely, for any set $S$ of $k\ge2$ points in $S^1_\pi$, the union of any set of singular objects satisfying (2a) or (2b) and any choice of discrete sublaminations for the $k$ intervals making up $S^1_\pi\backslash S$ will give a discrete lamination for $\cC_\pi^\psi$.
\end{thm}

\begin{proof} By Lemma \ref{finite number of X with el X ge delta}, $\cL$ has only finitely may singular objects. The rest follows from the description of all laminations given in Theorem \ref{description of all laminations} with the observation that Case 1 in Theorem \ref{description of all laminations} cannot occur for discrete laminations by Lemma \ref{lem: discrete sublaminations}.
\end{proof}

\begin{thm}\label{discrete laminations are clusters}\label{DL123}
Given any discrete lamination $\cL$ of $\cC_\pi^\psi$ and any object $T$ in $\cL$, there is, up to isomorphism, exactly one other object $T^\ast$ with the property that $\cL^\ast=\cL\backslash \{T\}\cup \{T^\ast\}$ is also a discrete lamination.
\end{thm}

\begin{proof}
$T$ is either regular or singular. 

\ul{Case 1}. Suppose $T$ is regular. There are three subcases. 

(a) $J(T)$ is properly contained in one of the components $I$ of $R(\cL)$.

(b) $J(T)=I_1$ is one of the $k\ge2$ components of the open set $R(\cL)$ and the $k$ singular objects of $\cL$ have the same sign ($\cL$ has no singular pairs).

(c) $J(T)=I_1$ is one of $2$ components of the open set $R(\cL)$ and the two singular objects of $\cL$ have opposite sign (forming a singular pair).

\ul{Case 1(a)}. Suppose $J(T)$ is properly contained in one of the components $I$ of $R(\cL)$. Let $X,Y$ be objects in $\cL$ so that $J(X)\subset J(T)\subset J(Y)\subseteq I$ and so that $J(Y)=(w,y)$ is minimal and $J(X)$ is maximal. Then $J(T)$ is maximal among subintervals of $J(Y)$. By the proof of Lemma \ref{decomposition of discrete sublaminations}, either $J(T)=(w,x)$ or $J(T)=(x,y)$ for some $x\in (w,y)$. Suppose, for example, that $J(T)=(x,y)$. We also have either $J(X)=(x,z)$ or $J(X)=(z,y)$ for some $x<z<y$. Then $T^\ast\cong SE(w,z)$.

\ul{Case 1(b,c)}. Suppose that $J(T)=I_1=(x,y)$ is one of the $k\ge2$ components of $R(\cL)$ and $I_2,\cdots,I_k$ are the other components. Then $\cL=\cL_0\coprod \cL_1\coprod\cdots\coprod \cL_k$ where $\cL_i$ is the set of all $Y$ in $\cL$ so that $J(Y)\subseteq I_i$ for $i=1,\cdots,k$ and $\cL_0$ is the set of all singular objects in $\cL$. Then $\cL_i$ is a discrete sublamination of the interval $I_i$ for $i\ge1$. Therefore, by Lemma \ref{decomposition of discrete sublaminations}, there is a unique $z\in I_1$ so that $\cL_1$ minus $T$ is the union of discrete sublaminations $\cL_\a,\cL_\b$ of $I_\a=(x,z)$ and $I_\b=(z,y)$. 

In Case 1(b), when the singular objects of $\cL_0$ have the same sign $\e$ then they include $Z_\e(x),Z_\e(y)$ and we must have $T^\ast\cong Z_\e(z)$.

If we are in {Case 1(c)}, where $k=2$, $I_2=(y,x+\pi)$ and the $2$ singular objects of $\cL$ have opposite sign then, by symmetry, we may assume that this pair of singular objects has support at the point $x$. Then $T^\ast$ will be isomorphic to $SE(z,x+\pi)$.

\ul{Case 2}. Suppose $T=Z_\e(z)$ is singular. There are three subcases. The subcases when $k=2$ are easy to analyze since $\cL\backslash T$ has only one singular object. Since every cluster has at least two singular objects, $T^\ast$ must be a singular object and there are only two possibilities.

\ul{Case 2(a)}. $k=2$ and the unique singular object of $\cL\backslash T$ is $Z_{-\e}(z)$. Then there is one other point $y$ in $S^1_\pi\backslash R(\cL)$. In this case $T^\ast$ must be $Z_{-\e}(y)$.

\ul{Case 2(b)}. $k=2$ and the other singular object of $\cL$ is $Z_\e(y)$. Then $T^\ast=Z_{-\e}(y)$.

\ul{Case 2(c)}. $k\ge 3$. In this case let $I_1=(x,z)$, $I_2=(z,y)$ be the two components of $R(\cL)$ with $z$ as endpoints. Then $T^\ast$ must be isomorphic to $SE(x,y)$ as described in Case 1(b).
\end{proof}

\begin{rem}
Although there are six cases, there are only four different mutation pairs: $1(a)\leftrightarrow 1(a)$, $1(c)\leftrightarrow 1(c)$, $1(b)\leftrightarrow 2(c)$ and $2(a)\leftrightarrow 2(b)$. The last case $2(a)\leftrightarrow 2(b)$ is different from the other three and we will refer to it as the \emph{exceptional case}.
\end{rem}

%\newpage

%%%%%%%%%%%%%%%%%
\section{Cluster structure}

We will show that cluster mutation follows the general pattern described in \cite{BIRSc} using the triangulated structure of the cluster category $\cC_\pi^\psi$, namely:

\begin{thm}\label{thm:cluster mutation given by approximation}
Suppose that $\cT$ is any cluster in $\cC_\pi^\psi$ and $T$ is any object in $\cT$. Then the mutation $T^\ast$ of $T$ fits into two distinguished triangles:
\[
	T\to X\to T^\ast\to T[1]
\quad\text{and}\quad
	T[-1]\to T^\ast\to Y\to T
\]
where $T\to X$ is a minimal left $add(\cT\backslash T)$-approximation of $T$ and $Y\to T$ is a minimal right $add(\cT\backslash T)$-approximation of $T$. Furthermore, $X,Y$ are indecomposable if and only if $\cT$ has exactly two singular objects and $T$ is one of those objects.
\end{thm}

We also show that, outside the exceptional mutation case $2(a)\leftrightarrow 2(b)$ where $\cT$ has exactly two singular objects of which $T$, the object being mutated, is one, the Octahedral Axiom is used, as in \cite{IT10}. An examination of the statement of the Octahedral Axiom shows that cluster mutation can be given by this axiom only in the case when $T^\ast\cong T^\ast[1]$. Indeed we have the following familiar diagrams:
\begin{equation}\label{Octahedron}
\xymatrixrowsep{10pt}\xymatrixcolsep{10pt}
\xymatrix{
X\ar[dr]_f && A\ar[ll]_{h'c} &&&& X\ar[dd]_{bf} && A\ar[dl]_c\\
	& T\ar[ur]_a\ar[dl]_b &&&\longrightarrow&&& T^\ast\ar[ul]_{h'}\ar[dr]_g\\
	B\ar[rr]^{gd} && Y\ar[ul]_h&&&& B\ar[ur]_d&&Y\ar[uu]_{ah}
	}%end xymatrix
\end{equation}
where the arrows $h,h',ah,h'c$ are not morphisms but extensions. E.g., $h$ and $h'$ are actually morphisms $h:Y\to T[1]$ and $h':T^\ast\to X[1]$. The 3-cycles are distinguished triangles. As part of the Octahedral Axiom \cite{Neeman}, we have distinguished triangles 
\begin{equation}\label{approximations from Octahedron}
	T\to A\oplus B\to T^\ast\to  T[1]
\quad\text{and}\quad
%\] \[
	T^\ast \xrarrow{\tiny\mat{h'\\-g}} X[1]\oplus Y\xrarrow{[f[1],h]} T[1]\to  T^\ast[1].
\end{equation}
In order for these to be the approximation triangles described in Theorem \ref{thm:cluster mutation given by approximation} above, we must have either $T\cong T[1]$ or $T^\ast\cong T^\ast[1]$. In the exceptional mutation case, both $T$ and $T^\ast$ are singular and thus not isomorphism to their shifts. This is another reason why the Octahedral Axiom does not directly induce cluster mutation in the exceptional case. However, it is indirectly involved as we shall see.

\begin{thm}\label{thm:cluster mutation given by Octahedral Axiom}
Suppose that $\cT$ is a cluster in $\ul\cF_\pi^\psi$ and $T$ is an object in $\cT$ which is not of the exceptional kind. Then the mutation $T^\ast$ of $T$ is given by the Octahedral Axiom \eqref{Octahedron}. Furthermore, the distinguished triangles \eqref{approximations from Octahedron} which come as part of the Octahedral Axiom gives the approximation sequences described in Theorem \ref{thm:cluster mutation given by approximation} above. In keeping with this, we also have that either $T\cong T[1]$ and $X\cong X[1]$ or $T^\ast\cong T^\ast[1]$ and $Y\cong Y[1]$.
\end{thm}

%(The $2a\leftrightarrow 2b$ mutation is the one which does not use the Octahedral Axiom.) We comment briefly on the well-known case of type $D_n$ which, in our setting, is a cluster in $\cC_b^\psi$ where $b=(n-1)\pi/n$.

\subsection{Morphisms between compatible objects} To find the approximation $T\to A\oplus B$, we list all possible maps from $T$ to a compatible object.

\begin{lem}\label{lem: maps between objects in C}
Let $T=SE(x,y)$ with $x<y<x+\pi$. Then $\cC_\pi^\psi(T,Y)\neq0$ if and only if one of the following holds.
\begin{enumerate}
\item $Y\cong SE(z,w)$ where $x\le z<y$  and $y\le w<z+\pi$,
\item $Y\cong SE(z,w)$ where $x\le w-\pi< z<x+\pi\le w<y+\pi$ or
\item $Y\cong Z_\pm(z)$ where $x\le z<y$.
\end{enumerate}
Furthermore, $\cC_\pi^\psi(T,Y)$ is two-dimensional if and only if $Y$ lies in both (1) and (2).
\end{lem}

\begin{proof} This is because $\cC_\pi^\psi(T,Y)\cong \cC_\pi(E(x,y),Y)=\cC_\pi(E(x,y),E(z,w))\oplus \cC_\pi(E(x,y),E(w-\pi,z+\pi))$. The first summand is nonzero in Case (1) and the second summand is nonzero in Case (2). Case (3) is ``half'' the intersection when $w=z+\pi$.
This is also evident from Figure \ref{fig:maps to compatible}. The union of Regions (f), (g) in Figure \ref{fig:maps to compatible} corresponds to the union of Cases (1) and (3) in the Lemma and the union of Regions (g), (e) in Figure \ref{fig:maps to compatible} correspond to the union of Cases (2) and (3) in the Lemma. The key point is that morphisms go right and up in the Figure, the objects on the bottom diagonal are all zero and the top diagonal is a ``reflecting wall'', so that $T$ maps to the points in Region (d) by bouncing off of this wall.
\end{proof}

\begin{lem}\label{lem: maps to compatible objects}\label{LtoC}
Let $T=SE(x,y)$ be a regular object with $x<y<x+\pi$. Then the nonzero morphisms from $T$ to other compatible indecomposable objects in $\cC_\pi^\psi$ are given as follows.
\begin{enumerate}
\item[(a)] $T\to A_z:=SE(z,y)$ where $x<z<y$
\item[(b1)] $T\to B_{w}:=SE(x,w)$ where $y<w<x+\pi$
\item[(b2)] $T\to B_{x+\pi}^\pm:=Z_{\pm}(x)$ (only two object in this case).
\item[(b3)] $T\to B_{s+\pi}:=SE(s,x+\pi)$ where $y\le s<x+\pi$
\end{enumerate}
Furthermore, 
\begin{enumerate}
\item $\cC_\pi^\psi(T,X)$ is one dimensional for $X=A_z,B_\ast$ in all the above cases.
\item There are no nonzero morphisms $A_z\to B_\ast$ for any $B_\ast$ in the above cases.
\item $\cC_\pi^\psi(A_z,A_w)\neq0$ if and only if $z\le w$. 
\item There is a nonzero morphism $B_i\to B_j$ under $T$ if and only if $i\le j$. 
\item Composition of any two nonzero morphisms between indecomposable compatible objects under $T$ is nonzero.
\end{enumerate}
\end{lem}

By a \emph{nonzero morphism under $T$} we mean a morphism $B_i\to B_j$ which induces a nonzero morphism of hom sets: $\cC_\pi^\psi(T,B_i)\to \cC_\pi^\psi(T,B_j)$. There are nonzero morphisms $B_{s+\pi}\to B_w$ from objects in (b3) to objects in (b1) but any composition $T\to B_{s+\pi}\to B_w$ is zero.

\begin{proof}
The objects compatible with $T=SE(x,y)$ lie in the regions (a), (b), (c) of the set $\cH$ as shown in Figure \ref{fig:maps to compatible}. (See the proof of Lemma \ref{lem: objects compatible with X} for an explanation.) $\cC_\pi^\psi(T,X)$ is 2-dimensional only for $X$ in Region (g) but these are incompatible with $T$.
\begin{enumerate}
\item[(a)] These are the objects $SE(z,w)$ where $x\le z<w\le y$. But $T$ maps only to those points where $w=y$ since morphisms only go right and up in the figure.
\item[(b)] These are $SE(z,w)$ and $Z_\pm(s)$ where $y\le z< w\le x+\pi$ and $y\le s\le x+\pi$. But $T$ maps only to those points where $z=x$ (and $s=x$) since morphisms cannot go left in the figure.
\item[(c)] $SE(s,t)$ where $y\le s<t\le x+\pi$. But $T$ maps only to those where $t=x+\pi$ because the map from $T$ to the other points maps through the zero object at point $y$ in the figure. The nonzero morphisms $T\to SE(s,x+\pi)$ is ``reflected'' through $Z_\pm(x)$.
\end{enumerate}
The other statements follow from the fact that morphisms go right and up in the diagram and are reflected off of the upper diagonal line. (The lower diagonal line is an ``absorbing'' wall.)
\end{proof}

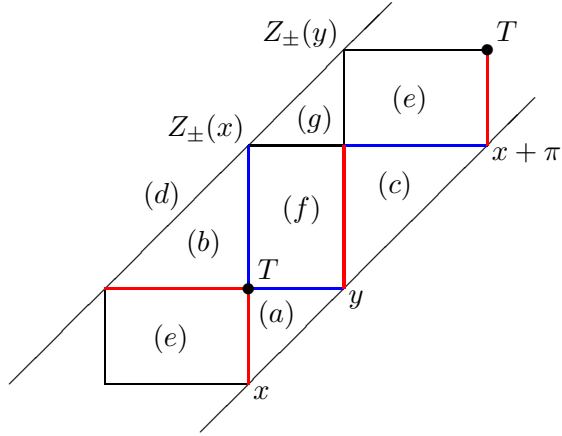
\begin{figure}[htbp] %{fig:maps to compatible}
\begin{center}
%
%\vs5
{
\setlength{\unitlength}{.25in}
%\centerline
{\mbox{
\begin{picture}(8,8)
%      \thicklines
    \thinlines
\put(-2,1){\line(1,1)8}
\put(2,0){\line(1,1)7}
\put(0,1){\line(1,0)3}
\put(5,8){\line(1,0)3}
\put(0,3){\line(1,0)5}
\put(3,6){\line(1,0)5}
\put(0,1){\line(0,1)2}
\put(3,1){\line(0,1)5}
\put(5,3){\line(0,1)5}
\thicklines
{\color{blue}
\put(0,0){
\put(3,3){\line(1,0)2}
\put(3,3){\line(0,1)3}
\put(5,6){\line(1,0)3}
}} % end color blue
\put(3.2,3.2){$T$}
\put(8.2,8.2){$T$}
\put(3.2,2.3){$(a)$}
\put(1.7,3.8){$(b)$}
\put(0.8,4.8){$(d)$}
\put(5.7,5){$(c)$}
\put(1,1.8){$(e)$}
\put(6,6.8){$(e)$}
\put(4,6.4){$(g)$}
\put(3.7,4.5){$(f)$}
\put(.1,-.3){
\put(3,1){$x$}
\put(5,3){$y$}   
\put(8,6){$x+\pi$}   
}
\put(-1.2,.2){
\put(2.5,6){$Z_\pm(x)$}
\put(4.5,8){$Z_\pm(y)$}
}
{\color{red}
\put(8,6){\line(0,1)2}
\put(3,3){\line(-1,0)3}
\put(3,3){\line(0,-1)2}
\put(5,6){\line(0,-1)3}
}
\put(2.85,2.85){$\bullet$}
\put(7.85,7.85){$\bullet$}
\end{picture}}
}}
%\vs5
\caption{Regions $(a),(b),(c)$ (with $(d)$ marking the hypotenuse of triangule $(b)$) hold objects compatible with $T=SE(x,y)$. Regions $(e),(f),(g)$ together with bottom and left boundaries form the support of $\cC_\pi^\psi(T,-)$. These intersect on the blue lines. The support of $\cC_\pi^\psi(T,-)$ is Regions $(e),(f),(g)$ together with top and right boundaries which meet the compatible region in the red lines.}
\label{fig:maps to compatible}
\end{center}
\end{figure}

Morphisms from compatible objects to $T$ are described by dual formulas and are displayed only in the figure. Lemma \ref{lem: maps to compatible objects} allow us to determine the compatible objects under $T$ which are universal in a family of such objects as long as the family is a closed set (so that it contains is limit points).

\begin{prop}\label{prop: add cT-T approximation of T}
Suppose that $T=SE(x,y)$ is a regular object in a cluster $\cT$ in $\cC_\pi^\psi$. Then there exist compatible objects $A_z,B_w$ in $\cT$ under $T$ and the left $add\,\cT\backslash\{T\}$ approximation of $T$ is $T\to A_z\oplus B_w$ where $z>x$ and $w>y$ are taken to be minimal. Dually, the right minimal $add\,\cT\backslash\{T\}$ approximation of $T$ has the form $A'_z\oplus B'_w\to T$ where $A'_z,B'_w$ are compatible objects in $\cT$ over $T$ similar to $A_z,B_w$.
\end{prop}

\begin{lem}\label{lem: maps from singular objects to compatible objects}
Nonzero morphisms from $T=Z_\e(x)$ to other compatible indecomposable objects are as follows.
\begin{enumerate}
\item[(a)] $T\to A_w:=SE(w,x+\pi)$ for $x<w<x+\pi$
\item[(b)] $T\to B_z:=Z_\e(z)$ for $x<z<x+\pi$ (i.e., all other $Z_\e(z)$).
\end{enumerate}
Furthermore:
\begin{enumerate}
\item There are no nonzero morphisms $B_z\to A_w$ under $T$ for any $z,w$.
\item There is a nonzero morphism $A_z\to A_w$ under $T$ iff $z\le w$. 
\item There is a nonzero morphism $B_z\to B_w$ under $T$ iff $z\le w$. 
\item There is a nonzero morphism $A_z\to B_w$ under $T$ iff $z\le w$. 
\end{enumerate}
\end{lem}

\begin{proof}
The regular objects compatible with $T$ are $SE(w,z)$ where $x\le w<z\le x+\pi$. (See the proof of Lemma \ref{lem: objects compatible with X}.) But $T$ maps to $SE(w,z)$ when $x\le w<x+\pi\le z$. The intersection is when $z=x+\pi$. Case (b) is easier since all $Z_\e(z)$ with the same sign $\e$ are pairwise compatible and map nontrivially to each other. Statements (1)-(4) are easy.
\end{proof}

There is a dual statement which we suppress.

\subsection{Basic distinguished triangles}

There are three basic distinguished triangles that we will use to show that the Octahedral Axiom is being used in the mutation process.

\subsubsection{Triangle (a)}

Suppose that $x<y<z<x+\pi$ and consider the following sequence of morphisms.
\begin{equation}\label{eq:triangle A}
	SE(x,y)\xrarrow{Sf} SE(x,z)\xrarrow{Sg} SE(y,z)\xrarrow{Sh} SE(x,y)[1]
\end{equation}
We note that each hom set is one dimensional, as pointed out in Lemma \ref{lem: maps to compatible objects} (1), and thus each morphism is in the image of the functor $S$ as indicated.

\begin{prop}\label{prop:triangle A}%[Basic triangle A]
\eqref{eq:triangle A} is a distinguished triangle in $\cC_\pi^\psi$ if and only if $f,g,h$ are scalar multiples of basic morphism so that the three scalars multiply to $1\in R$.
\end{prop}

\begin{rem}\label{remark on orientation of objects and sign of triangle}
If $SE(x,y)$ is replaced by the isomorphic object $SE(y,x+2\pi)$ in \eqref{eq:triangle A}, the sequence is a distinguished triangle in $\cC_\pi^\psi$ if the scalars multiply to $-1$. We draw the two isomorphic objects as arcs in the punctured disk with orientation. Then the mnemonic rule is that the scalars multiply to 1 if the first object in the triangle is oriented counterclockwise and $-1$ otherwise. (See Figure \ref{Fig:Triangle (a)}.)
\end{rem}

\begin{proof}
(Sufficiency) $S$ takes distinguished triangles in $\cC_\pi$ to distinguished triangles in $\cC_\pi^\psi$ and the stated condition is sufficient by the description of basic positive triangles given in \cite{IT10}.

(Necessity) Since $\cC_\pi^\psi(SE(y,z), SE(x,y)[1])$ is one-dimensional, any morphism must have the form $Sh$ and the choice of $h$ determines $f,g$ up to multiplication of $f$ by a scalar and of $g$ by the inverse of that scalar. The statement follows.
\end{proof}

%INSERT FIGURE for TRIANGLE A \label{Fig:Triangle (a)}
%
\begin{figure}[htbp]
\begin{center}
%
%\vs5
{
\setlength{\unitlength}{.8in}
%\centerline
{\mbox{
\begin{picture}(3,2.2)
      \thicklines
%    \thinlines
  %
    \put(1.5,1){ % beginning of circle
    \qbezier(1,0)(1,.4142)(.7071,.7071)
    \qbezier(.7071,.7071)(.4142,1)(0,1)
    \qbezier(-1,0)(-1,.4142)(-.7071,.7071)
    \qbezier(-.7071,.7071)(-.4142,1)(0,1)
    \qbezier(1,-0)(1,-.4142)(.7071,-.7071)
    \qbezier(.7071,-.7071)(.4142,-1)(0,-1)
    \qbezier(-1,-0)(-1,-.4142)(-.7071,-.7071)
    \qbezier(-.7071,-.7071)(-.4142,-1)(0,-1)
    } % end of circle
    \put(1.5,1){$\ast$}  % center of circle
\put(-.02,0){     \put(2.47,1){$\bullet\ x$}
}
\put(-.08,0){     \put(.4,.7){$z\ \bullet$}
}
  \put(0.6,1.3){$SZ$}
     \qbezier(.53,.75)(1.4,1.7)(2.5,1.05) % arc z to x
    \qbezier(.53,.75)(1,1.5)(1.5,2)    % arc z to y
   \put(1,1){$SY$}
   \put(1.45,1.95){$\bullet$}
   \put(1.5,2.1){$y$}
   \put(1.6,1.5){$SX$}
   \qbezier(1.7,1.87)(1.85,1.78)(1.9,1.8)
   \qbezier(1.7,1.87)(1.8,1.75)(1.8,1.7)
   \qbezier(1.5,2)(2,1.7)(2.5,1.05)
\end{picture}}
}}
%\vs5
\caption{Triangle (a): $SX \xrarrow{Sf} SY \xrarrow{Sg} SZ\xrarrow{Sh} SX[1]$. This is distinguished if $Sf,Sg,Sh$ are scalar multiples of basic maps with scalars multiplying to $1$.}
\label{Fig:Triangle (a)}
\end{center}
\end{figure}
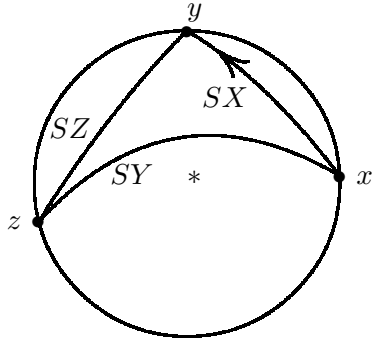

\subsubsection{Triangle (b)}

In the limiting case where $z=x+\pi$, the middle term decomposes as follows according to Proposition \ref{prop: decomposition of SE(x,x+pi)}.
\begin{equation}\label{eq:decomposition of SE middle}
	\mat{1&\z\\-\z&1}:SE(x,x+\pi)=E(x,x+\pi)\oplus E(x+\pi,x+2\pi)\cong Z_+(x)\oplus Z_-(x)
\end{equation}
The inverse of this isomorphism is given by
\[
	\mat{1&\z\\-\z&1}^{-1}=\frac12\mat{1&-\z\\ \z&1}
\]
We use the first column of the first matrix and the second row of the second matrix to get the following sequence.
\begin{equation}\label{eq:triangle B}
%\xymatrixrowsep{10pt}\xymatrixcolsep{10pt}
\xymatrix{%begin xy matrix
SE(x,y) \ar[r]^(.4){\tiny\mat{f_1\\-f_1}}& 
	Z_+(x)\oplus Z_-(x) \ar[r]^{[g_1,g_2]}&
	SE(y,x+\pi) \ar[r]^{Sh}& SE(x,y)[1]
	} % end xy matrix
\end{equation}
Here $g_1f_1=g_2f_2$ as elements of the one-dimensional hom set $\cC_\pi^\psi(SE(x,y),SE(y,x+\pi))$. The basic nonzero morphism is not in the image of the functor $S$ since $\cC_\pi(E(x,y),E(y,x+\pi))=0$. So, it is the counter-diagonal morphism which is given by the $(1,2)$ entry of the matrix. This is why we take the first column of the matrix in \eqref{eq:decomposition of SE middle} and the second row of its inverse. Note that the negative sign in \eqref{eq:triangle B} is essential in order for the composition of the first two maps to be zero.

\begin{prop}\label{prop:triangle B}%[Basic triangle B]
\eqref{eq:triangle B} is a distinguished triangle in $\cC_\pi^\psi$ if and only if $g_1f_1=g_2f_2$ is a scalar multiple, say $a$, times the basic counter-diagonal morphism from $SE(x,y)$ to $SE(y,x+\pi)$ and $h:E(y,x+\pi)\to E(x,y)[1]$ is equal to $b$ times the basic morphism and so that $ab=\frac12\in R$.
\end{prop}

%INSERT FIGURE for TRIANGLE B  \label{Fig:Triangle (b)}
%
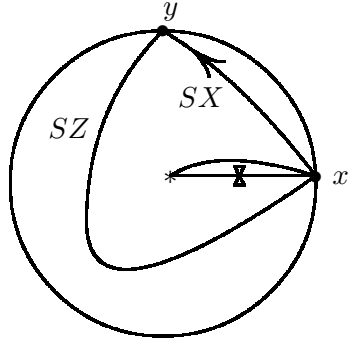
\begin{figure}[htbp]
\begin{center}
%
%\vs5
{
\setlength{\unitlength}{.8in}
%\centerline
{\mbox{
\begin{picture}(3,2.2)
      \thicklines
%    \thinlines
  %
    \put(1.5,1){ % beginning of circle
    \qbezier(1,0)(1,.4142)(.7071,.7071)
    \qbezier(.7071,.7071)(.4142,1)(0,1)
    \qbezier(-1,0)(-1,.4142)(-.7071,.7071)
    \qbezier(-.7071,.7071)(-.4142,1)(0,1)
    \qbezier(1,-0)(1,-.4142)(.7071,-.7071)
    \qbezier(.7071,-.7071)(.4142,-1)(0,-1)
    \qbezier(-1,-0)(-1,-.4142)(-.7071,-.7071)
    \qbezier(-.7071,-.7071)(-.4142,-1)(0,-1)
    } % end of circle
    \put(1.5,1){$\ast$}  % center of circle
\put(-.02,0){     \put(2.47,1){$\bullet\ x$}
}
  \put(0.75,1.3){$SZ$}
%     \qbezier(2.5,1.05)(.53,.75)(2.5,1.05) % arc z to x
    \qbezier(1,.75)(1,1.5)(1.5,2)    % y to middle of x to y arc
    \qbezier(1,.75)(1,0)(2.5,1.05)    % x to middle of x to y arc
%   \put(1,1){$SY$}
\put(1.55,1.05){\line(1,0){.95}} % line from x to center of circle
\put(2,1.05){\qbezier(-.03,-.06)(0,0)(.03,.06) % begin bowtie
\qbezier(-.03,.06)(0,0)(.03,-.06)
\put(0,-.06){\qbezier(-.03,0)(0,0)(.03,0)}
\put(0,.06){\qbezier(-.03,0)(0,0)(.03,0)}
} % end bowtie
\put(1.55,1.05){\qbezier(0,0)(.3,.2)(.95,0)} % curve from x to center of circle
   \put(1.45,1.95){$\bullet$} % bullet at y
   \put(1.5,2.1){$y$}
   \put(1.6,1.5){$SX$}
   \qbezier(1.7,1.87)(1.85,1.78)(1.9,1.8) % top of arrow head
   \qbezier(1.7,1.87)(1.8,1.75)(1.8,1.7) % bottom of arrow head
   \qbezier(1.5,2)(2,1.7)(2.5,1.05) % x to y arc
\end{picture}}
}}
%\vs5
\caption{Triangle (b): $SX \xrarrow{f} Z_+(x)\oplus Z_-(x) \xrarrow{g} SZ\xrarrow{Sh} SX[1]$. This is distinguished if corresponding scalars multiply to $\frac12$. $Z_-(x)$ is indicated with a bowtie.}
\label{Fig:Triangle (b)}
\end{center}
\end{figure}

\subsubsection{Triangle (c)}

In order to assign scalars to morphisms, we need to choose one basic morphism. For morphisms between compatible singular and regular objects we make the following choices.

Suppose $x<y<x+\pi$. Then

(1) The \emph{basic morphism} $Z_\e(y)=E(y,y+\pi)\to SE(x,y)$ is defined to be the one given by:
\[
	\small\mat{1\\ \e}:E(y,y+\pi)\to E(y,x+2\pi)\oplus  E(x+\pi,y+\pi)\cong E(x,y)\oplus E(x+\pi,y+\pi)=SE(x,y)
\]

(2) The \emph{basic morphism} $SE(x,y)\to Z_\e(x)$ is defined to be the one given by:
\[
	SE(x,y)=E(x,y)\oplus E(x+\pi,y+\pi)\cong E(x,y)\oplus E(y-\pi,x+\pi)\xrarrow{[\e,1]} E(x,x+\pi)=Z_\e(x)
\]

\begin{prop}\label{prop: precursor to Triangle (c)} Suppose that $x<y<x+\pi$ and $\e=+$ or $-$. Then the following is a distinguished triangle in $\cC_\pi$.
\begin{equation}\label{eq:precursor of Triangle (c)}
%\xymatrixrowsep{10pt}\xymatrixcolsep{10pt}
%\xymatrix{%begin xy matrix
	E(x,x+\pi) \xrarrow 1  E(y,y+\pi) \xrarrow{\tiny\mat{\e\\1}}
	E(x+\pi,y+\pi)\oplus E(y,x+2\pi) \xrarrow{[\e,-1]}
	E(x+\pi,x+2\pi)
%	} % end xy matrix
\end{equation}
Furthermore, these morphism are equivariant with respect to the automorphism $\psi$ in the sense that they give a distinguished triangle in $\cC_\pi^\psi$ as follows.
\begin{equation}\label{eq:Triangle (c)}
	Z_\e(x)\xrarrow f Z_\e(y) \xrarrow g SE(x,y) \xrarrow h Z_{\e}(x)[1]\cong Z_{-\e}(x)
\end{equation}
where $f,g,h$ are scalar multiples of the basic morphisms (defined above) in such a way that the scalars multiply to $-1$.
\end{prop}

\begin{proof}
To see that \eqref{eq:precursor of Triangle (c)} is a distinguished triangle in $\cC_\pi$, we use the definition given by Happel \cite{HappelBook}. The construction is to take the pushout of the two-way approximation sequence for $E(x,x+\pi)$ along the basic morphism $E(x,x+\pi)\to E(y,y+\pi)$ which we are denoting by $1$.
\begin{equation}\label{eq: sextuple generating Triangle (c)}
%\xymatrixrowsep{10pt}\xymatrixcolsep{10pt}
\xymatrix{%begin xy matrix
 E(x,x+\pi)\ar[d]^1\ar[r]^(.3){\tiny\mat{1\\-1}} &
	E(x+\pi,x+\pi)\oplus E(x,x+2\pi)\ar[d]^{\tiny\mat{\e&0\\0&-1}}\ar[r]^(.6){[1,1]} &
	E(x+\pi,x+2\pi)\ar[d]^= \\
E(y,y+\pi) \ar[r]^(.3){\tiny\mat{\e\\1}}& 
	E(x+\pi,y+\pi)\oplus E(y,x+2\pi) \ar[r]^(.6){[\e,-1]}&
	E(x+\pi,x+2\pi)
	}%end xy matrix
\end{equation}
We claim that this same diagram \eqref{eq: sextuple generating Triangle (c)} gives a distinguished triangle of the form \eqref{eq:Triangle (c)} in $\cC_\pi^\psi$. This comes from the fact that each morphism is $\psi$-equivariant, i.e., satisfies Definition \ref{defn: definition of completed orbit category} for a morphism in $\cC_\pi^\psi$. For example, the middle vertical arrow is $\psi$-equivariant since:
\[
	\mat{\e&0\\0&-1}\mat{0&-\e\\-\e&0}=\mat{0&1\\1&0}\mat{\e&0\\0&-1}
\]
where the matrix $\small\mat{0&-\e\\-\e&0}$ comes from the proof of Proposition \ref{shift of singular objects in Cbpsi}.

Finally, to see that the distinguished triangle \eqref{eq:precursor of Triangle (c)} in $\cC_\pi$ gives the distinguished triangle \eqref{eq:Triangle (c)}, with scalars multiplying to $-1$, note that the last morphism is $-1$ times a basic morphism since $Z_\e(x)[1]\cong Z_{-\e}(x)$ and the second to last morphism is basic since the summands $E(x+\pi,y+\pi)\oplus E(y,x+2\pi)$ are reversed. Therefore \eqref{eq:Triangle (c)} is a distringuished triangle in $\cC_\pi^\psi$.
\end{proof}

%INSERT Triangle (c). \label{Fig:Triangle (c)}
%
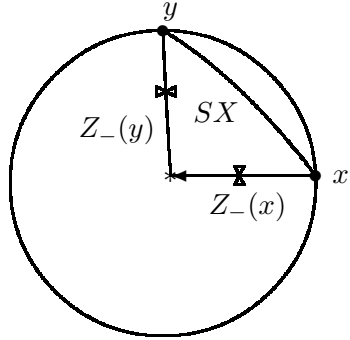
\begin{figure}[htbp]
\begin{center}
%
%\vs5
{
\setlength{\unitlength}{.8in}
%\centerline
{\mbox{
\begin{picture}(3,2.2)
      \thicklines
%    \thinlines
  %
    \put(1.5,1){ % beginning of circle
    \qbezier(1,0)(1,.4142)(.7071,.7071)
    \qbezier(.7071,.7071)(.4142,1)(0,1)
    \qbezier(-1,0)(-1,.4142)(-.7071,.7071)
    \qbezier(-.7071,.7071)(-.4142,1)(0,1)
    \qbezier(1,-0)(1,-.4142)(.7071,-.7071)
    \qbezier(.7071,-.7071)(.4142,-1)(0,-1)
    \qbezier(-1,-0)(-1,-.4142)(-.7071,-.7071)
    \qbezier(-.7071,-.7071)(-.4142,-1)(0,-1)
    } % end of circle
    \put(1.5,1){$\ast$}  % center of circle
\put(-.02,0){     \put(2.47,1){$\bullet\ x$}
}
  \put(.95,1.3){$Z_-(y)$}
%     \qbezier(2.5,1.05)(.53,.75)(2.5,1.05) % arc z to x
    \qbezier(1.55,1.05)(1.52,1.5)(1.5,2)    % y to center of circle
%    \qbezier(1,.75)(1,0)(2.5,1.05)    % x to middle of x to y arc
%   \put(1,1){$SY$}
%\put(1.55,1.05){\line(1,0){.95}} % line from x to center of circle
\put(2.5,1.05){\vector(-1,0){.95}} % arrow from x to center of circle
\put(2,1.05){\qbezier(-.03,-.06)(0,0)(.03,.06) % begin bowtie
\qbezier(-.03,.06)(0,0)(.03,-.06)
\put(0,-.06){\qbezier(-.03,0)(0,0)(.03,0)}
\put(0,.06){\qbezier(-.03,0)(0,0)(.03,0)}
} % end bowtie
\put(1.52,1.6){ % begin bowtie for Z_-(y)
\qbezier(-.06,-.03)(0,0)(.06,.03)
\qbezier(-.06,.03)(0,0)(.06,-.03)
\put(-.06,0){\qbezier(0,-.03)(0,0)(0,.03)}
\put(.06,0){\qbezier(0,-.03)(0,0)(0,.03)}
} % end bowtie for Z_-(y)
\put(1.8,.8){$Z_-(x)$}
%\put(1.55,1.05){\qbezier(0,0)(.3,.2)(.95,0)} % curve from x to center of circle
   \put(1.45,1.95){$\bullet$} % bullet at y
   \put(1.5,2.1){$y$}
   \put(1.7,1.4){$SX$}
%   \qbezier(1.7,1.87)(1.85,1.78)(1.9,1.8) % top of arrow head
%   \qbezier(1.7,1.87)(1.8,1.75)(1.8,1.7) % bottom of arrow head
   \qbezier(1.5,2)(2,1.7)(2.5,1.05) % x to y arc
\end{picture}}
}}
%\vs5
\caption{Triangle (c): $Z_-(x) \xrarrow{f} Z_-(y) \xrarrow{g} SX\xrarrow{h} Z_-(x)[1]\cong Z_+(x)$. This is distinguished if corresponding scalars multiply to $-1$ which we interpret to mean that $Z_-(x)$ is oriented clockwise (inward).}
\label{Fig:Triangle (c)}
\end{center}
\end{figure}

\subsection{Octahedral axiom}

\subsubsection{Case 1(a)-1(a)}

The figure illustrates the mutation described in Case 1(a) of Theorem \ref{DL123}. The objects $A=A_z,B=B_{w+\pi}$ are as described in Lemma \ref{LtoC}, Cases (a) and (b3). We use the well-known description, first described to us in detain by Thomas Brustle, that objects, when drawn as arcs on a surface will extend each other if and only if they cross and morphisms between compatible objects are given by counterclockwise rotation about one endpoint. That these descriptions hold on the continuous cluster category of type D is proved in Proposition \ref{nonC} and Lemma \ref{LtoC}.

Since $T\to A\to Z$, $T\to B\to Y$, $T^\ast\to Y\to A$, $T^\ast\to Z\to B$ are examples of distinguished Triangle (a) described above and all objects are regular (and thus isomorphic to their shifts), the Octahedral Axiom applies to show that
\[
	T\to A\oplus B\to T^\ast\to T[1],\qquad T[-1]\to T^\ast\to Y\oplus Z\to T
\]
are distinguished triangles. Also, $A\oplus B$ is the $add\,\cT\backslash T$-left approximation of $T$ since can be no other arcs in the cluster $\cT$ which have endpoint at $x$, since such an arc would cross $Y$. Similarly, $Y\oplus Z$ is the $add\,\cT\backslash T$-right approximation of $T$. This proves Theorems \ref{thm:cluster mutation given by approximation} and \ref{thm:cluster mutation given by Octahedral Axiom} for Case 1(a)-1(a) mutation.

%   Figure for Case 1(a) <--> 1(a) mutation
%
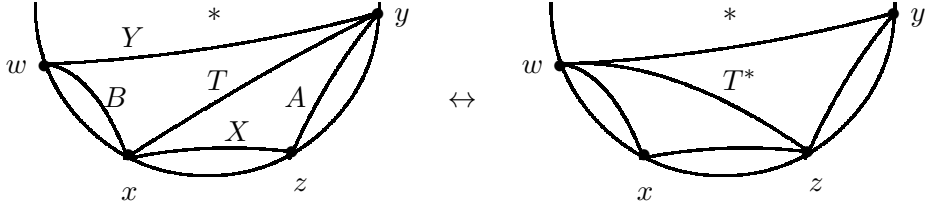
\begin{figure}[htbp]
\begin{center}
%
%\vs5
{
\setlength{\unitlength}{.9in}
%\centerline
{\mbox{
\begin{picture}(6,1.2)
      \thicklines
%    \thinlines
  %
 {   \put(1.5,1,1){ 
    % bottom of circle
    \qbezier(1,-0)(1,-.4142)(.7071,-.7071)
    \qbezier(.7071,-.7071)(.4142,-1)(0,-1)
    \qbezier(-1,-0)(-1,-.4142)(-.7071,-.7071)
    \qbezier(-.7071,-.7071)(-.4142,-1)(0,-1)
    }
    \put(1.5,1){$\ast$}
\put(-.02,0){     \put(2.47,1){$\bullet\ y$}
}
\put(-.07,0){     \put(.4,.7){$w\ \bullet$}
}
\put(1,.18){$\bullet$}
\put(1,-.04){$x$}
\put(1.95,0.2){$\bullet$}
\put(2,.015){$z$}
    \qbezier(.57,.75)(.85,.7)(1.04,.2) % w to x
    \qbezier(.57,.75)(1.6,.8)(2.5,1.05) % w to y
    \qbezier(1.98,.2)(2.2,.7)(2.5,1.05) % z to y
    \qbezier(1.04,.2)(1.5,.3)(2.0,0.24) % x to z
    \qbezier(1.02,.2)(1.8,.7)(2.5,1.05) % x to y
        \put(1,.85){$Y$}
    \put(1.5,.6){$T$}
    \put(.9,.5){$B$}
    \put(1.6,.3){$X$}
    \put(1.95,.5){$A$}
    }
    \put(2.9,.5){$\leftrightarrow$}
     \put(3,0){   \put(1.5,1,1){   % beginning of second picture
    % bottom of circle
    \qbezier(1,-0)(1,-.4142)(.7071,-.7071)
    \qbezier(.7071,-.7071)(.4142,-1)(0,-1)
    \qbezier(-1,-0)(-1,-.4142)(-.7071,-.7071)
    \qbezier(-.7071,-.7071)(-.4142,-1)(0,-1)
    }
    \put(1.5,1){$\ast$}
\put(-.02,0){     \put(2.47,1){$\bullet\ y$}
}
\put(-.07,0){     \put(.4,.7){$w\ \bullet$}
}
\put(1,.18){$\bullet$}
\put(1,-.04){$x$}
\put(1.95,0.2){$\bullet$}
\put(2,.015){$z$}
    \qbezier(.57,.75)(.85,.7)(1.04,.2) % w to x
    \qbezier(.57,.75)(1.6,.8)(2.5,1.05) % w to y
    \qbezier(1.98,.2)(2.2,.7)(2.5,1.05) % z to y
    \qbezier(.57,.75)(1.2,.8)(2.0,0.24) % w to z
    \qbezier(1.02,.2)(1.5,.3)(2.0,0.24) % x to z
    \put(1.5,.6){$T^\ast$}
    }
%/    \put(1.8,.55){$Y$}
\end{picture}}
}}
%\vs5
\caption{Case 1(a)$\leftrightarrow$1(a) mutation. $T$ maps to $A$, $B$ by pivoting at $y$ and $x$ respectively. And $Y,X$ map to $T$.}
\label{Figure Case 1a}
\end{center}
\end{figure}

\subsubsection{Case 1(b)-2(c)}

We recall Cases 1(b,c) of Theorem \ref{DL123}. These are the cases when $T=SE(x,y)$ is a regular object in a discrete cluster $\cL$ so that $J(T)$ is one of the regular intervals $I_j$ of $\cL$. There are two subcases as shown in Figures \ref{Figure Case 1b} and \ref{Figure Case 1c}. In Case 1(b), all of the singular objects in the cluster $\cT$ have the same sign, say $\e$. Then $\cT$ must contain the singular objects $Z_\e(x),Z_\e(y)$ which, together with the object $T=SE(x,y)$, form a triangle of type (a) as described above.

The objects $A=A_z=SE(z,y)$ and $B=Z_\e(x)$ in Figure \ref{Figure Case 1b} are as described in Lemma \ref{LtoC} Cases (a) and (b2). The object $T^\ast=Z_\e(z)$ is the only other object which is compatible with all other objects in $\cT$, as is apparent from the figure.

Since $X\to T\to A\to X[1]$ is an example of Triangle (a), $T\to B\to Y\to T[1]$, $X\to B\to T\to X[1]$ and $A\to T^\ast\to Y\to A[1]$ are examples of Triangle (c). Therefore, the Octahedral Axiom applies and, using the fact that $X,T$ are isomorphic to their shifts,
\[
	T\to A\oplus B\to T^\ast\to T[1],\qquad T[-1]\to T^\ast\to X\oplus Y\to T
\]
are distinguished triangles. As in the previous case, this implies Theorems \ref{thm:cluster mutation given by approximation} and \ref{thm:cluster mutation given by Octahedral Axiom} for Case 1(b)-2(c) mutation.

%   Figure for Case 1(b) <--> 2(c) mutation
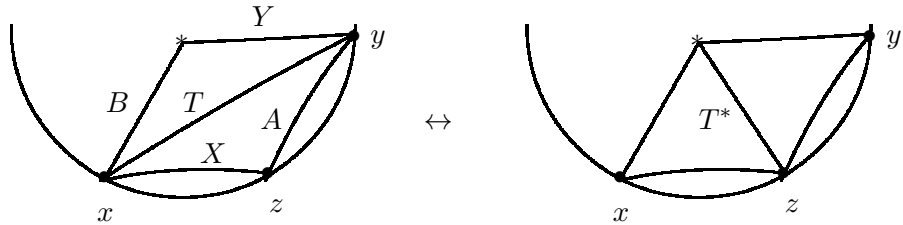
\begin{figure}[htbp]
\begin{center}
%
%\vs5
{
\setlength{\unitlength}{.9in}
%\centerline
{\mbox{
\begin{picture}(6,1.3)
      \thicklines
      {   \put(1.5,1,1){ 
    % bottom of circle
    \qbezier(1,-0)(1,-.4142)(.7071,-.7071)
    \qbezier(.7071,-.7071)(.4142,-1)(0,-1)
    \qbezier(-1,-0)(-1,-.4142)(-.7071,-.7071)
    \qbezier(-.7071,-.7071)(-.4142,-1)(0,-1)
    }
    \put(1.45,.97){$\ast$}  % center of circle
\put(-.02,0){     \put(2.47,1){$\bullet\ y$}
}
\put(1,.18){$\bullet$}
\put(1,-.04){$x$}
\put(1.95,0.2){$\bullet$}
\put(2,.015){$z$}
    \qbezier(1.5,1)(1.27,.6)(1.04,.2) % center to x
    \qbezier(1.5,1)(2,1.02)(2.5,1.05) % center to y
    \qbezier(1.98,.2)(2.2,.7)(2.5,1.05) % z to y
    \qbezier(1.04,.2)(1.5,.3)(2.0,0.24) % x to z
    \qbezier(1.02,.2)(1.8,.7)(2.5,1.05) % x to y
        \put(1.9,1.1){$Y$}
    \put(1.5,.6){$T$}
    \put(1.05,.6){$B$}
    \put(1.6,.3){$X$}
    \put(1.95,.5){$A$}
    }
        \put(2.9,.5){$\leftrightarrow$}
     \put(3,0){   \put(1.5,1,1){   %% beginning of second picture
    % bottom of circle
    \qbezier(1,-0)(1,-.4142)(.7071,-.7071)
    \qbezier(.7071,-.7071)(.4142,-1)(0,-1)
    \qbezier(-1,-0)(-1,-.4142)(-.7071,-.7071)
    \qbezier(-.7071,-.7071)(-.4142,-1)(0,-1)
    }
    \put(1.45,.97){$\ast$}  % center of circle
\put(-.02,0){     \put(2.47,1){$\bullet\ y$}
}
\put(1,.18){$\bullet$}
\put(1,-.04){$x$}
\put(1.95,0.2){$\bullet$}
\put(2,.015){$z$}
    \qbezier(1.5,1)(1.27,.6)(1.04,.2) % center to x
    \qbezier(1.5,1)(2,1.02)(2.5,1.05) % center to y
    \qbezier(1.98,.2)(2.2,.7)(2.5,1.05) % z to y
    \qbezier(1.04,.2)(1.5,.3)(2.0,0.24) % x to z
    \qbezier(1.5,1)(1.75,.62)(2.0,0.24) % center to z
 %   \qbezier(1.02,.2)(1.8,.7)(2.5,1.05) % x to y
%        \put(1.9,1.1){$Y$}
    \put(1.5,.5){$T^\ast$}
%    \put(1.05,.6){$B$}
%    \put(1.6,.3){$X$}
%    \put(1.95,.5){$A$}
    }
\end{picture}}
}}
%\vs5
\caption{Case 1(b)$\leftrightarrow$2(c) mutation.}
\label{Figure Case 1b}
\end{center}
\end{figure}

\subsubsection{Case 1(c)-1(c)}

Case 1(c) of Theorem \ref{DL123} is shown in Figure \ref{Figure Case 1c}. This is the case when $T$ is a regular object in a discrete cluster $\cL$ so that $J(T)$ is one of the regular intervals $I_j$ of $\cL$ and there are only two singular objects in the cluster of opposite signs at the same point. By symmetry we assume that $T=SE(x,y)$ and the singular objects in the cluster are $Z_+(x),Z_-(x)$. Then the objects $A=A_z=SE(z,y)$ and $B=Z_\e(x)$ are as described in Lemma \ref{LtoC} Cases (a) and (b2). The object $T^\ast=SE(z,x+\pi)$ is the only other object which is compatible with all other objects in $\cT$.

Since $T\to A\to X\to T[1]$, $T^\ast\to Y\to A\to T^\ast[1]$ are examples of Triangle (a) and $T\to Z_+(x)\oplus Z_-(x)\to Y\to T[1]$, $X\to Z_+(x)\oplus Z_-(x)\to T^\ast\to X[1]$ are examples of Triangle (b), the Octahedral Axiom applies and, using the fact that $X\cong X[1],T\cong T[1]$, we have distinguished triangles
\[
	T\to A\oplus Z_+(x)\oplus Z_-(x)\to T^\ast\to T[1],\qquad T[-1]\to T^\ast\to X\oplus Y\to T.
\]
As in the previous case, this implies Theorems \ref{thm:cluster mutation given by approximation} and \ref{thm:cluster mutation given by Octahedral Axiom} for Case 1(c)-1(c) mutation.

%   Figure for Case 1(c) <--> 1(c) mutation
\begin{figure}[htbp]
\begin{center}
%
%\vs5
{
\setlength{\unitlength}{.9in}
%\centerline
{\mbox{
\begin{picture}(6,2)
      \thicklines
{  \put(1.5,1){ % beginning of circle  %% begin second figure
    \qbezier(1,0)(1,.4142)(.7071,.7071)
    \qbezier(.7071,.7071)(.4142,1)(0,1)
    \qbezier(-1,0)(-1,.4142)(-.7071,.7071)
    \qbezier(-.7071,.7071)(-.4142,1)(0,1)
    \qbezier(1,-0)(1,-.4142)(.7071,-.7071)
    \qbezier(.7071,-.7071)(.4142,-1)(0,-1)
    \qbezier(-1,-0)(-1,-.4142)(-.7071,-.7071)
    \qbezier(-.7071,-.7071)(-.4142,-1)(0,-1)
    } % end of circle
    \put(1.5,1){$\ast$}  % center of circle
\put(-.02,0){     \put(2.47,1){$\bullet\ x$}
}
  \put(1.3,1.3){$T$}
%     \qbezier(2.5,1.05)(.53,.75)(2.5,1.05) % arc z to x
%    \qbezier(1,.75)(1,1.5)(1.5,2)    % y to middle of x to y arc
%    \qbezier(1,.75)(1,0)(2.5,1.05)    % x to middle of x to y arc
%   \put(1,1){$SY$}
\qbezier(.93,.14)(1.3,1.9)(2.45,1.08) % curved line from y to x
\qbezier(.92,.145)(1.7,.4)(2.5,1.07) % lower curved line from y to x
\qbezier(.93,.14)(.7,1)(1.5,2) % curved line from y to z 
\put(1.55,1.05){\line(1,0){.95}} % line from x to center of circle
\put(2,1.05){\qbezier(-.03,-.06)(0,0)(.03,.06) % begin bowtie
\qbezier(-.03,.06)(0,0)(.03,-.06)
\put(0,-.06){\qbezier(-.03,0)(0,0)(.03,0)}
\put(0,.06){\qbezier(-.03,0)(0,0)(.03,0)}
} % end bowtie
  \put(0.9,1.3){$Y$}
\put(1.55,1.05){\qbezier(0,0)(.3,.2)(.95,0)} % curve from x to center of circle
   \put(1.45,1.95){$\bullet$} % bullet at y
   \put(1.5,2.1){$z$}
%   \put(1.6,1.5){$SX$}
%   \qbezier(1.7,1.87)(1.85,1.78)(1.9,1.8) % top of arrow head
%   \qbezier(1.7,1.87)(1.8,1.75)(1.8,1.7) % bottom of arrow head
   \qbezier(1.5,2)(2,1.7)(2.5,1.05) % x to y arc
   \put(.9,.13){$\bullet$}
\put(.8,-.04){$y$}
\put(1.8,.4){$A$}
\put(1.8,.8){$B$}
   }  % end first figure
            \put(2.9,1){$\leftrightarrow$}
 \put(3,0){  \put(1.5,1){ % beginning of circle  %% begin second figure
    \qbezier(1,0)(1,.4142)(.7071,.7071)
    \qbezier(.7071,.7071)(.4142,1)(0,1)
    \qbezier(-1,0)(-1,.4142)(-.7071,.7071)
    \qbezier(-.7071,.7071)(-.4142,1)(0,1)
    \qbezier(1,-0)(1,-.4142)(.7071,-.7071)
    \qbezier(.7071,-.7071)(.4142,-1)(0,-1)
    \qbezier(-1,-0)(-1,-.4142)(-.7071,-.7071)
    \qbezier(-.7071,-.7071)(-.4142,-1)(0,-1)
    } % end of circle
    \put(1.5,1){$\ast$}  % center of circle
\put(-.02,0){     \put(2.47,1){$\bullet\ x$}
}
%     \qbezier(2.5,1.05)(.53,.75)(2.5,1.05) % arc z to x
    \qbezier(1.2,.95)(1.2,1.5)(1.5,2)    % y to middle of x to y arc
    \qbezier(1.2,.95)(1.2,.5)(2.5,1.05)    % x to middle of x to y arc
   \put(1.2,.5){$T^\ast$}
\put(1.55,1.05){\line(1,0){.95}} % line from x to center of circle
\put(2,1.05){\qbezier(-.03,-.06)(0,0)(.03,.06) % begin bowtie
\qbezier(-.03,.06)(0,0)(.03,-.06)
\put(0,-.06){\qbezier(-.03,0)(0,0)(.03,0)}
\put(0,.06){\qbezier(-.03,0)(0,0)(.03,0)}
} % end bowtie
\put(1.55,1.05){\qbezier(0,0)(.3,.2)(.95,0)} % curve from x to center of circle
   \put(1.45,1.95){$\bullet$} % bullet at y
   \put(1.5,2.1){$z$}
   \put(1.7,1.5){$X$}
\qbezier(.92,.145)(1.7,.4)(2.5,1.07) % lower curved line from y to x
\qbezier(.93,.14)(.7,1)(1.5,2) % curved line from y to z 
   \qbezier(1.5,2)(2,1.7)(2.5,1.05) % x to y arc
   \put(.9,.13){$\bullet$}
\put(.8,-.04){$y$}
    }
\end{picture}}
}}
%\vs5
\caption{Case 1(c)$\leftrightarrow$1(c) mutation. $B=B_+(x)\oplus B_-(x)$.}
\label{Figure Case 1c}
\end{center}
\end{figure}
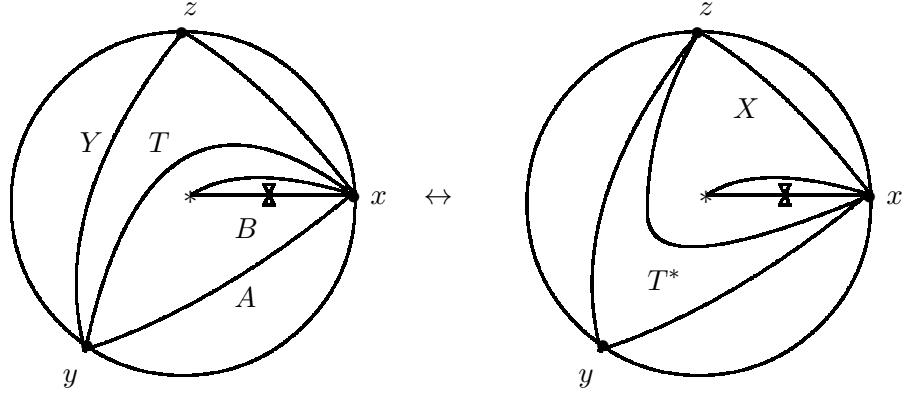

\subsubsection{Exceptional Case 2(a)-2(b)}

Suppose now that $\cT$ is a cluster in $\cC_\pi^\psi$ with exactly two singular objects and $T$ is one of those objects. Then we are in Case 2(a) or 2(b) of Theorem \ref{DL123}. Figure \ref{Figure Case 2a} illustrates both of these cases. On the left is Case 2(a). Here $T=Z_\e(z)$ maps nontrivially to $A=A_y$ and to $B=B_y=Z_\e(y)$ as examples of Lemma \ref{lem: maps from singular objects to compatible objects} Cases (a) and (b). By Lemma \ref{lem: maps from singular objects to compatible objects} (4), the morphism $T\to B$ factors through $A$. Therefore, $T\to A$ is the left $add\,\cT\backslash T$-approximation of $T$. Similarly, $X\to T$ is the right $add\,\cT\backslash T$-approximation of $T$. As examples of Triangle (c) we have the distinguished triangles:
\[
	T\to A\to B[1]\to T[1]\quad\text{and}\quad T[-1]\to B[-1]\to X\to T
\]
showing that the mutation $T^\ast=B[1]$ of $T$ as given by Theorem \ref{DL123} is given by $add\,\cT\backslash T$-approximations. This proves Theorem \ref{thm:cluster mutation given by approximation} for the exceptional Case 2(a)-2(b) mutation.

The Octahedral Axiom does in fact apply in the following trivial way. Using the fact that $B\to T\to A\to B[1]$ is an example of Triangle (c) with $A\to B[1]\to T[1]\to A[1]$ being a rotation of the same triangle, we have the following example of the Octahedral Axiom:
\begin{equation}\label{Octahedron Case 2(a,b)}
\xymatrixrowsep{10pt}\xymatrixcolsep{10pt}
\xymatrix{
B\ar[dr] && A\ar[ll] &&&& B\ar[dd] && A\ar[dl]\\
	& T\ar[ur]\ar[dl] &&&\longrightarrow&&& B[1]\ar[ul]_1\ar[dr]\\
	0\ar[rr] && T[1]\ar[ul]_1&&&& 0\ar[ur]&&T[1]\ar[uu]
	}%end xymatrix
\end{equation}
This gives the correct $T^\ast=B[1]$ and the correct left $add\,\cT\backslash T$-approximation sequence $T\to A\to B[1]\to T[1]$. However, it does not give the correct right $add\,\cT\backslash T$-approximation since $X$ does not appear in the diagram.

%   Figure for Case 2(a) <--> 2(b) mutation
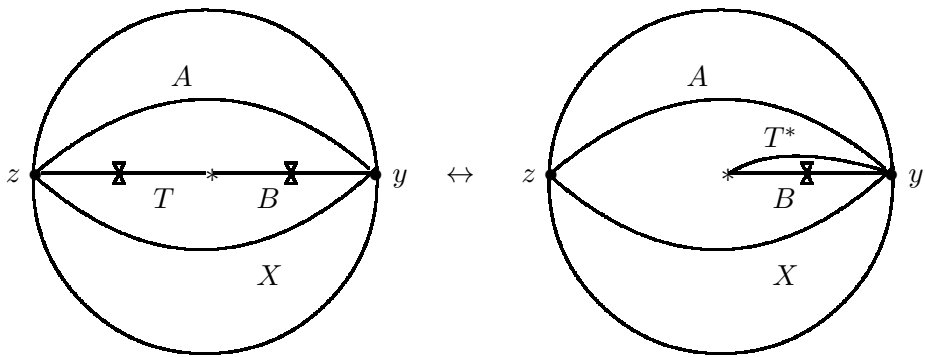
\begin{figure}[htbp]
\begin{center}
%
%\vs5
{
\setlength{\unitlength}{.9in}
%\centerline
{\mbox{
\begin{picture}(6,2)
      \thicklines
          % end first figure
{  \put(1.5,1){ % beginning of circle  %% begin first figure
    \qbezier(1,0)(1,.4142)(.7071,.7071)
    \qbezier(.7071,.7071)(.4142,1)(0,1)
    \qbezier(-1,0)(-1,.4142)(-.7071,.7071)
    \qbezier(-.7071,.7071)(-.4142,1)(0,1)
    \qbezier(1,-0)(1,-.4142)(.7071,-.7071)
    \qbezier(.7071,-.7071)(.4142,-1)(0,-1)
    \qbezier(-1,-0)(-1,-.4142)(-.7071,-.7071)
    \qbezier(-.7071,-.7071)(-.4142,-1)(0,-1)
    } % end of circle
    \put(1.5,1){$\ast$}  % center of circle
\put(-.02,0){     \put(2.47,1){$\bullet\ y$}
}
  \put(1.3,1.56){$A$}
   \put(0.47,1){$\bullet$} % bullet at z
   \put(.34,1){$z$}
\qbezier(.5,1.04)(1.5,1.9)(2.45,1.08) % curved line from z to y
\qbezier(.5,1.04)(1.5,.15)(2.5,1.08) % curved line from z to y
\put(1.55,1.05){\line(1,0){.95}} % line from y to center of circle
\put(.55,1.05){\line(1,0){.95}} % line from y to center of circle
%\put(1.5,1){\line(-1,0){.95}} % line from z to center of circle
\put(2,1.05){\qbezier(-.03,-.06)(0,0)(.03,.06) % begin bowtie
\qbezier(-.03,.06)(0,0)(.03,-.06)
\put(0,-.06){\qbezier(-.03,0)(0,0)(.03,0)}
\put(0,.06){\qbezier(-.03,0)(0,0)(.03,0)}
} % end bowtie
\put(1,1.05){\qbezier(-.03,-.06)(0,0)(.03,.06) % begin bowtie
\qbezier(-.03,.06)(0,0)(.03,-.06)
\put(0,-.06){\qbezier(-.03,0)(0,0)(.03,0)}
\put(0,.06){\qbezier(-.03,0)(0,0)(.03,0)}
} % end bowtie
%  \put(0.9,1.3){$Y$}
%\put(1.55,1.05){\qbezier(0,0)(.3,.2)(.95,0)} % curve from y to center of circle
\put(1.8,.4){$X$}
\put(1.8,.85){$B$}
\put(1.2,.85){$T$}
   }  % end first figure
            \put(2.9,1){$\leftrightarrow$}
 \put(3,0){  \put(1.5,1){ % beginning of circle  %% begin second figure
    \qbezier(1,0)(1,.4142)(.7071,.7071)
    \qbezier(.7071,.7071)(.4142,1)(0,1)
    \qbezier(-1,0)(-1,.4142)(-.7071,.7071)
    \qbezier(-.7071,.7071)(-.4142,1)(0,1)
    \qbezier(1,-0)(1,-.4142)(.7071,-.7071)
    \qbezier(.7071,-.7071)(.4142,-1)(0,-1)
    \qbezier(-1,-0)(-1,-.4142)(-.7071,-.7071)
    \qbezier(-.7071,-.7071)(-.4142,-1)(0,-1)
    } % end of circle
    \put(1.5,1){$\ast$}  % center of circle
\put(-.02,0){     \put(2.47,1){$\bullet\ y$}
}
  \put(1.3,1.56){$A$}
   \put(0.47,1){$\bullet$} % bullet at z
   \put(.34,1){$z$}
\qbezier(.5,1.04)(1.5,1.9)(2.45,1.08) % curved line from z to y
\qbezier(.5,1.04)(1.5,.15)(2.5,1.08) % curved line from z to y
\put(1.55,1.05){\line(1,0){.95}} % line from x to center of circle
\put(2,1.05){\qbezier(-.03,-.06)(0,0)(.03,.06) % begin bowtie
\qbezier(-.03,.06)(0,0)(.03,-.06)
\put(0,-.06){\qbezier(-.03,0)(0,0)(.03,0)}
\put(0,.06){\qbezier(-.03,0)(0,0)(.03,0)}
} % end bowtie
%  \put(0.9,1.3){$Y$}
\put(1.55,1.05){\qbezier(0,0)(.3,.2)(.95,0)} % curve from x to center of circle
\put(1.8,.4){$X$}
\put(1.8,.85){$B$}
\put(1.75,1.2){$T^\ast$}
   }  %% end second picture
\end{picture}}
}}
%\vs5
\caption{Exceptional Case 2(a)$\leftrightarrow$2(b) mutation. $T^\ast=B[1]$}
\label{Figure Case 2a}
\end{center}
\end{figure}
%

%\newpage

\subsection{Embedding finite $D_n$}

For any $n\ge 4$, the cluster category of type $D_n$ can be realized as a subquotient category of the continuous cluster category $\cC_\pi^\psi$ in two different ways. One method is to ``freeze'' all regular objects $SE(x,y)$ in $\cC_\pi^\psi$ where $|y-x|\le \pi/n$. We can start with any lamination $\cL$ with singular set $S(\cL)=\{k\pi/n\,|\, k\in\ZZ\}$. The lamination will include regular objects isomorphic to $X_k=SE((k-1)\pi/n,k\pi/n)$ and have sublaminations of the $n$ open intervals $J(X_k)=((k-1)\pi/n,k\pi/n)$ which are ``frozen'', i.e., fixed.

Now consider only laminations which include the objects $X_k$. The possible mutable objects are:
\begin{enumerate}
\item The $2n$ singular objects $Z_\pm(k\pi/n)$ and
\item The $n^2-2n$ regular objects $SE(j\pi/n,k\pi/n)$ where $1\le j\le n$ and $j+2\le k<j+n$.
\end{enumerate}
Readers familiar with cluster categories can readily check that clusters of this form mutate in the same way that clusters of type $D_n$ mutate. A more formal approach is given below.

Another method of embedding the cluster category of type $D_n$ into a quotient of $\cC_\pi^\psi$ is to take $b=\frac{(n-1)\pi}n,n\ge4$ and consider $\cC_b^\psi$. This is the stable category of the Frobenius category $\cF_b^\psi$ which is the idempotent completed orbit category of $\cF_b$. It was shown in \cite{IT10} that the stable category of $\cF_b$ contains the cluster category of type $A_{2n-3}$. The inclusion functor $\cC_{A_{2n-3}}\into \cC_b$ is equivariant with respect to the $\ZZ/2$ action given by rotation by $\pi$ and therefore induces an inclusion functor on orbit cluster categories $\cC_{D_n}\into \cC_b^\psi$.

\subsection{Comments on $\ZZ/p$ actions}

Assume that $p$ is an odd prime. Then the group $\ZZ/p$ acting by rotation by $2\pi/p$ acts freely on the set of unordered pairs of point on the circle. Therefore, all of the objects in $\cF_\pi^{\ZZ/p}$ are regular. However, as objects of the stable category $\cC^{\ZZ/p}=\ul\cF_\pi^{\ZZ/p}$, they are not all ``rigid''. Since all objects are isomorphic to their shifts, the identity map is a self-extension of any object in $\cC^{\ZZ/p}$. Therefore, we define an object $X$ to be \emph{rigid} is its endomorphism ring is $K$. We define $X$ to be \emph{almost rigid} if $X$ is the limit of a sequence of rigid objects. One can also take the following lemma as the definition of ``rigid'' and ``almost rigid''.

\begin{lem}
An indecomposable object $X$ in $\cC^{\ZZ/p}$ is rigid if and only if $X$ is isomorphic to $SE(x,y)$ where $x<y< x+2\pi/p$. $X$ is almost rigid if an only if $X\cong SE(x,y)$ where $x<y\le x+2\pi/p$. If $X$ is almost rigid but not rigid then its endomorphism ring is isomorphic to $K[\e]$, with ring of dual number, with $\e^2=0$.
\end{lem}

\begin{proof} The second statement follows from the first since the set of objects described in the second statement is the closure of the set described in the first statement. To prove the first statement, we may assume by symmetry that $X=SE(x,y)$ where $x<y\le x+\pi$. Then we use the adjunction formula to give
\[
	\cC^{\ZZ/p}(SE(x,y),SE(x,y))=\bigoplus_k \cC(E(x,y),E(x+k\th,y+k\th))
\]
where $\th=2\pi/p$. But $\cC(E(x,y),E(x+k\th,y+k\th))\neq0$ if and only if either $x\le x+k\th<y$ or $x\le y+k\th <y$. This holds iff $x\le y-\th$. The lemma follows.
\end{proof}

By this lemma, all almost rigid objects are given, up to isomorphism, by $X=SE(x,y)$ where $x<y\le x+2\pi/p$. Let $J(X)$ be the open interval $(x,y)$ in the circle $S^1_{2\pi/p}=\RR/\frac{2\pi}p\ZZ$. We represent the object $SE(x,y)$ by a embedded arc in the disk $D^2_p=D^2/(\ZZ/p)$ connecting the points $x,y\in S^1_{2\pi/p}$, disjoint from the center point, and homotopic to $J(X)$ fixing the endpoints and avoiding the center. Then the two objects $SE(x,y)$, $SE(y,x+2\pi/p)$ are nonisomorphic but have the same endpoints. In the limiting case where $y=x+2\pi/p$, this embedded arc is a loop at the point $x$ which encloses the center of the disk $D^2_p$. We use the shorthand notation $SE(x)=SE(x,x+2\pi/p)$.

Almost rigid objects $X,Y$ are defined to be \emph{compatible} if there is a sequence of objects $Y_n$ converging to $Y$ so that $\cC_\pi^{\ZZ/p}(X,Y_n)=0$. 

\begin{lem}
Two rigid objects $X,Y$ are compatible if and only if the sets $J(X),J(Y)$ are either disjoint or one contains the other. In particular, two almost rigid objects $SE(x),SE(y)$ are compatible if and only if they are isomorphic (equivalently, $x-y$ is an integer multiple of $2\pi/p$). 
\end{lem}

As before, this is equivalent to the statement that the corresponding arcs in the punctured disk do not cross (i.e., they intersect only at their endpoints). We define a \emph{cluster} to be a discrete maximal compatible sets of almost rigid objects. We will see that these form a cluster structure in the triangulated category $\cC_\pi^{\ZZ/p}$. We also refer to them as \emph{discrete laminations} of the disk $D^2_p$.

Following the same arguments and using the same visualization as in the case $p=2$, we get the following theorem.

\begin{thm}
Every discrete lamination $\cL$ of the disk $D^2_p$ is given by $\cL=\{SE(x)\}\coprod \cL_a\coprod \cL_b$ where, for some pair of points $x<y<x+2\pi/p$ in $S^1_{2\pi/p}$, $\cL_a$ is a discrete sublamination of $(x,y)$ and $\cL_b$ is a discrete sublamination of $(y,x+2\pi/p)$. In particular, every discrete lamination of $D^2_p$ has exactly one object which is not rigid.
\end{thm}

\begin{proof} First of all, the circle $S^1_{2\pi/p}$ cannot be covered with a collection of open intervals which are either disjoint or one contains the other. Therefore, there is at least one point $z$ which is not in $J(X)$ for any $X\in\cL$. Then $SE(z)$ is compatible with every object in $\cL$. So, $SE(z)\in\cL$ (up to isomorphism).

Consider the set of all intervals $J(X)=(x,y)$ for all rigid $X\in\cL$. This set is ordered by inclusion and contains the supremum of any ascending tower. This follows from the fact that, if $(x,y)$ is the union of any increasing sequence of intervals, then $SE(x,y)$ will be compatible with every object of $\cL$ and therefore must lie in $\cL$ by maximality of $\cL$. Also, the case $y=x+2\pi/p$ is not possible since $\cL$ is discrete and therefore any converging sequence is eventually stationary.

By Zorn's lemma, the set of rigid objects of $\cL$ contains an object $X$ so that $J(X)=(x,y)$ is maximal. Then $x<y<x+2\pi/p$ by the lemma. Then $SE(x,y)$ will be compatible with all objects of $\cL$. So, $SE(x,y)\in\cL$. Furthermore, $z\notin(x,y)$ since $SE(z)$ is compatible with $SE(x,y)$. But then $z$ must equal $x$ or $y$ since, otherwise $Y=SE(x,z)$ will be a rigid object compatible with all objects of $\cL$. So, $Y\in\cL$ contradicting the maximality of the interval $(x,y)$. From this it follows that $SE(y,x+2\pi/p)$ is compatible with all objects in $\cL$. So, it must be an object in $\cL$. Then, for any $Z\in \cL$, either $Z=SE(z)$ or $J(Z)$ is contained in one of the intervals $(x,y)$ or $(y,x+2\pi/p)$. The objects of the first kind form a sublamination of $(x,y)$ and the objects of the second kind form a sublamination of $(y,x+2\pi/p)$ proving the theorem.
\end{proof}

\begin{cor}
Any two clusters $\cL,\cL'$ in $\cC_\pi^{\ZZ/p}$ are equivalent in the sense that there is a strictly triangulated automorphism of $\cC_\pi^{\ZZ/p}$ induced by an homeomorphism of the circle which takes $\cL$ to $\cL'$.
\end{cor}

\begin{thm}
Given any object $T$ in any cluster $\cL$, there is, up to isomorphism, exactly one object $T^\ast$ so that $\cL\backslash T\cup T^\ast$ is a cluster. Furthermore, the object $T^\ast$ can be obtained from $T$ by forming distinguished triangles:
\[
	T\to A\to T^\ast\to T[1],\quad T[-1]\to T^\ast\to B\to T
\]
where $T\to A$ and $B\to T$ are left and right $add\,\cL\backslash T$-approximations of $T$.
\end{thm}

\begin{proof}
We works in the same way as before except in the case when $T$ is the unique nonrigid object of $\cL$. So, we examine only those cases.

Let $T_4=SE(x)$ be the unique nonrigid object of $\cL$. Then $\cL$ also contains two maximal rigid objects $T_3=SE(x,y)$ and $T_5=SE(y,x+2\pi/p)$. By the description of all discrete sublaminations of an interval in Lemma \ref{decomposition of discrete sublaminations}, there exists a point $x<z<y$ so that $T_1=SE(x,z),T_2=SE(z,y)$ are objects of $\cL$ (up to isomorphism). When $T=T_4$, it is easy to see that $T^\ast=SE(y)$ is the only possible mutation of $T$. Also, $T_3^\ast=SE(x,z)$ is the only possible mutation of $T_3$. $T_5$ is similar to $T_3$ and the other objects follow the same pattern as the 1(a)$\leftrightarrow$1(a) case for $\cC_\pi^\psi$ mutation. This proves the first part of the statement.

In the case $T=T_4$, the left $add\,\cL\backslash T_4$ approximation of $T_4$ is the direct sum of two copies of $T_5$ since (with notation $\th=2\pi/p$),
\[
	\cC_\pi^{\ZZ/p}(T_4,T_5)=\bigoplus \cC_\pi(E(x,x+\th),E(y+k\th,x+(k+1)\th))
\]
\[
	= \cC_\pi(E(x,x+\th),E(y,x+\th))\oplus \cC_\pi(E(x,x+\th),E(y-\th,x))=K^2
\]
and we get the distinguished triangle:
\[
	T_4\to T_5^2\to SE(y)\to T_4[1]\cong SE(x).
\]
By symmetry the right $add\,\cL\backslash T_4$ approximation of $T_4$ is $T_3^2$ and we get the approximation triangle:
\[
	SE(y)\to T_3^2\to T_4\to SE(y)[1]\cong SE(y).
\]
The other cases work as before.
\end{proof}

The mutation of the quiver also follows the Fomin-Zelevinsky formula:
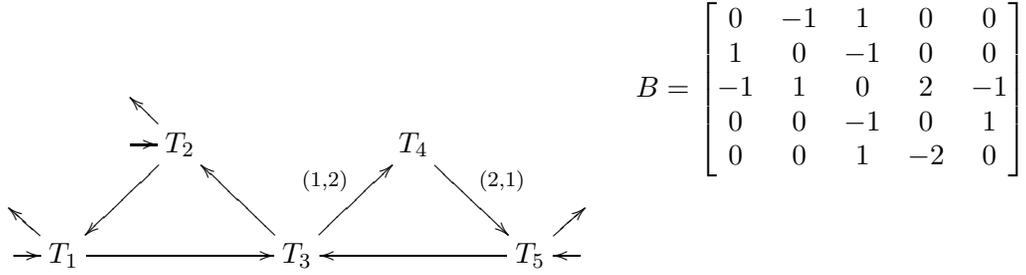
\[
\xymatrixrowsep{10pt}\xymatrixcolsep{10pt}
\xymatrix{%begin xy matrix
&&&&&&&&&&\\
&&\ar[r]& T_2\ar[lu]\ar[ddll]&&&&T_4\ar[ddrr]^{(2,1)}\\
&&&&&&&&&&\\
\ \ar[r]&T_1\ar[lu]\ar[rrrr]&&&& T_3\ar[lluu]\ar[rruu]^{(1,2)} &&&& T_5\ar[llll]\ar[ru]&\ \ar[l]
	}%end xy matrix
	\quad B=\mat{0&-1&1 & 0&0\\
	1&0&-1 &0&0\\
	-1&1&0 &2&-1\\
	0&0&-1&0&1\\
	0&0&1&-2&0}
\]
The valuation $T_4\xrarrow{(2,1)}T_5$ comes from the fact that $\cC_\pi^{\ZZ/p}(T_4,T_5)=K^2$ is free with one generator as an $\End(T_4)$-module but free with 2 generators as an $\End(T_5)$-module. Similarly for $T_3\xrarrow{(1,2)}T_4$. Mutation at $T_4$ gives:
\[
\xymatrixrowsep{10pt}\xymatrixcolsep{10pt}
\xymatrix{%begin xy matrix
&&&&&&&&&&&\\
&&\ar[r]& T_2\ar[lu]\ar[ddll]&&&&T_4^\ast\ar[ddll]_{(1,2)}\\
&&&&&&&&&&&&&&&\\
\ \ar[r]&T_1\ar[lu]\ar[rrrr]&&&& T_3\ar[rrrr]\ar[lluu] &&&& T_5\ar[lluu]_{(2,1)}\ar[ru]&\ar[l]
	}%end xy matrix
	\quad B=\mat{0&-1&1 & 0&0\\
	1&0&-1 &0&0\\
	-1&1&0 &-2&1\\
	0&0&1&0&-1\\
	0&0&-1&2&0}
\]
Subsequent mutation at $T_3$ gives:
\[
\xymatrixrowsep{10pt}\xymatrixcolsep{10pt}
\xymatrix{%begin xy matrix
&&\ar[r]& T_2\ar[ld]\ar[ddrr]&&&&T_4^\ast\ar[llll]_{(1,2)}\\
&&&&&&&\\
\ar[r]&T_1\ar[lu]\ar@/_2pc/[rrrrrrrr]&&&& T_3^\ast\ar[llll]\ar[rruu]^{(1,2)} &&&& T_5\ar[llll]\ar[r]&\\
&&&&&&&&&&\ar[lu]
	}%end xy matrix
	\quad B=\mat{0&0&-1 & 0&1\\
	0&0&1 &-2&0\\
	1&-1&0 &2&-1\\
	0&1&-1&0&0\\
	-1&0&1&0&0}
\]

%   Figure for Zp cluster
\begin{figure}[htbp]
\begin{center}
%
%\vs5
{
\setlength{\unitlength}{.9in}
%\centerline
{\mbox{
\begin{picture}(3,2)
      \thicklines
          % end first figure
{  \put(1.5,1){ % beginning of circle  %% begin second figure
    \qbezier(1,0)(1,.4142)(.7071,.7071)
    \qbezier(.7071,.7071)(.4142,1)(0,1)
    \qbezier(-1,0)(-1,.4142)(-.7071,.7071)
    \qbezier(-.7071,.7071)(-.4142,1)(0,1)
    \qbezier(1,-0)(1,-.4142)(.7071,-.7071)
    \qbezier(.7071,-.7071)(.4142,-1)(0,-1)
    \qbezier(-1,-0)(-1,-.4142)(-.7071,-.7071)
    \qbezier(-.7071,-.7071)(-.4142,-1)(0,-1)
    } % end of circle
    \put(1.5,1){$\ast$}  % center of circle
\put(-.02,0){     \put(2.47,1){$\bullet\ x$}
}
  \put(1.7,1.6){$T_1$}
  \put(1.35,1.4){$T_3$}
%     \qbezier(2.5,1.05)(.53,.75)(2.5,1.05) % arc z to x
%    \qbezier(1,.75)(1,1.5)(1.5,2)    % y to middle of x to y arc
%    \qbezier(1,.75)(1,0)(2.5,1.05)    % x to middle of x to y arc
%   \put(1,1){$SY$}
\qbezier(.93,.14)(1.1,1.9)(2.45,1.08) % curved line from y to x
\qbezier(.92,.145)(1.7,.4)(2.5,1.07) % lower curved line from y to x
\qbezier(.93,.14)(.7,1)(1.5,2) % curved line from y to z 
%\put(1.55,1.05){\line(1,0){.95}} % line from x to center of circle
  \put(0.85,1.3){$T_2$}
\put(1.55,1.05){\qbezier(-.15,0)(-.15,.2)(.95,0)} % curve from x to center of circle
\put(1.55,1.05){\qbezier(-.15,0)(-.15,-.2)(.95,0)} % curve from x to center of circle
   \put(1.45,1.95){$\bullet$} % bullet at y
   \put(1.5,2.1){$z$}
%   \put(1.6,1.5){$SX$}
%   \qbezier(1.7,1.87)(1.85,1.78)(1.9,1.8) % top of arrow head
%   \qbezier(1.7,1.87)(1.8,1.75)(1.8,1.7) % bottom of arrow head
   \qbezier(1.5,2)(2,1.7)(2.5,1.05) % x to y arc
   \put(.9,.13){$\bullet$}
\put(.8,-.04){$y$}
\put(1.8,.4){$T_5$}
\put(1.8,.8){$T_4$}
   }  % end first figure
             %% end second picture
\end{picture}}
}}
%\vs5
\caption{Cluster in $\cC_\pi^{\ZZ/p}$.}
\label{Figure Zp cluster}
\end{center}
\end{figure}

\begin{rem} Note that the description of the clusters, their quivers and exchange matrices, partially indicated above, is independent of $p$. And in fact it also applies to the case $p=2$ when $K=R/\mathfrak m$ is a field of characteristic 2. In that case, the exceptional objects $SE(x,x+\pi)$ are indecomposable with endomorphism ring $K[\e]$ with $\e^2=0$ by Remark \ref{char 2 case} and the theorems and formulas for the $\ZZ/p$ case apply. In particular this implies that for any cluster $\cL$ in $\cC_\pi^\psi$ in the case when $K$ has characteristic 2, and any cluster $\cL'$ in $\cC_\pi^{\ZZ/p}$ for any odd prime $p$, there is a homeomorphism of the circle that takes $\cL$ as a subset of the set of order points on the circle to $\cL'$. This is not a functor since these categories are defined over fields of different characteristic.

We also note that this particular example of orbifold cluster structure is given as just one example: $\widetilde{IV}$ in a complete classification of such structures (triangulations of surface orbifolds) in \cite{FSTu}.
\end{rem}

%\newpage

\section{Appendix: nonabelian stablizers}

In this appendix we assume that $G$ is a finite group acting on a Frobenius cyclic poset $(\cX,\cX_0,c)$ but we allow effective stabilizers $H_X$ to be any (finite) group. We know by Theorem \ref{FG is a Frobenius category} that $\cF^G(\cX,\cX_0)$ is Frobenius whose projective-injective objects are components of $SP$ for $P\in\cF(\cX_0)$. We will show that $\cF^G$ is Krull-Schmidt and describe all of the indecomposable objects. To make this easy, we will assume that $R$ is the complete local ring $R=\CC[[t]]$. 

Let $n=|G|$. We fix an object $X\in \cX$ and let $H=H_X$ with order $|H_X|=m$ and index $|G:H_X|=\el=n/m$. Let $\s_i, i=1,\cdots, \el$ be representatives of the left cosets of $H$ in $G$ so that $G=\coprod \s_iH$. For any $\b\in H$, we have a basic isomorphism $\b X\cong X$ and, since $\cF(\b X,X)\cong R$, any unit $a\in R$ also gives an isomorphism $T_a:\b X\cong X$. Similarly any element of $A\in GL(d,R)$ gives an isomorphism $T_A:\b X^d\cong X^d$. More generally, we will use the notation $[f]$ to denote the matrix with entries $[f]_{ij}\in R$ for any morphism $f:\oplus X_j\to \oplus Y_i$. In particular $T_{[f]}=f$ for any $f:\b X^d\cong X^d$. 

\emph{Warning:} Composition is not given by matrix multiplication! For $f:Y\to Z, g:X\to Y$, we have
\[
	[fg]_{ik}=t^n[f]_{ij}[g]_{jk}
\]
where $n=c(Z_i,Y_j,X_k)$.

\begin{defn}
Suppose that $X\in \cX$ and $\r:H\to GL(d,\CC)\subset GL(d,R)$ is a representation over $\CC$ of the effective stabilizer $H=H_X$ of $X$. Then, for every $\b\in H$, the $d\times d$ matrix $\r(\b)$ gives an isomorphism $T_{\r(\b)}:\b X^d\to X^d$.  Then let \[
S_\r X:=(\smallcoprod\s_i X^d,\xi)\]
where $\xi$ is defined as follows.

For every $\g\in G$ and every $\s_i$, we have $\g\s_i=\s_j\eta_j(\g)$ where $\s_j$ and $\eta_j(\g)\in H$ are uniquely determined by $\g$ and $\s_i$. Let $\xi_\g:\coprod\g\s_i X^d\to \coprod \s_jX^d$ be the isomorphism whose matrix $[\xi_\g]$ is monomial with $(j,i)$ block the $d\times d$ matrix $[\xi_\g]_{ji}=\r(\eta_j(\g))$ with entries in $\CC$ so that
\[
	(\xi_\f)_{ji}=T_{\r(\eta_j(\g))}:\g\s_i X^d=\s_j\eta_j(\g) X^d\to \s_j X^d
\]
Since each $\r(\eta_j(\g))$ is an invertible matrix, $\xi_\g:\g S_\r X\to S_\r X$ is an isomorphism. 
\end{defn}

To check that this is an object of $\cF^G(\cX,\cX_0)$, take any $\b\in G$. Then $\b\s_j=\s_k\eta_k(\b)$. So $\b\g\s_i=\b\s_j\eta_j(\g)=\s_k\eta_k(\b)\eta_j(\g)=\s_k\eta_k(\b\g)$ and we have:
\[
	{\xi_{\b\g}}=T_{\r(\eta_k(\b\g))_{ki}
	}:\coprod \b\g\s_i X^d\to\coprod \s_k X^d
\]
\[
	={\xi_\b} \circ{\b\xi_\g} =T_{ \r(\eta_k(\b))_{kj}
	}
	T_{ \r(\eta_j(\g))_{ji}
	}:\coprod \b\g\s_i X^d\to\coprod \b\s_j X^d\to\coprod \s_k X^d.
\]

We have the following basic properties of these objects.

\begin{prop}
\begin{enumerate}
\item Two representations $\r,\r':H\to GL(d,\CC)$ are equivalent, i.e. give isomorphism $\CC H$-modules, if and only if $S_\r X\cong S_{\r'}X$ as objects of $\cF^G$.
\item If $\ll,\mu$ are two representations of $H$ then $S_{\ll\oplus \mu}X\cong S_\ll X\oplus S_\mu X$.
\item $SX\cong S_\r X$ where $\r:H\to GL(m,\CC)$ is the regular representation of $H$.
\end{enumerate}
\end{prop}

\begin{lem}
If $\r$ is an irreducible representation of $H=H_X$ over $\CC$ then $S_\r X$ is strongly indecomposable in $\cF^G$, i.e., its endomorphism ring is local.
\end{lem}

\begin{proof}
Let $f\in \End(S_\r X)$ with components $f_{ji}\in\cF(\s_i X^d,\s_j X^d)$. By definition of a morphism in $\cF^G$ we have the following commuting diagram (on the left) for any $\g\in G$. Since $\xi\g$ is monomial, this restricts to the commutative diagram on the right if $\g\in \s_jH\s_i^{-1}$.
\[
%\xymatrixrowsep{10pt}\xymatrixcolsep{10pt}
\xymatrix{%begin xy matrix
\g S_\r X\ar[r]^{\g f}\ar[d]_{\xi_\g}& S_\r X\ar[d]^{\xi_\g} &&\g \s_iX^d\ar[d]_{(\xi_\g)_{ji}}\ar[r]^{\g f_{ii}} &
	\g \s_iX^d\ar[d]^{(\xi_\g)_{ji}}\\
\g S_\r X\ar[r]^f& S_\r X && \s_jX^d \ar[r]^{f_{jj}}& 
	\s_jX^d
	}%end xy matrix
\]
By definition, the matrix of $(\xi_\g)_{ji}$ is $\r(\b)$ where $\g=\eta_j(\g))$. Since $\g f_{ii},f_{ii}$ have the same matrix $[\g f_{ii}]=[f_{ii}]$, we get: $[f_{jj}]\r(\b)=\r(\b)[f_{ii}]$. This holds for all $\b\in H$ and for all $i,j$. (Take $\g=\s_j \b\s_i^{-1}$.)

Take $i=j$. Then we conclude from Schur's lemma that $[f_{ii}]$ is a scalar, say $a_i\in R$ times the identity matrix $I_d$. So, $f_{ii}$ is $a_i$ times the identity map on $\s_i X^d$. Take $i\neq j$. Then we see that $a_i=a_j$. So, $a_i=a_1$ for all $i$. Let $\ov a_1$ be the image of $a_1$ in $\CC$.

\emph{Claim 1}: The map $\f:End(S_\r X)\to \CC$ which sends $f$ to $\ov a_1$ is a ring homomorphism.

Proof: Clearly $\f$ is $\CC$-linear. So, suppose $h=f\circ g$ where $f,g\in\End(S_\r X)$. Then for any $j\neq i$, $f_{ij}g_{ji}$ will have entries in $t^k R$ where $k=c(\s_iX,\s_jX,\s_iX)>0$. Therefore, module $(t)$, $h_{ii}=f_{ii}g_{ii}$. So, $\f(h)=\f(f)\f(g)$.

\emph{Claim 2}: $f\in\End(S_\r X)$ is invertible if and only if $\f(f)\neq0$. 

Proof: Certainly this condition is necessary. To show it is sufficient, it suffices to show that any element of the form $1+g$ where $g$ is in the kernel of $\f$ is invertible. This follows from the fact that, given any $g$ in the kernel of $\f$, we have that $g^\el$ (where $\el=|G:H|$) has all entries divisible by $t$ and thus $1-g+g^2-g^3+\cdots$ converges to the inverse of $1+g$. To see this note that, in block form, any entry of $g^\el$ is a sum of products of $\el$ blocks:
\[
	 g_{b_{\el}b_{\el-1}} \cdots g_{b_2b_1} g_{b_1b_0}
\]
Since there are only $\el$ possible subscripts, two of them must be equal, say $b_i=b_j$. In that case the product $g_{b_i\ast}\cdots g_{\ast b_j}$ is divisible by $t$ by the proof of Claim 1.
\end{proof}

\begin{thm}
Every object of the Frobenius category $\cF^G$ is isomorphic to a direct sum of strongly indecomposable objects of the form $S_\r X$ where $\r$ is an irreducible complex representation of $H_X$. Furthermore, two such objects $S_\ll X, S_\mu Y$ are isomorphic if and only if there is some $\g\in G$ so that
\begin{enumerate} 
\item $Y\cong \g X$. In particular, $H_Y=\g H_X\g^{-1}$. And 
\item $\ll$ is conjugate to the composition of $\mu$ with the isomorphism $H_X\to\g H_X\g^{-1}$ given by conjugation with $\g$.
\end{enumerate}
\end{thm}

\begin{proof}
Given any object $(Y,\xi)$ in $\cF^G$, the maps
\[
	(Y,\xi)\to SY\to (Y,\xi)
\]
which are adjoint to the identity on $Y$ are equal to $\sum \xi_\g^{-1}$ and $\sum \xi_\g$. Their composition is $\sum \xi_\g\xi_\g^{-1}$ which is multiplication by $n=|G|$. Therefore, $(Y,\xi)$ is a sum of components of $SY$ and these all have the form $S_\r X$ for indecomposable $X$ and irreducible $\r$.

If $f: S_\mu Y\cong S_\ll X$ is an isomorphism in $\cF^G$ then, by counting components in $\cF$, we see that $\mu,\ll$ have the same degree, say $d$. As an isomorphism in $\cF$, the composition $f_{i1}:Y^d=\s_1 Y^d\into S_\mu Y\to S_\ll X\onto \s_i X^d$ is an isomorphism for some $i$ proving (1) for $\g=\s_i$. Then $H_Y=\s_i H_X\s_i^{-1}$. For any $\b\in H_Y$, we have the following commuting diagram where $\b'=\s_i^{-1}\b\s_i\in H_X$:
\[
%\xymatrixrowsep{10pt}\xymatrixcolsep{10pt}
\xymatrix{%begin xy matrix
\b Y^d\ar[d]_{T_{\mu(\b)}}\ar[r]^{\b f_{i1}} &
	\b \s_iX^d\ar[d]^{T_{\ll(\b')}}\\
 Y^d \ar[r]^{f_{i1}}& 
	\s_i X^d
	}%end xy matrix
\]
This gives the matrix equation $[f_{i1}]\mu(\b)=\ll(\s_i^{-1}\b\s_i)[f_{i1}]$. In other words conjugation of $\mu$ by the invertible matrix $[f_{i1}]$ is equal to the composition of $\ll$ with the isomorphism $H_Y\to H_X$ given by conjugation by $\s_i^{-1}$. This is equivalent to (2).

It is easy to see that, conversely, (1) and (2) imply that $S_\mu Y\cong S_\ll X$.
\end{proof}

%\newpage

%\bibliography{onlineBib2010}

%%%%%%%%%%%%%%%%%%%%%%%%%%%%%%%%%%%
\end{document}
%%%%%%%%%%%%%%%%%%%%%%%%%%%%%%%%%%%
%